\documentclass[11pt]{amsart}
\usepackage{amssymb, latexsym, graphicx, mathrsfs, enumerate}

\numberwithin{equation}{section}
\newtheorem{Thm}[equation]{Theorem}
\newtheorem{Prop}[equation]{Proposition}
\newtheorem{Cor}[equation]{Corollary}
\newtheorem{Lem}[equation]{Lemma}

\theoremstyle{definition}
\newtheorem{Def}[equation]{Definition}
\newtheorem{Exa}[equation]{Example}
\newtheorem{Rmk}[equation]{Remark}

\begin{document}

\title [Eisenstein Series on Affine Kac-Moody Groups over Function Fields]
{Eisenstein Series on Affine Kac-Moody Groups \\ over Function Fields}
\author[K.-H. Lee]{Kyu-Hwan Lee}
\address{Department of Mathematics, University of Connecticut, Storrs, CT 06269}
\email{khlee@math.uconn.edu}
\author[P. Lombardo]{Philip Lombardo}
\address{Department of Mathematics and Computer Science, St. Joseph's College, Patchogue, NY 11772}
\email{plombardo@sjcny.edu}
\begin{abstract}
In his pioneering work \cite{R}, H. Garland constructed Eisenstein series on affine Kac-Moody groups over the field of real numbers.  He established the almost everywhere convergence of these series, obtained a formula for their constant terms, and proved a functional equation for the constant terms.  In his subsequent paper \cite{AC}, the convergence of the Eisenstein series was obtained. In this paper, we define Eisenstein series on affine Kac-Moody groups over global function fields using an adelic approach. In the course of proving the convergence of these Eisenstein series, we also calculate a formula for the constant terms and prove their convergence and functional equations.
\end{abstract}

\maketitle

\section*{Introduction}
Classical Eisenstein series are central objects in the study of automorphic forms.
  The classical Eisenstein series were generalized to the case of reductive groups and  studied by R. Langlands  \cite{La, La1}. As in the classical case, he found these Eisenstein series have certain analytic properties as well as Fourier series expansions where $L$-functions appear in the constant terms. Because of this relationship, these $L$-functions inherit important analytic properties from the Eisenstein series.  This approach to studying automorphic $L$-functions is known as the Langlands-Shahidi method. (See \cite{GeMi} for a survey.) The Eisenstein series over function fields were studied by G. Harder in \cite{Harder}.

  In \cite{R}, H. Garland defines and studies Eisenstein series on affine Kac-Moody group over $\mathbb R$. He proved the almost everywhere convergence of the series while placing specific emphasis on calculating the constant term, finding its region of convergence, and proving functional equations of the constant term. More precisely, let $\hat G_{\mathbb R}$ be an affine Kac-Moody group over $\mathbb R$. For each character $\chi$ of the positive part $\hat A$ of the torus, he defines a function $\Phi_\chi:\hat A\rightarrow\mathbb C^{\times}$ and, as in the classical case, uses the Iwasawa decomposition $\hat G_{\mathbb R}=\hat{K}\,\hat{A}\,\hat{U}$ to extend $\Phi_\chi$ to a function on $\hat G_{\mathbb R}$.  Garland extends this group by the automorphism $e^{-rD}\in\mathrm{Aut}(V^{\lambda}_{\mathbb R})$, where $r>0$ and $D$ is the degree operator of the Kac-Moody Lie algebra associated with $\hat G_{\mathbb R}$.  Setting $\Phi_\chi(ge^{-rD})=\Phi_\chi(g)$, he defines an Eisenstein series $E_\chi$ on the space $\hat{G}_{\mathbb R}\,e^{-rD} \subseteq \mathrm{Aut}(V^{\lambda}_{\mathbb R})$ by $$\sum_{\gamma\in\hat{\Gamma}/(\hat{\Gamma}\cap \hat B)} \Phi_\chi(ge^{-rD}\gamma),$$ where $\hat{\Gamma}$ is a discrete subgroup of $\hat G_{\mathbb R}$.

  With suitable conditions on the character $\chi$, Garland proves the almost everywhere convergence of $E_\chi$ and calculates the constant term of this series, $E_\chi^\#$, representing it as a sum over the affine Weyl group $\hat W$:
$$ \sum_{w\in\hat W} (a e^{-rD})^{w(\chi+\rho)-\rho}\  \tilde{c}(\chi,w).$$
Here the function $\tilde{c}(\chi,w)$ is a ratio of completed Riemann zeta functions, $\xi(s):= \pi^{-s/2} \Gamma(\frac{s}{2})\zeta(s)$.  After establishing when this infinite sum converges, he proves functional equations for the constant term $E_\chi^\#$. In \cite{AC}, Garland proves the Eisenstein series $E_\chi$ converge absolutely.

  For an arbitrary field $F$, we can construct an affine Kac-Moody group $\hat G_F$ (\cite{LG1}). Let $\mathcal V$ be the set of places of $F$. In this paper, we consider the fields $F_\nu$ for $\nu \in \mathcal V$, completions of a global function field $F$. We work adelically and define an Eisenstein series $E_\chi$ on $\hat G_{\mathbb A}$, a restricted direct product of the groups $\hat G_{F_\nu}$.
Then we calculate the constant term of the Eisenstein series $E_\chi$ and, as in Garland's work, find that we can express the constant term as an infinite sum over the affine Weyl group. Moreover, this expression contains $c(\chi,w)$-functions composed of ratios of $\zeta_F$, the zeta function for the function field $F$. This calculation leads to a proof of convergence of the Eisenstein series.

\begin{Thm}
Let $\chi\in\hat{\frak h}^*$ such that $\mathrm{Re} (\chi(h_{\alpha_i}) )<-2$ for $i=1,\dots,l+1$, and let $m=(m_\nu)_{\nu\in\mathcal V}$ be a tuple  such that $m_\nu \in \mathbb Z_{\ge 0}$ and $0<\sum_\nu m_\nu <\infty$.
Then the Eisenstein series $$E_{\chi}(h\eta^{mD}u):=\sum_{\gamma\in\hat{\Gamma}_F/(\hat{\Gamma}_F\cap\hat{B}_F)}\Phi_{\chi}(h\eta^{mD} u \gamma)$$
is convergent for all $(h,u) \in \hat H_{\mathbb A} \times \hat{U}_{\mathbb A}/(\hat{U}_{\mathbb A}\cap\hat{\Gamma}_F)$. (See \eqref{eqn-eta} for the definition of $\eta^{mD}$.)
\end{Thm}

The zeta function $\zeta_F$, which appears in the $c(\chi,w)$-functions of the constant term, satisfies a functional equation. Using this, we prove that the constant term of the Eisenstein series satisfies a family of functional equations indexed by elements in the affine Weyl group.

This paper has seven sections. In Section 1, we will provide a basic construction of an affine Kac-Moody Lie algebra and its corresponding groups, and then proceed to Section 2 where we prove an Iwasawa decomposition for these groups. Section 3 uses the Iwasawa decomposition to define an Eisenstein series. We also describe the characters we use for this definition and our analogue of the automorphism $e^{-rD}$ that appears in Garland's definition above.
  In Section 4 we calculate the constant term of our series.  The content of Section 5 establishes the region of convergence for this infinite sum, which we use in Section 6 to prove the convergence of the Eisenstein series $E_\chi$.   Finally, in Section 7 we make use of the functional equation for $\zeta_F$ to prove functional equations for the constant term of our Eisenstein series.

\subsection*{Acknowledgments} We are greatly indebted to H. Garland for his inspiring works on this subject. We especially thank him for reading an earlier version of this paper, which was the Ph.D. thesis of the second author. We were informed that M. Patnaik also considered loop Eisenstein series over function fields in his thesis \cite{Pa} without writing up details for convergence of the Eisenstein series. Instead, his thesis focused on geometric aspects of the Eisenstein series.  We thank M. Patnaik for helpful comments. We also thank Keith Conrad, who made many useful comments. Finally, we thank the referee whose suggestions improved this paper very much.

\vskip 1cm

\section{Affine Kac-Moody Lie Algebras and Groups}

  In this section, we describe the affine Kac-Moody groups that we use to define our Eisenstein series.  We will first fix notations for an affine Kac-Moody Lie algebra $\hat{\frak g}^{\, e}$. Then, following Chevalley's construction, we will use automorphisms of a representation space $V^{\lambda}$ of $\hat{\frak g}^{\, e}$ to define the associated affine Kac-Moody group $\hat G^{\lambda}$.

\subsection{Affine Kac-Moody Lie Algebras}

  Let $\frak g$ be a simple, real Lie algebra. Then we define
\begin{equation} \hat{\frak g}^{\,e}=(\mathbb R[t, t^{-1}]\otimes \frak g ) \oplus \mathbb R c \oplus \mathbb R D,
\end{equation} and endow $\hat{\frak g}^{\,e}$ with the standard bracket operation. (See \cite{AC} or \cite{Kac}.) In the expression above, $c$ is a central element and $D$ is the degree operator that acts as $t \frac{d}{dt}$ on $\mathbb R[t, t^{-1}]\otimes \frak g $ and annihilates $c$.  We set
\begin{equation} \hat{\frak h}^{\,e} = \frak h + \mathbb R c + \mathbb R D ,  \end{equation} where $\frak h$ is the Cartan subalgebra of $\frak g$.  We also set
\begin{equation}\hat{\frak g}= (\mathbb R[t, t^{-1}]\otimes \frak g ) \oplus \mathbb R c  \qquad \text{ and } \qquad \hat{\frak h}=\frak h + \mathbb R c . \end{equation}

Let $\Delta$ be the classical root system of $\frak g$, and denote the simple roots by  $\{\alpha_1,\dots,\alpha_l\}$ and the highest root by $\alpha_0$.  Then the affine roots $\hat{\Delta}$ of $\hat{\frak g}^{\,e}$ contain $l+1$ simple roots $\{\alpha_1, \dots\, \alpha_{l+1}\}$. By setting $\delta=\alpha_{0}+\alpha_{l+1}$, we can describe the set of affine roots $\hat{\Delta}$ associated to $\hat{\frak g}$ as
\begin{equation} \hat{\Delta} = \{\alpha + n\delta \ | \ \alpha\in\Delta, \ n\in \mathbb Z \} \cup \{ n\delta \ | \ n\in\mathbb Z_{\neq0} \}.
\end{equation}
The set of affine Weyl (or real) roots is denoted by  $\hat{\Delta}_W=\{\alpha + n\delta \ | \ \alpha\in\Delta, \ n\in \mathbb Z \}$.

The set of the affine roots decompose into a disjoint union of positive roots $\hat{\Delta}_+$ and negative roots $\hat{\Delta}_- = - \hat{\Delta}_+$, where
$$\hat{\Delta}_+ = \{\alpha+n\delta \ | \ \alpha\in\Delta_+, n\in\mathbb Z_{\geq0} \} \cup \{\alpha+n\delta \ | \ \alpha\in\Delta_-, n\in\mathbb Z_{>0}\} \cup \{n\delta \ | \ n\in\mathbb Z_{>0} \}.$$
Similarly we can describe the positive and negative Weyl roots by setting
\begin{equation}
\hat{\Delta}_{W,+} = \{\alpha+n\delta \ | \ \alpha\in\Delta_+, n\in\mathbb Z_{\geq0} \} \cup \{\alpha+n\delta \ | \ \alpha\in\Delta_-, n\in\mathbb Z_{>0}\},
\end{equation}
and $\hat{\Delta}_{W,-}= -\hat{\Delta}_{W,+}$.

  Let $\{h_{\alpha_1}, \dots, h_{\alpha_l+1}\}$ denote the set of simple co-roots associated to the affine simple roots $\alpha_1, \dots, \alpha_{l+1}$. In general, for any $a\in\hat{\Delta}$, we let $h_a$ denote the corresponding co-root.
Recall that we have the Killing form $(  \ , \ )$ on $\frak h$, which we normalize so that $(h_{\alpha_0},h_{\alpha_0})=2$. As in \cite{R}, we extend this bilinear form to a non-degenerate, bilinear form on $\hat{\frak h}^{\,e}$.
For each simple root $\alpha_1,\dots\alpha_{l+1}\in\hat{\Delta}$, we define a simple reflection $w_i\in\mathrm{Aut}(\hat{\frak h}^{\,e})$ for $i=1,\dots,l+1$ by
\begin{equation} w_i(h) =h- \alpha_i(h) h_{\alpha_i}.
\end{equation}
Then the affine Weyl group $\hat{W}$ is defined to be  $$\hat{W} = < w_i \ | \ i=1,\dots,l+1 >  \ \subset  \mathrm{Aut}(\hat{\frak h}^{\,e}).$$ We have
\begin{equation}\hat{W} = W \ltimes T,
\end{equation}
where $W$ is the classical Weyl group, and $T$ is a group of translations that we can index by $H\in\frak h_{\mathbb Z}$ (\cite{IM}, \cite{Kac}).

   As with $\frak g$, the algebra $\hat{\frak g}$ has a Chevalley basis which we can construct using the Chevalley basis for $\frak g$ (\cite{LG1}). First, fix a Chevalley basis for $\frak g$,
$$\Psi= \{h_{\alpha_1},\dots, h_{\alpha_l} \} \cup \{E_\alpha\}_{\alpha\in \Delta}.$$
Now we define some important elements of $\hat{\frak g}$.  For each $a=\alpha+n\delta\in\hat{\Delta}_W$, we let \[ \xi_a =  t^n\otimes E_\alpha \quad \text{ and } \quad \xi_i(n)  = t^n\otimes h_{\alpha_i} \quad \text{ for }i=1,\dots,l . \] Also, we set $h_i= h_{\alpha_i}$ for $i=1,\dots,l$, and $h_{l+1} =  -h_{\alpha_0}+\frac{2}{(\alpha_0,\alpha_0)}\, c$.
Using these elements, we fix a Chevalley basis for the algebra $\hat{\frak g}$: $$\hat{\Psi}=\{h_1,\dots,h_{l+1} \}\cup \{\xi_a \}_{a\in \hat{\Delta}_W}\cup \{\xi_i(n) \}_{i=1,...,\ell; \ n\in \mathbb{Z}_{\geq0}}.$$

  Finally, we denote  the $\mathbb Z$-span of $\hat{\Psi}$ by $\hat{\frak g}_{\mathbb Z}$. Then $\hat{\frak g}_{\mathbb Z}$ is closed under the bracket operation $[\ ,\ ]$ for $\hat{\frak g}^{\,e}$. Using $\hat{\frak g}_{\mathbb Z}$, we can make sense of an affine Kac-Moody Lie algebra over an arbitrary field $F$ by setting
\begin{equation} \hat{\frak g}^{\,e}_F = ( F\otimes_{\mathbb Z} \hat{\frak g}_{\mathbb Z} ) \oplus F D.
\end{equation}

  Let $\mathbf{D}$ denote the set of $\lambda\in (\hat{\frak h}^{\,e} )^*$ such that for $i=1,\dots,l+1$ we have $\lambda(h_{\alpha_i})\in \mathbb Z_{\geq0}$ and $\lambda(h_{\alpha_i})\neq0$ for some $i$. This is the set of \textit{dominant}, \textit{integral}, \textit{normal} weights of $\hat{\frak g}^{\,e}$.
  In \cite{LG1} we see that for each $\lambda \in \mathbf{D}$, we have an irreducible $\hat{\frak g}^{\,e}$-module,  $V^{\lambda }$, with a highest weight vector $v_{\lambda}$.  This $V^{\lambda }$ contains a $\mathbb Z$-module $V_{\mathbb Z}^{\lambda }$ satisfying
\begin{equation}
\label{eqn-z}
\frac{(\xi_a)^m}{m!}\cdot V_{\mathbb Z}^{\lambda } \subseteq V_{\mathbb Z}^{\lambda },\end{equation}
for any $a\in \hat{\Delta}_W$ and $m \in \mathbb Z_{\geq0}$. We fix this $\mathbb Z$-module $V_{\mathbb Z}^{\lambda }$ and call it the {\it Chevalley form} of $V^{\lambda }$. The representation space $V^{\lambda}$ and $V^{\lambda}_{\mathbb Z}$ decompose into a direct sum of weight spaces, $V^{\lambda}_{\mu}$ and $V^{\lambda}_{\mu,\mathbb Z}=V^{\lambda}_{\mu}\cap V^{\lambda}_{\mathbb Z} $, respectively. As a highest weight module, we know that any weight $\mu$ of $V^{\lambda}$ is of the form
\begin{equation}
\mu=\lambda-\sum_{i=1}^{l+1} k_i \alpha_i,
\end{equation}
where $k_i\in\mathbb Z_{\geq0}$. For an arbitrary field $F$, we set $V^{\lambda}_F= F\otimes_{\mathbb Z} V^{\lambda}_{\mathbb Z}$. Then $V^{\lambda}_F$ is a highest weight $\hat{\frak g}^{\,e}_F$-module.

  In the next subsection, we will use the elements of $\hat{\frak g}^{\,e}_F$ to describe some special automorphisms of the vector space $V^{\lambda}_F$.  These automorphisms generate the affine Kac-Moody group over the arbitrary field $F$.

\subsection{Construction of Affine Kac-Moody Groups}

  Let $F$ be an arbitrary field. For $a\in \hat{\Delta}_W$ and $s\in F$, we define the automorphism $\chi_{a}(s)$ of $V^{\lambda}_F$ as
\begin{equation} \label{eqn-xi}
\chi_a(s)=  \displaystyle\sum_{n\geq 0}s^n \ \frac{\xi_a^n}{n!} \, .\end{equation}
By (\ref{eqn-z}) we know that this definition works for fields of arbitrary characteristic.  For each $v\in V^{\lambda}_F$ and $a\in\hat{\Delta}_W$, there exists an $n_0$ such that for all $n\geq n_0$ we have $$\frac{\xi_a^n}{n!}\cdot v=0.$$
Hence for each $v\in V^{\lambda}_F$ and $a\in\hat{\Delta}_W$, the sum in (\ref{eqn-xi}) acts as a finite sum (\cite{LG1}).

We let $F((X))$ be the field of Laurent series in the variable $X$ with coefficients from $F$. Then $\sigma\in F((X))$ has an expression as $\sigma= \displaystyle\sum_{i\geq i_0} s_i X^i$ where $i_0\in\mathbb Z$ and $s_i\in F$.  For $\alpha\in \Delta$ (the classical roots), we let \begin{equation} \label{eqn-xi-sigma} \chi_{\alpha}(\sigma)= \displaystyle\prod_{i \geq i_0} \chi_{\alpha + i\delta} (s_i).\end{equation} For each $v\in V^{\lambda}_F$ there exists an $i_k$ such that for all $i\geq i_k$, $\chi_{\alpha+i\delta}(s) \cdot v= v$ for any $s\in F$  and so for each $v$ the product in (\ref{eqn-xi-sigma}) acts as a finite product (\cite{LG1}).

  As a result of these observations, each $\chi_{\alpha}(\sigma)$ is an automorphism of the representation space $V^{\lambda}_F$. Finally, we make the following definition.

\begin{Def}
Let $F$ be an arbitrary field and $\lambda\in\mathbf{D}$. The {\em affine Kac-Moody group} associated to $\hat{\frak g}^{\,e}_F$ and its representation space $V^{\lambda}_F$ is the following subgroup of $\mathrm{Aut}(V^{\lambda}_F)$:
\begin{equation} \hat{G}^{\lambda}_F = \left<\ \chi_{\alpha}(\sigma) \ \big| \ \ \alpha \in \Delta, \ \sigma\in F((X)) \right> \,.\end{equation}
\end{Def}

\begin{Rmk} \hfill
\begin{enumerate} \item Since we are considering the automorphisms of $V^{\lambda}_F$, our group depends on the choice of $\lambda$.  We fix a $\lambda\in\mathbf{D}$ and drop the $\lambda$ from our notation.
\item One may note that in the construction of $\hat G_F$ we only used elements of $\hat{\frak g}$ and ignored the degree operator $D$.  In section 3, we will extend our group $\hat G_F$ by a particular automorphism $\eta^{mD}$ related to $D$, thereby establishing a more complete relationship between $\hat{\frak g}^{\,e}_F=\hat{\frak g}_F \oplus F D$ and our group.  Garland extends his group in a similar way by the automorphism $e^{-rD}$ for $r>0$ (\cite{R}, \cite{AC}).
\end{enumerate}
\end{Rmk}

   In the next section, we will begin working with this group when $F$ is a field with a non-Archimedean absolute value.  Our first objective is to develop an Iwasawa decomposition for $\hat G_F$ in this case, from which we will be able to begin defining our Eisenstein series.

\medskip

\section{Iwasawa Decomposition for Affine Kac-Moody Groups}{}

In this section, we prove an Iwasawa decomposition for $\hat G_F$, where $F$ is a local field with a non-Archimedean absolute value.  In particular, we will apply this result to the groups $\hat G_{F_\nu}=\hat G_{\nu}$, where $F_\nu$ is a completion of a global function field $F$.

\medskip

  We now consider the particular case of $\hat G_{F_\nu}:=\hat G_{\nu}$, where $F_\nu$ is an arbitrary field with a non-Archimedean discrete valuation $\nu$.  For $x\in F_\nu$ let $|x|_{\nu}$ denote the absolute value that corresponds to the valuation $\nu$. We define $\mathcal O_{\nu}= \left\{x\in F_\nu \ \big| \ |x|_{\nu}\leq 1\right\}$ and $\mathcal P_{\nu} =\left\{x\in F_\nu \ \big| \ |x|_{\nu}<1\right\}$, noting that $\mathcal P_{\nu}$ is the unique maximal ideal of the ring $\mathcal O_{\nu}$.

For any $a\in\hat{\Delta}_W$ and $s\in F_\nu^{\times}$, we set
\[
w_a(s)=  \chi_a(s)\chi_{-a}(-s^{-1}) \chi_a(s) \quad \text{ and } \quad
h_a(s) = w_a(s)w_a(1)^{-1}.
\]
Likewise, for $\alpha\in\Delta$ and nonzero $\sigma \in F_\nu((X))$,  we set
\[
w_{\alpha}(\sigma) = \chi_{\alpha}(\sigma)\chi_{-\alpha}(-\sigma^{-1}) \chi_{\alpha}(\sigma)  \quad \text{ and } \quad
h_{\alpha}(\sigma) =  w_{\alpha}(\sigma)w_{\alpha}(1)^{-1}.
\]

  Using these elements, we can define the subgroups of $\hat G_{\nu}$ that appear in the Iwasawa decomposition of the affine Kac-Moody group. We let $F_\nu[[X]]\subset F_\nu((X))$  denote the ring of power series over $F_\nu$.  We then set
\begin{eqnarray*}
\hat{U}_{\nu} & \ = \ &\bigg< \chi_{\alpha}(\sigma) \ \bigg| \ \alpha\in\Delta_+, \ \sigma \in F_\nu[[X]]\ \ \text{   or   } \ \ \alpha\in\Delta_-, \ \sigma \in XF_\nu[[X]] \bigg> ,\\
\hat{H}_{\nu} & \ = \ &\bigg< h_{\alpha_i}(s) \ \bigg| \ i=1,\dots, l+1, \ s\in F_\nu^{\times} \bigg>
\end{eqnarray*}
and define $\hat{B}_\nu$ to be the group generated by $\hat{U}_\nu$ and $\hat{H}_\nu$, which can be realized as a semi-direct product $\hat{H}_{\nu} \ltimes\hat{U}_\nu$. Finally, set $$\hat{K}_\nu=\bigg<\chi_{\alpha}(\sigma) \ \bigg| \ \alpha\in\Delta, \ \sigma \in\mathcal O_{\nu}((X)) \bigg>,$$ where $\mathcal O_{\nu}$ is defined above.

\begin{Lem}[\cite{LG1}, \S14 ] Let
$\hat{N}_\nu = \bigg<w_a(s) \ \bigg| \ a\in\hat{\Delta}_W, \ s\in F_\nu^{\times} \bigg> $, $\hat{W}= \hat{N}_\nu/(\hat{N}_\nu\cap\hat{B}_\nu)=\hat{N}_\nu/\hat{H}_{\nu}$,  and $S= \{w_{\alpha_i}(1) \ | \ i=1,\dots,l+1\} $.
Then $(\hat G_{\nu},  \hat{B}_\nu, \hat{N}_\nu, S)$ is a $BN$-pair for $\hat G_{\nu}$.
\end{Lem}

\begin{Rmk} There is a natural isomorphism between $\hat{W}=\hat{N}_\nu/(\hat{N}_\nu\cap\hat{B}_\nu)$ and the affine Weyl group $\hat{W}=\big< w_i \ | \ i=1,\dots,l+1 \big>$, which identifies the element $w_{\alpha_i}(1)$ for $\alpha_i$  with the simple reflection $w_i$. See \cite{LG1}.
\end{Rmk}

  In light of the theory of $BN$-pairs (\cite{Kumar} \S 5, \cite{Hum} \S 29), we have the following facts regarding $\hat G_{\nu}$:

\begin{Cor} \label{cor-disjoint}
For the affine Kac-Moody group $\hat G_{\nu}$ we have:
\begin{enumerate}
\item $\hat G_{\nu}=  \displaystyle\bigcup_{w\in \hat{W}} \ \hat{B}_\nu w \hat{B}_\nu$       (disjoint union).
\item For $w\in \hat{W}$, let $w=w_{i_1}\dots w_{i_k}$ be a reduced expression. Then we have $\hat{B}_\nu w\hat{B}_\nu=Y_{i_1}\dots Y_{i_k} \hat{B}_\nu$, where $Y_i$ is a set of representatives for the cosets $(\hat{B}_\nu w_i\hat{B}_\nu)/\hat{B}_\nu$, for each $i=1,\dots,l+1$.
\end{enumerate}
\end{Cor}

  We continue by choosing a specific set of coset representatives for each $Y_i$. To this end, we recall the following lemma from \cite{LG1}, \S 16:

\begin{Lem} \label{lem-expression}
 Every element $x\in\hat{B}_\nu w\hat{B}_\nu$ has an expression
$$x=\left(\prod_{a\in\hat{\Delta}_w}\chi_{a}(s_{a}) \right) w y$$
where $s_{a}\in F_\nu$, $y\in\hat{B}_\nu$, and $\hat{\Delta}_w=\hat{\Delta}_{W,+}\cap w\hat{\Delta}_{W,-}$.
\end{Lem}

\noindent In particular, if we take  $w=w_i$ for some $i$, each element $x\in\hat{B}_\nu w_i \hat{B}_\nu$ has an expression as $x=\chi_{\alpha_i}(s) w_i y$ with $y\in\hat{B}_\nu$.

\medskip

  Now since we choose our $Y_i$'s to be representatives of the coset space $\hat{B}_\nu w_i \hat{B}/\hat{B}_\nu$, we can choose our $Y_i$'s to consist of elements of the form $\chi_{\alpha_i}(s)w_i$ where $s\in F_\nu$.

\begin{Lem}[\cite{LG1}, \S16] \label{lem-hom}
For an arbitrary field $F$, and any $a\in\hat{\Delta}_W$, we have a homomorphism $\varphi_a:SL_2(F)\rightarrow\hat G_F$ that is defined by the conditions:
 \[\varphi_a( (\begin{smallmatrix}1&s\\0&1\end{smallmatrix})) \ = \ \chi_a(s), \quad
 \varphi_a((\begin{smallmatrix}1&0\\s&1\end{smallmatrix})) \ = \ \chi_{-a}(s), \quad
 \varphi_a ((\begin{smallmatrix}0&1\\-1&0\end{smallmatrix} ) ) \ = w_a(1), \quad
 \varphi_a ( (\begin{smallmatrix}r&0\\0&r^{-1}\end{smallmatrix} ) ) \ = \ h_a(r)  \]
  for $s\in F_\nu$ and $r\in F_\nu^{\times}$.
\end{Lem}

  For the field $F_\nu$, we know that the group $SL_2(F_\nu)$ contains the subgroup $B$ consisting of upper triangular matrices, and the subgroup $K$ = $SL_2(\mathcal O_{\nu})$. Moreoever,  we have the Iwasawa decomposition $SL_2(F_\nu)=K\,B$. For more information see \cite{IM}, \S2.

\begin{Lem} \label{lem-parts}
We have the following:
\begin{enumerate}
\item $\varphi_a(B)\subset\hat{B}_\nu$, if $a\in\hat{\Delta}_{W,+}$.
\item $\varphi_a(K)\subset\hat{K}_\nu$, for any $a\in\hat{\Delta}_W$.
\item We can choose the elements of $Y_i$ for $i=1,\dots,l+1$, so that they are the images through $\varphi_{\alpha_i}$ of elements in $K$.  In particular, we may assume that $Y_i\subset \hat{K}_\nu$ for $i=1,\dots,l+1$.
\end{enumerate}
\end{Lem}

\begin{proof}

(1)  Consider an arbitrary element $(\begin{smallmatrix}r&s\\0&r^{-1}\end{smallmatrix})\in B$. We can express it in the following way:
$$(\begin{smallmatrix}r&s\\0&r^{-1}\end{smallmatrix})=(\begin{smallmatrix}r&0\\0&r^{-1}\end{smallmatrix})(\begin{smallmatrix}1&r^{-1}s\\0&1\end{smallmatrix}).$$
Thus $\varphi_a((\begin{smallmatrix}r&s\\0&r^{-1}\end{smallmatrix}))=h_a(r)\chi_a(r^{-1}s).$  Since $\hat{B}_\nu$ is generated by $\hat{H}_{\nu} $ and $\hat{U}_\nu$, it suffices to show that $\chi_a(r^{-1}s)\in\hat{U}_\nu$. Since we assumed that $a\in\Delta_{W,+}$, we know $a=\alpha +n\delta$ where either $n=0$ and $\alpha\in\Delta_+$, or $n> 0$ and $\alpha\in\Delta$.  We can write $\chi_a(r^{-1}s)=\chi_{\alpha}(\sigma)$ by setting $\sigma=(r^{-1}s)X^n$ (see the definition of $\chi_{\alpha}(\sigma), \ \S1)$. The conditions on $\alpha$ and $n$ guarantee that $\chi_{\alpha}(\sigma)\in\hat{U}_\nu$.  Therefore, when $a\in\hat{\Delta}_{W,+}$ we have the desired result.

(2)  We know that $$K=\bigg<(\begin{smallmatrix}1&s\\0&1\end{smallmatrix}), (\begin{smallmatrix}1&0\\s'&1\end{smallmatrix}) \ \bigg| \ s,s'\in \mathcal O_{\nu}\bigg>$$ (see \cite{IM}, \S2 ), and so an arbitrary element of $K$ will be a finite product of these matrices. As a result, the image of an element in $K$ through the homomorphism $\varphi_a$ will be a finite product of $\chi_a(s)$ and $\chi_{-a}(s')$ with $s, s'\in \mathcal O_{\nu}$. It suffices to show that $\chi_a(s)$ and $\chi_{-a}(s')$ are elements of $\hat{K}_\nu$. If $a=\alpha + n\delta$, then $-a=-\alpha-n\delta$ and we can express $\chi_a(s)=\chi_\alpha(sX^n)$ and $\chi_{-a}(s')=\chi_{-\alpha}(s'X^{-n})$.  Since the coefficients $s$ and $s'$ are each elements of $\mathcal O_{\nu}$, we have that $\chi_a(s)$ and $\chi_{-a}(s')$ are elements of $\hat{K}_\nu$.

 (3)  Recall that $Y_i$ is a set of coset representatives for $(\hat{B}_\nu w_i\hat{B}_\nu)/\hat{B}_\nu$, for $w_i\in S$. By Lemma  \ref{lem-expression} we can choose these representatives to be of the form $\chi_{\alpha_i}(s) w_i$ for $s\in F_\nu$, but then $$\chi_{\alpha_i}(s) w_i=\varphi_{\alpha_i}((\begin{smallmatrix}1&s\\0&1\end{smallmatrix}))\varphi_{\alpha_i}((\begin{smallmatrix}0&1\\-1&0\end{smallmatrix}))=\varphi_{\alpha_i}((\begin{smallmatrix}-s&1\\-1&0\end{smallmatrix})).$$  However, by the Iwasawa decomposition for $SL_2(F)$, we know that $(\begin{smallmatrix}-s&1\\-1&0\end{smallmatrix})=kb$ for $k\in K$ and $b\in B$. Thus, we have $\chi_{\alpha_i}(s) w_i=\varphi_{\alpha_i}(k)\varphi_{\alpha_i}(b)$ and by part (1) of this lemma we know that the coset $\chi_{\alpha_i}(s) w_i\hat{B}=\varphi_{\alpha_i}(k)\varphi_{\alpha_i}(b)\hat{B}=\varphi_{\alpha_i}(k)\hat{B}$. So we can take our representatives for $(\hat{B}_\nu w_i \hat{B}_\nu)/\hat{B}_\nu$ to be of the form $\varphi_{\alpha_i}(k)$ for some $k\in K$.   Finally, part (2) of this lemma implies that we can choose our $Y_i$ to be a subset of $\hat{K}_\nu$.

\end{proof}

Now we can prove that for any $\nu\in\mathcal V$, the affine Kac-Moody group $\hat G_{\nu}$ has an Iwasawa decomposition.

\begin{Thm}[Iwasawa decomposition, \cite{LG1}] Let $F_\nu$ be a field with a non-Archimedean discrete valuation $\nu$, and let $\hat G_{\nu}$ be an affine Kac-Moody group over $F_\nu$. Then
$$\hat G_{\nu} \,=\, \hat{K}_\nu \, \hat{H}_{\nu}  \, \hat{U}_\nu,$$
where $\hat{K}_\nu$, $\hat{H}_{\nu} $, and $\hat{U}_\nu$ are defined as above.
\end{Thm}

\begin{proof}
We have already established a Bruhat decomposition for $\hat G_{\nu}$; in other words, $$\hat G_{\nu}=  \displaystyle\bigcup_{w\in \hat{W}} \ \hat{B}_\nu w \hat{B}_\nu.$$
Since this is a disjoint union, it suffices to show that each $\hat{B}_\nu w \hat{B}_\nu$ decomposes in the desired way. It follows from Corollary \ref{cor-disjoint} and part (iii) of Lemma \ref{lem-parts} that for each $w\in\hat{W}$ we have $\hat{B}_\nu w \hat{B}_\nu\subset\hat{K}_\nu\hat{B}_\nu$, and thus $\hat G_{\nu}=\hat{K}_\nu\hat{B}_\nu$. We have already observed that $\hat{B}_\nu=\hat{H}_{\nu} \ltimes\hat{U}_\nu$, and so we obtain the Iwasawa decomposition:
$$\hat G_{\nu}\,=\, \hat{K}_\nu \, \hat{H}_{\nu}  \, \hat{U}_\nu.$$
\end{proof}

The Iwasawa decomposition of an element is not uniquely determined. See Remark \ref{rmk-I} and Corollary \ref{cor-I}.

\medskip

\section{Defining the Eisenstein Series}{}

  For the remainder of this paper, we set $F$ to be a global function field of genus $g$. Let $\mathcal V$ denote the set of all places of $F$, and for $\nu\in\mathcal V$ let $F_\nu$ denote the completion of $F$ with respect to $\nu$. As in the previous section, we let $|\cdot|_{\nu}$ denote the corresponding non-Archimedean absolute value on $F_\nu$, and we define the local ring $\mathcal O_{\nu}$ and its maximal ideal $\mathcal P_{\nu}$ as before. We fix a uniformizer $\pi_\nu\in\mathcal O_{\nu}$, so $\pi_\nu$ generates the ideal $\mathcal P_{\nu}$. Finally let the integer $q_\nu$ denote the cardinality of the residue field $\mathcal O_{\nu}/\mathcal P_{\nu}$.

  In this section, we will use an ad\`{e}lic approach and define our Eisenstein series on the group $\hat G_{\mathbb A}$, a restricted direct product of affine Kac-Moody groups over the completions of the field $F$.

%Section 1: Adelic Approach
\subsection{The Ad\`{e}lic Approach}

  Using the completions $F_\nu$, we define the ad\`{e}le ring $\mathbb A$ as the restricted direct product: \[ \mathbb A ={\prod_{\nu \in \mathcal{V}}}'\ F_\nu, \text{ with respect to the subrings }\mathcal O_{\nu}. \]
The units of this ring, the group of id\`{e}les $\mathbb A^{\times}$, can also be realized as a restricted direct product: \[ \mathbb A^\times={\prod_{\nu \in \mathcal{V}}}'\ F_\nu^\times, \text{ with respect to }\mathcal O_{\nu}^\times, \]
where $\mathcal O_{\nu}^\times= \{x\in\mathcal O_{\nu} \ | \ |x|_{\nu} = 1 \} \subset \mathcal O_{\nu}$. For $s=(s_{\nu})\in\mathbb A^{\times}$, we define the id\`{e}lic norm $$|s|=\prod_{\nu\in\mathcal V} |s_{\nu}|_{\nu}.$$
We set $\hat G_{\nu}=\hat G_{F_\nu}$ and define the group $\hat G_{\mathbb A}$ as the following restricted direct product: \[ \hat{G}_{\mathbb{A}}={\prod_{\nu \in \mathcal{V}}}'\ \hat G_{\nu}, \text{ with respect to the subgroups }\hat{K}_\nu. \]

   In order to define our Eisenstein series on $\hat G_{\mathbb A}$, we must first establish an Iwasawa decomposition for this group.  With the appropriately defined subgroups of $\hat G_{\mathbb A}$, this will be a direct result of Theorem 2.10. To this end, we distinguish the following subgroups of $\hat G_{\mathbb A}$:
\[ \hat{\mathbb K}=\prod_{\nu \in \mathcal V}\hat{K}_\nu , \quad
 \hat{H}_{\mathbb{A}} = {\prod_{\nu \in \mathcal{V}}}'\ \hat{H}_{\nu} \quad \text{ and } \quad \hat{U}_{\mathbb A} = {\prod_{\nu \in \mathcal{V}}}' \, \hat{U}_\nu, \] where the restricted direct products are with respect to $\hat{H}_\nu\cap\hat{K}_\nu$ and
$\hat{U}_{\nu} \cap\hat{K}_\nu$, respectively.

\begin{Thm} We have the following Iwasawa decomposition for the group $\hat G_{\mathbb A}$: $$\hat G_{\mathbb A} \, = \, \hat{\mathbb K} \, \hat H_{\mathbb A} \, \hat{U}_{\mathbb A}.$$
\end{Thm}

\begin{Rmk} \label{rmk-I}
In \cite{LG1}, Garland develops an Iwasawa decomposition for the group $\hat G_{\mathbb R}$ and establishes that the decomposition of an element is unique. In our setting, the Iwasawa decomposition of an element is not unique.  This is potentially problematic because we will use this decomposition to define the Eisenstein series.  However, we will see in Proposition \ref{prop-well} that due to the structure of $\hat H_{\mathbb A}$ this is not an issue.
\end{Rmk}

\medskip

%Section 2: Structure of \hat H_{\mathbb A}

\subsection{The Structure and Topology of the Torus}

  We fixed a normal weight $\lambda\in\mathbf{D}$, and hence a $\hat{\frak g}^{\,e}$-module $V^{\lambda}_{F_\nu}$.  It is a highest weight module, so we let $v_{\lambda}$ denote the highest weight vector. In \cite{LG1} we see that the representation space $V^{\lambda}_{F_\nu}$ decomposes into a direct sum of its weight spaces
$$V_{F_\nu}^{\lambda}= \displaystyle\bigoplus_{\mu\in(\hat{\frak h}^{\,e})^*} V^{\lambda}_{\mu, F_\nu},\ \text{  where }V^{\lambda}_{\mu, F_\nu}= \{v\in V_{F_\nu}^{\lambda} \ | \ h\cdot v= \mu(h) v, \ h\in \hat{\frak h}^{\,e} \}.$$
Moreover, every weight of $V^{\lambda}_{F_\nu}$ is of the form $\mu=\lambda-\sum_{i=1}^{l+1} k_i \alpha_i$, for $k_i\in\mathbb Z_{\geq 0}$. Using this unique expression, we define the {\em depth} of $\mu$ as $$\text{dp}(\mu)=\sum_{i=1}^{l+1} k_i \,.$$

We fix a basis $\mathcal B$ of $V^{\lambda}_{\mathbb Z}$ by choosing basis vectors $\{v_{\lambda }, v_1, \dots, v_n, \dots \}$  in $V^{\lambda}_{\mathbb Z}$ and ordering them so that
\begin{enumerate}
\item if $v_i\in V^{\lambda}_{\mu, \mathbb Z}$, $v_j\in V^{\lambda}_{\mu', \mathbb Z}$, and $i<j$, then we have $\text{dp}(\mu)\leq \text{dp}(\mu')$, and
\item for each weight $\mu$ of $V^{\lambda}_{\mathbb Z}$, the basis vectors of $V^{\lambda}_{\mu, \mathbb Z}$ appear consecutively.
\end{enumerate}
A basis of $V^{\lambda}_{\mathbb Z}$ that satisfies these conditions is called \textit{coherently ordered}. It is important to note that since we chose our basis vectors from $V^{\lambda}_{\mathbb Z}$, the basis $\mathcal B$ serves as a basis for $V^{\lambda}_{F}$ as well as $V^{\lambda}_{F_\nu}$ for every $\nu\in\mathcal V$. The advantage to fixing such a basis is that with respect to $\mathcal B$ we can view the elements of $\hat{H}_{\nu} $ as (infinite) diagonal matrices which are scalar matrices when we restrict to a weight space. In addition, the elements of $\hat{U}_\nu$ are strictly upper triangular block (infinite) matrices where the blocks are determined by the weight spaces of $V^{\lambda}_{\mathbb Z}$. For more information, see \cite{LG1}.

  By this observation, we can clearly see that elements of $\hat{H}_{\nu} $ commute with each other, and $\hat{H}_{\nu} $ normalizes the subgroup $\hat{U}_\nu$.  Since this holds for all $\nu\in\mathcal V$, we obtain the same results for $\hat H_{\mathbb A}$ and $\hat{U}_{\mathbb A}$.

\medskip

Because the definition of our Eisenstein series depends on it, we are interested in studying the structure of $\hat H_{\mathbb A}$. Let $h\in\hat H_{\mathbb A}$.  Considering the local components, we let
$$h_{\nu}=\prod_{i=1}^{l+1} h_{\alpha_i}(s_{i,\nu}).$$
 As a result, we may write $h\in\hat H_{\mathbb A}$ as the following product:
$$h=\prod_{i=1}^{l+1} h_{\alpha_i}\big((s_{i,\nu})_{\nu\in\mathcal V}\big).$$

 We will prove the following:
\begin{Prop} \label{prop-zero}  Assume that $\lambda\in\mathbf{D}$. Suppose that
$$h^{\lambda}=\prod_{i=1}^{l+1} h^{\lambda}_{\alpha_i}(s_{i,\nu})\in\hat{H}_{\nu} ^{\lambda}\cap\hat{K}_\nu^{\lambda}. $$ Then we have $\mathrm{ord}_{\nu}(s_{i,\nu})=0$ for $i=1,\dots,l+1$.
\end{Prop}

\begin{Rmk} We have included the superscript $\lambda$ above since we must consider different $\lambda$'s in the proof below.
\end{Rmk}

\begin{proof}
As a consequence of Lemma 15.7 and Theorem 15.9 in \cite{LG1}, we know that for each fundamental weight $\Lambda_i$ there exist a positive integer $m_i$ and a surjective group homomorphism
$$\pi(\lambda, m_i\Lambda_i):\hat  G^{\lambda}_\nu\rightarrow\hat  G^{m_i \Lambda_i}_\nu, $$
and this homomorphism is characterized by the fact that it maps $\chi^{\lambda}_{\alpha}(\sigma)$ to $\chi^{m_i\Lambda_i}_{\alpha}(\sigma)$. As a result,
$$\prod_{i=1}^{l+1} h^{\lambda}_{\alpha_i}(s_{i,\nu}) \longmapsto \prod_{i=1}^{l+1} h^{m_i\Lambda_i}_{\alpha_i}(s_{i,\nu}),$$ as well.
Set $h^{\lambda}=\prod_{i=1}^{l+1} h^{\lambda}_{\alpha_i}(s_{i,\nu})$ and $h^{m_i \Lambda_i}=\prod_{i=1}^{l+1} h^{m_i \Lambda_i}_{\alpha_i}(s_{i,\nu})$. Since $h^{\lambda} \in\hat{H}_{\nu}^{\lambda}\cap\hat{K}_\nu^{\lambda}$ by assumption, we have  $h^{m_i \Lambda_i} \in\hat{H}_{\nu} ^{m_i \Lambda_i}\cap\hat{K}_\nu^{m_i \Lambda_i}$.
Choose a highest weight vector $1\otimes v\in V_{F_\nu}^{m_i \Lambda_i}= F_{\nu} \otimes V^{m_i \Lambda _i}_{\mathbb Z}$. Then by \cite{LG1} we know that \begin{equation}h^{m_i \Lambda_i} \cdot (1 \otimes v)= \prod_{i=1}^{l+1} h ^{m_i \Lambda_i} _{\alpha_i}(s_{i,\nu})\cdot (1\otimes v)=  \left(\ \prod_{i=1}^{l+1}  s_{i,\nu}^{{m_i \Lambda_i} (h_{\alpha_i})} \right)\otimes v  = s_{i,\nu}^{m_i} \otimes v.\end{equation} Since elements of $\hat{K}_\nu$ preserve the subspace $V^{\lambda}_{\mathcal O_{\nu}}$, we have $\mathrm{ord}_{\nu}(s_{i,\nu}) \geq 0$.

  Moreover, since $\hat{H}_{\nu}^\lambda \cap\hat{K}_\nu^\lambda$ is a group, we know $(h_{\nu}^\lambda)^{-1}=\prod_{i=1}^{l+1} h^\lambda_{\alpha_i}(s_{i,\nu}^{-1})$ is also in the intersection.  Applying the argument above to $(h^\lambda_{\nu})^{-1}$, we find $\mathrm{ord}_{\nu}(s_{i,\nu}) \leq 0$. Thus we have $\mathrm{ord}_{\nu}(s_{i,\nu}) = 0$ for each $i =1, \dots, l+1$.

\end{proof}

\begin{Cor}
The subgroup $\hat H_{\mathbb A}\leq\hat G_{\mathbb A}$ may be realized in the following way:
$$\hat H_{\mathbb A}=\left\{\ \prod_{i=1}^{l+1} h_{\alpha_i} (s_i) \ \Bigg| \ s_i\in\mathbb A^{\times} \ \right\}.$$
\end{Cor}
\begin{proof}
Proposition \ref{prop-zero} shows that for almost all $\nu\in\mathcal V$, we have $s_{i,\nu}\in\mathcal O_{\nu}^\times$.  As a result, these infinite tuples $s_i$ are actually elements of $\mathbb A^{\times}$.
\end{proof}

\begin{Rmk} \label{rmk-iso} We have a surjective group homomorphism $\vartheta:  (\mathbb A^{\times} )^{l+1}  \longrightarrow \hat H_{\mathbb A}$ defined by
$$ (s_1, s_2,  \dots, s_{l+1} )  \mapsto h=\prod_{i=1}^{l+1} h_{\alpha_i} (s_i).$$
The group $(\mathbb A^{\times})^{l+1}$ inherits the product topology induced by the usual topological structure of $\mathbb A^{\times}$, and we give the space $\hat H_{\mathbb A}$ the quotient topology induced by the map $\vartheta$. We will use the map $\vartheta$ again in Section 5.
\end{Rmk}

% Section 3: Characters of the Torus
\subsection{Characters}

  Defining our Eisenstein series on $\hat G_{\mathbb A}$ requires that we specify a character of the subgroup $\hat H_{\mathbb A}$.  Let $|\cdot|$ be the id\`{e}lic norm. We define a character by fixing a linear functional $\chi\in\hat{\frak h}^*$ and setting
$$h^{\, \chi} = \prod_{i=1}^{l+1} \, |s_i|^{\ \chi(h_{\alpha_i})},$$
when $h=\prod_{i=1}^{l+1} h_{\alpha_i}(s_i)\in\hat H_{\mathbb A}$.

\begin{Rmk} \label{rmk-K} If $h\in\hat H_{\mathbb A}\cap\hat{\mathbb K}$, then $h=\prod_{i=1}^{l+1} h_{\alpha_i}(s_i)$ with $s_i\in \prod_{\nu\in\mathcal V} \mathcal O_{\nu}^\times$.  By our definition of $\chi$, we see that $h^\chi=1$ for any $h\in\hat H_{\mathbb A}\cap\hat{\mathbb K}$, since $|s_i|=1$ for each $i=1,\dots,l+1$.
\end{Rmk}

  Fix $\chi\in\hat{\frak h}^*$, and define $\Phi_{\chi}:\hat G_{\mathbb A}\longrightarrow \mathbb C^{\times}$ to be the function induced by the character $\chi$ on $\hat H_{\mathbb A}$ and the Iwasawa decomposition for $\hat G_{\mathbb A}$.  In other words, if $g=k \, h \, u$ is an element of $\hat G_{\mathbb A}$, then we set \[ \Phi_{\chi}(g)=\Phi_{\chi}(k \, h \, u)=h^{\, \chi}.\] We noted earlier that the Iwasawa decomposition for $\hat G_{\mathbb A}$ is not unique, so we need to prove that this function is well-defined.

\begin{Lem} \label{lem-semi}
For any place $\nu\in\mathcal V$, the subgroup $\hat{B}_\nu\cap\hat{K}_\nu$ is the semidirect product $(\hat{H}_{\nu} \cap\hat{K}_\nu)\ltimes(\hat{U}_\nu\cap\hat{K}_\nu)$.
\end{Lem}
\begin{proof}
With respect to the coherently ordered basis $\mathcal B$, the elements of $\hat{K}_\nu$ are matrices with elements from the ring $\mathcal O_{\nu}$ and the elements of $\hat{B}_\nu$ are upper triangular block matrices.  Thus, we can view $b\in\hat{B}_\nu\cap\hat{K}_\nu$ as an upper triangular infinite block matrix with entries from $\mathcal O_{\nu}$. By the definition of $\hat{B}_\nu$, we know that $b=h\,u$ for $h\in\hat{H}_{\nu} $ and $u\in\hat{U}_\nu$. In fact, with respect to $\mathcal B$, the matrix $h$ will be diagonal with the same diagonal entries that appear in the matrix $b$. In particular, $h\in\hat{H}_{\nu} \cap\hat{K}_\nu$ which also implies that $u\in\hat{U}_\nu\cap\hat{K}_\nu$.
\end{proof}

\begin{Prop} \label{prop-well}
Let $g=k\, h\, u = k'\, h'\, u'$ be two Iwasawa decompositions for $g\in\hat G_{\mathbb A}$. Then $$\Phi_{\chi}(k'\,h'\, u')=\Phi_{\chi}(k\, h\, u).$$ In particular, $\Phi_{\chi}$ is a well-defined function from $\hat G_{\mathbb A}$ into $\mathbb C^{\times}$.
\end{Prop}
\begin{proof}
We begin by noting that if $k\, h\, u = k'\, h'\, u'$, then $$k'^{-1}k=h' u' (hu)^{-1}\in\hat{B}_{\mathbb A} \cap\hat{\mathbb K}.$$ By Lemma \ref{lem-semi}, $k'^{-1}k=\bar{h}\bar{u}$ with $\bar{h}\in\hat H_{\mathbb A}\cap\hat{\mathbb K}$ and $\bar{u}\in\hat{U}_{\mathbb A}\cap\hat{\mathbb K}$. Using this information, we see
$$k'\, h'\, u' =k\, h\, u \ =k'\,\bar{h}\bar{u} h u.$$ Since $\hat H_{\mathbb A}$ normalizes $\hat{U}_{\mathbb A}$, we can express $k'\, h'\, u'=k'\,\bar{h} h u_1$, for some $u_1\in\hat{U}_{\mathbb A}$.  Finally, we observe that by Remark \ref{rmk-K} we have
\begin{eqnarray*}
\Phi_{\chi}(k'\, h'\, u')\ &=& \ \Phi_{\chi}(k'\,\bar{h} h u_1)= \ (\bar{h}h)^{\chi}\\
&=&\  \bar{h}^{\chi} \, h^{\chi}=\ h^{\chi} = \ \Phi_{\chi}(k\,h\,u).
\end{eqnarray*}
\end{proof}

\begin{Cor} \label{cor-I}
The $\hat H_{\mathbb A}$-component of an Iwasawa decomposition is uniquely determined up to $\hat H_{\mathbb A}\cap\hat{\mathbb K}$.
\end{Cor}

\begin{proof}
From the proof of the above proposition, we obtain $h'\, u'=\bar{h} h u_1$. Since $\hat{B}_{\mathbb A}=\hat{H}_{\mathbb A} \ltimes\hat{U}_{\mathbb A}$, we actually have $h'=\bar h h$.
\end{proof}

%Section 4: AN IMPORTANT AUTOMORPHISM \ETA^{mD}
\subsection{An Important Automorphism}

  As in \cite{R}, we need to extend the group $\hat G_{\mathbb A}$ by an automorphism related to the degree operator $D$ that appears in the associated affine Kac-Moody Lie algebra $$\hat{\frak g}^{\,e}=(\mathbb C[t,t^{-1}]\otimes \frak g )\oplus \mathbb C c \oplus \mathbb C D.$$  In \cite{R}, Garland uses $e^{-rD}\in\mathrm{Aut}(V^{\lambda}_{\mathbb R})$ for $r>0$ to extend the group $\hat G_{\mathbb R}$; however, in our case we are considering the restricted direct product $\hat G_{\mathbb A}$.  For this reason, we first define local automorphisms $\eta_{\nu}^{m_{\nu}D}\in\mathrm{Aut}(V^{\lambda}_{F_\nu})$ for each $\nu\in\mathcal V$, and then work with a product of the local automorphisms.

  For each $\nu\in\mathcal V$ and integer $m_{\nu} \in \mathbb Z$, we define $\eta_{\nu}^{m_{\nu}D}$ to be the automorphism of $V^{\lambda}_{F_\nu}$ defined by the conditions
\begin{enumerate}
\item the automorphism $\eta_{\nu}^{m_{\nu}D}$ fixes each weight space $V^{\lambda}_{\mu, F_\nu}$, and

\item we have $\eta_{\nu}^{m_{\nu}D} \cdot v= \pi_{\nu}^{m_{\nu}\,\mu(D)} v$ for $v \in V^{\lambda}_{\mu, F_\nu}$.
\end{enumerate}

Since $\eta_{\nu}^{m_{\nu}D}$ acts as scalar multiplication on the weight spaces, we can consider this automorphism as being a diagonal block matrix with respect to the coherently ordered basis $\mathcal B$, and as such the automorphism will commute with $\hat{H}_{\nu} $ and normalize $\hat{U}_\nu$. Moreover, note that if we chose $m_{\nu}=0$, then $\eta_{\nu}^{m_{\nu}D}$ is the identity map.

  We fix a  tuple $m= (m_{\nu})_{\nu\in\mathcal V}$ such that $m_{\nu}\in\mathbb Z$ and $m_{\nu}=0$ for all but a finite number of $\nu$. By doing so, we fix the associated automorphism $\eta^{mD}$ defined as the product \begin{equation} \label{eqn-eta} \eta^{mD}=\prod_{\nu\in\mathcal V} \, \eta_{\nu}^{m_{\nu}D} \ \in \ \prod_{\nu\in\mathcal V} \mathrm{Aut} (V^{\lambda}_{F_\nu} ).\end{equation}

  We will define the Eisenstein series on $\hat G_{\mathbb A}\eta^{mD}$ for our fixed automorphism $\eta^{mD}$. In particular, we will consider $\Phi_{\chi}$ as a function on  $\hat G_{\mathbb A}\eta^{mD}$ by setting $\Phi_{\chi}(g \eta^{mD})=\Phi_{\chi}(g)$.

%Section 5:
\subsection{Defining the Eisenstein Series}

  For each completion $F_\nu$, there is the natural injection $i_{\nu}: F\hookrightarrow F_\nu$ which induces the injection $F((X))\hookrightarrow F_\nu((X))$ by sending
$$\sum_{i\geq i_0} s_i X^i \mapsto \sum_{i\geq i_0} \,i_{\nu}(s_i) X^i .$$
From this map, we see that there is a natural injection $i_{\nu}:\hat G_F\hookrightarrow\hat G_{\nu}$ for each $\nu\in\mathcal V$, and we may define the diagonal embedding $i:\hat G_F\hookrightarrow\prod_{\nu\in\mathcal V}\hat G_{\nu}$ by
$$\chi_{\alpha}(\sigma)\mapsto \big(i_{\nu} (\chi_{\alpha}(\sigma) )\big)_{\nu\in\mathcal V} \ .$$
The image of the map $i$ is not entirely contained in the group $\hat G_{\mathbb A}$. To see this clearly, we construct an example of an element from $\hat G_F$ that does not diagonally embed into $\hat G_{\mathbb A}$.

\begin{Exa} Let $F=\mathbb{F}_q(T)$. It is well known that all but one of the places (the ``infinite'' place corresponding to $\frac{1}{T})$ are indexed by monic, irreducible polynomials in $\mathbb{F}_q[T]$. We let $f_{\nu}(T)$ denote the polynomial associated with the place $\nu$. Set
$$\sigma= T+ \frac{1}{f_{\nu_1}}X + \frac{1}{f_{\nu_2}}X^2 + \frac{1}{f_{\nu_3}}X^3 + \dots,$$
where we set the coefficient of $X^i$ to be $\frac{1}{f_{\nu_i}}$ for some $\nu_i$ that has not previously appeared in the expansion.  Clearly, we have $\frac{1}{f_{\nu}}\in\mathbb{F}_q(T)$ for all $\nu$, so $\sigma \in F((X))$. However, by design $\sigma \notin\mathcal O_{\nu}((X))$ for an infinite number of $\nu\in\mathcal V$, and as a result $i_{\nu}(\chi_{\alpha}(\sigma))$ is not an element of $\hat{K}_\nu$ for an infinite number of $\nu$. Therefore, $i(\chi_{\alpha}(\sigma))\notin\hat G_{\mathbb A}$.
\end{Exa}

Keeping this example in mind, we consider the subgroup $\hat{\Gamma}_F$ defined by \[ \hat{\Gamma}_F =  \{g\in\hat G_F \ \big| \ i(g)\in\hat G_{\mathbb A} \}.\] By an abuse of notation, $\hat{\Gamma}_F$ will be considered as a subgroup of $\hat G_F$ as well as $\hat G_{\mathbb A}$, where in the latter case we consider the elements as being diagonally embedded.  Note that  that $h^{\chi}=1$ for any $h\in\hat H_{\mathbb A}\cap\hat{\Gamma}_F$ and $\chi\in\hat{\frak h}^*$. We also have the subgroups $\hat{H}_F$, $\hat{U}_F$ and $\hat{B}_F$ of the group $\hat G_F$.

\medskip

  In the definition of the Eisenstein series, $\hat{\Gamma}_F/(\hat{\Gamma}_F\cap\hat{B}_F)$ will be the coset space over which we index our sum.  Before continuing, we first establish the following facts about $\Phi_{\chi}$.

\begin{Lem} \label{lem-123}
Let $g,\beta \in\hat G_{\mathbb A}$, and $\gamma \in\hat \Gamma_F\cap \hat{B}_F$; then
\begin{enumerate}
\item $\Phi_{\chi}(gh')= (h')^{\chi}\,\Phi_{\chi}(g)$ for any $h'\in\hat H_{\mathbb A}$,
\item $\Phi_{\chi}(g\gamma)=\Phi_{\chi}(g)$, and
\item $\Phi_{\chi}(g\eta^{mD}\beta  \gamma)=\Phi_{\chi}(g\eta^{mD}\beta)$.
\end{enumerate}
\end{Lem}
\begin{proof}
(1) We write $g=khu$ according to the Iwasawa decomposition.  Since $\hat H_{\mathbb A}$ normalizes $\hat{U}_{\mathbb A}$, we have
\[
\Phi_{\chi}(gh') = \Phi_{\chi}(khuh')= \Phi_{\chi}(khh'u').
\]
Now by definition of $\Phi_{\chi}$, we know
\[
\Phi_{\chi}(khh'u')=  (h h')^{\chi}=(h)^{\chi}(h')^{\chi}.
\]
Since $h$ is the $\hat H_{\mathbb A}$-component of $g$, we know this last expression equals $(h')^{\chi} \, \Phi_{\chi}(g).$

(2)   Using an argument similar to that of Lemma \ref{lem-semi}, we can show that for any $\gamma\in\hat{\Gamma}_F\cap \hat{B}_{\mathbb A}$, we have a decomposition $\gamma=h_1u_1$ with $h_1\in\hat H_{\mathbb A}\cap\hat{\Gamma}_F$ and $u_1\in\hat{U}_{\mathbb A}$.  Using our Iwasawa decomposition, express $g=khu$. Then
\begin{eqnarray*}
\Phi_{\chi}(g\gamma) \ &=&\ \Phi_{\chi}(khuh_1u_1)\ \  =\ \Phi_{\chi}(kh h_1 u' u_1)\\
&=&\  (h h_1)^{\chi} = \  (h)^{\chi} = \  \Phi_{\chi}(g),
\end{eqnarray*}
where the second to last equality holds because $h_1\in\hat H_{\mathbb A}\cap\hat{\Gamma}_F$ .

(3)  As before we let $\gamma=h_1u_1$ with $h_1\in\hat{\Gamma}_F\cap\hat H_{\mathbb A}$ and $u_1\in\hat{U}_{\mathbb A}$.  Since $\eta^{mD}$ commutes with $\hat H_{\mathbb A}$ and normalizes $\hat{U}_{\mathbb A}$, we see $\eta^{mD}\gamma(\eta^{mD})^{-1}=h_1u_2$ for some $u_2\in\hat{U}_{\mathbb A}$. Let $\eta^{mD}\beta=\beta'\eta^{mD}$ for some $\beta'\in\hat G_{\mathbb A}$. Then
$$\Phi_{\chi}(g\eta^{mD}\beta\gamma) = \Phi_{\chi}(g\beta'\eta^{mD}\gamma) = \Phi_{\chi}(g\beta'h_1u_2\eta^{mD})\ \Phi_{\chi}(g\beta'h_1u_2).$$
We know $\Phi_{\chi}$ is right invariant by $\hat{U}_{\mathbb A}$, and in light of part (1) of this lemma, the following equalities hold:
\begin{eqnarray*}
\Phi_{\chi}(g\beta'h_1u_2) &=&  \Phi_{\chi}(g\beta' h_1) \ \ \, = \   (h_1)^{\chi} \, \Phi_{\chi}(g\beta') \ = \ \Phi_{\chi}(g\beta')\\
&=& \Phi_{\chi}(g\beta'\eta^{mD})  = \ \Phi_{\chi}(g\eta^{mD}\beta).
\end{eqnarray*}
\end{proof}

  Due to part (3) of the previous lemma, the following definition of the Eisenstein series $E_\chi$ on the space $\hat G_{\mathbb A}\eta^{mD}$ is well-defined.

\begin{Def} \label{def-Eis}
For $g \in \hat G_{\mathbb A}$ and $\chi\in\hat{\frak h}^*$, we define \begin{equation}E_{\chi}(g\eta^{mD})=\sum_{\gamma\in\hat{\Gamma}_F/(\hat{\Gamma}_F\cap\hat{B}_F)}\Phi_{\chi}(g\eta^{mD} \gamma),\end{equation} whenever the series converges to a complex number; otherwise we define $E_{\chi}(g\eta^{mD})= \infty$.
\end{Def}

  The goal of this paper is to prove the convergence of the series $E_\chi$. Later, we will see that after some reductions we can consider $E_\chi(g\eta^ mD)$ as a function on the space $\hat H_{\mathbb A}\times\hat{U}_{\mathbb A}/(\hat{U}_{\mathbb A}\cap\hat{\Gamma}_F)$. In the next three sections, we will prove the following theorem:

\begin{Thm} \label{thm-main}
Let $\chi\in\hat{\frak h}^*$ such that $\mathrm{Re} (\chi(h_{\alpha_i}) )<-2$ for $i=1,\dots,l+1$, and let $m=(m_\nu)_{\nu\in\mathcal V}$ be a tuple  such that $m_\nu \in \mathbb Z_{\ge 0}$ and $0<\sum_\nu m_\nu <\infty$. Then the Eisenstein series $$E_{\chi}(h\eta^{mD}u) =\sum_{\gamma\in\hat{\Gamma}_F/(\hat{\Gamma}_F\cap\hat{B}_F)}\Phi_{\chi}(h\eta^{mD} u \gamma)$$
is convergent for all $(h,u)\in\hat H_{\mathbb A}\times\hat{U}_{\mathbb A}/(\hat{U}_{\mathbb A}\cap\hat{\Gamma}_F)$.
\end{Thm}

  To prove this theorem, we first assume that $\chi$ is a real character, so  $\chi:\hat H_{\mathbb A}\rightarrow\mathbb R_{>0}$. As a result, the Eisenstein series $E_\chi$ takes values in $\mathbb R_{>0}\cup \{\infty\}$.
This assumption is not very restrictive because for any complex character $\chi$, the series $E_\chi$ is dominated by $E_{\mathrm{Re}(\chi)}$. Hence, we can apply the dominated convergence theorem for the complex case after we consider the real character $\chi$.

As in Corollary \ref{cor-disjoint}, we see that $\hat G_F$ has the Bruhat decomposition into the following disjoint union
$$\hat G_F= \bigcup_{w\in\hat{W}} \hat{B}_F \, w \, \hat{B}_F.$$
If we let $\hat{\Gamma}_F(w) = \ \hat{\Gamma}_F\cap(\hat{B}_F \, w \, \hat{B}_F)$ and define
\begin{equation} E_{\chi, w}(g\eta^{mD}) = \ \sum_{\gamma\in\hat{\Gamma}_F(w)/(\hat{\Gamma}_F(w)\cap\hat{B}_F)}\Phi_{\chi}(g\eta^{mD} \gamma),\end{equation}
 then the Bruhat decomposition above allows us to express our Eisenstein series as
\begin{equation}E_{\chi}(g\eta^{mD})=\sum_{w\in\hat{W}} E_{\chi, w}(g\eta^{mD}).\end{equation} This is simply a regrouping of the sum in Definition \ref{def-Eis}.   As with $E_{\chi}$, we can consider each $E_{\chi, w}$ as a function on the space $\hat H_{\mathbb A}\times\hat{U}_{\mathbb A}/(\hat{U}_{\mathbb A}\cap\hat{\Gamma}_F)$.  In order to prove Theorem \ref{thm-main}, it suffices to show that
\begin{equation} \label{eqn-conv} \sum_{w\in\hat{W}} \ \ \int_{\hat{U}_{\mathbb A}/(\hat{U}_{\mathbb A}\cap\hat{\Gamma}_F)} E_{\chi, w}(h\eta^{mD} u)  \ du \ < \infty \end{equation}
for $h$ varying in an arbitrary compact set of $\hat H_{\mathbb A}$ and for real $\chi$.

If we establish the convergence (\ref{eqn-conv}), then
\begin{eqnarray} \sum_{w\in\hat{W}} \ \ \int_{\hat{U}_{\mathbb A}/(\hat{U}_{\mathbb A}\cap\hat{\Gamma}_F)} E_{\chi, w}(h\eta^{mD} u)  \ du  &=& \int_{\hat{U}_{\mathbb A}/(\hat{U}_{\mathbb A}\cap\hat{\Gamma}_F)}  \sum_{w\in\hat{W}} E_{\chi, w}(h\eta^{mD} u) \ du \nonumber \\ &=& \int_{\hat{U}_{\mathbb A}/(\hat{U}_{\mathbb A}\cap\hat{\Gamma}_F)}  E_{\chi}(h\eta^{mD} u) \ du . \label{eqn-const} \end{eqnarray} Note that the last expression is nothing but the constant term of the Eisenstein series $E_\chi$.

\begin{Def} \label{def-constant-term}
We set \[ E^{\#}_\chi(g \eta^{mD}) =  \sum_{w\in\hat{W}} \ \int_{\hat{U}_{\mathbb A}/(\hat{U}_{\mathbb A}\cap\hat{\Gamma}_F)} E_{\chi, w}(g\eta^{mD} u) \ du,  \] and call $E^{\#}_\chi$ the {\em constant term} of the Eisenstein series $E_\chi$.
\end{Def}

In the next section, we will calculate the integrals $$ \int_{\hat{U}_{\mathbb A}/(\hat{U}_{\mathbb A}\cap\hat{\Gamma}_F)} E_{\chi, w}(h\eta^{mD} u)  \ du$$ for $w\in\hat{W}$.  In Section 5, we establish the convergence (\ref{eqn-conv}) when $\chi$ is a real character and $h$ varies in a compact set of $\hat{H}_\mathbb A$. As a result, we will obtain the almost everywhere convergence of the Eisenstein series and a concrete description of its constant term.

\vskip 1 cm

\section{Calculating the Constant Term of the Eisenstein Series} \label{section-Cal}

  In this section, we simply state the existence and properties of the measures necessary for our calculation, leaving the details to Appendix A. Constructing these measures involves taking the projective limit of a family of measures. For now we will also refrain from showing that $E_{\chi}$ is a measurable function, a topic that we will address in Section 6.

\subsection{Definition and Preliminary Calculation}

  From Appendix A, we have an invariant probability measure $du$ on the space $\hat{U}_{\mathbb A}/(\hat{U}_{\mathbb A}\cap\hat{\Gamma}_F)$. As was discussed at the end of the previous section, we now  turn our attention to calculating the expression \begin{equation}\sum_{w\in\hat{W}} \ \int_{\hat{U}_{\mathbb A}/(\hat{U}_{\mathbb A}\cap\hat{\Gamma}_F)} E_{\chi, w}(g\eta^{mD} u) \ du.\end{equation}
We first calculate the integrals $$\int_{\hat{U}_{\mathbb A}/(\hat{U}_{\mathbb A}\cap\hat{\Gamma}_F)} E_{\chi, w}(g\eta^{mD} u) \ du,$$ for $w\in\hat{W}$.

  \medskip

Let $\hat{U}_{-,F}$ be the subgroup of $\hat G_F$ consisting of the elements that are strictly lower triangular block matrices with respect to our coherently ordered basis $\mathcal B$ of $V^{\lambda}_F$.
We define $$\hat{U}_{w,F}=\hat{U}_F\cap w \hat{U}_{-,F}w^{-1}.$$ Note that this definition works over $F_\nu$ as well, so the notations $\hat{U}_{-,\nu}$ and $\hat{U}_{w,\nu}$ are clear. Finally, we set $\hat{U}_{-,\mathbb A}$ and $\hat{U}_{w,\mathbb A}$ to be the expected restricted direct products.  The Bruhat decomposition has this refinement:
$$\hat G_F=  \bigcup_{w\in\hat{W}} \hat{U}_{w,F} \, w \, \hat{B}_F \ \ \ \text{(disjoint union)}.$$
Moreover, every element of $u\in\hat{U}_{w,F}$ is of the form
\begin{equation}
u=\prod_{a\in\hat{\Delta}_{W,+}\cap w\,\hat{\Delta}_{W,-}}\chi_a(s_a) \quad  \text{ for } s_a\in F \, .\end{equation} (See \cite{R}  \S6,  \cite{LG2}  \S6.)
It is straightforward to check that $\hat{U}_{w,F}\subset\hat{\Gamma}_F$, which implies
$$\hat{\Gamma}_F\cap (\hat{U}_{w,F}\, w\,  \hat{B}_F )=\hat{U}_{w,F}\, w\, (\hat{\Gamma}_F\cap\hat{B}_F ). $$
As a result, we can choose the coset representatives of $\hat{\Gamma}_F(w)/\hat{\Gamma}_F(w)\cap\hat{B}_F$ to be $\{bw\}$ for $b\in\hat{U}_{w,F}$. Thus,
\begin{eqnarray*}
\int_{\hat{U}_{\mathbb A}/(\hat{U}_{\mathbb A}\cap\hat{\Gamma}_F)} E_{\chi, w}(g\eta^{mD} u) \ du&=&\int_{\hat{U}_{\mathbb A}/(\hat{U}_{\mathbb A}\cap\hat{\Gamma}_F)} \sum_{\gamma\in\hat{\Gamma}_F(w)/\hat{\Gamma}_F(w)\cap\hat{B}_F} \Phi_{\chi}(g\eta^{mD} u \gamma) \ du\\
&=& \int_{\hat{U}_{\mathbb A}/(\hat{U}_{\mathbb A}\cap\hat{\Gamma}_F)} \sum_{b\in\hat{U}_{w,F}} \Phi_{\chi}(g\eta^{mD} u b w) \ du.
\end{eqnarray*}

We have the following decompositions:
\begin{equation} \label{eqn-decom}
\hat{U}_F =\ \hat{U}_{w,F}\, (\hat{U}_F\cap w\hat{U}_F w^{-1}) \quad \text{ and } \quad
\hat{U}_{\mathbb A} = \hat{U}_{w,\mathbb A} \, (\hat{U}_{\mathbb A}\cap w\hat{U}_{\mathbb A} w^{-1}).
\end{equation}
  This decomposition, along with the fact that $\hat{U}_{w,F}\subset\hat{\Gamma}_F$, implies that
\begin{equation} \hat{U}_F\cap\hat{\Gamma}_F= \hat{U}_{w,F} \, (\hat{\Gamma}_F\cap\hat{U}_F\cap \ w\,\hat{U}_F\,w^{-1}).
\end{equation}
So we can consider the set of $b\in\hat{U}_{w,F}$ as a set of coset representatives for $$(\hat{\Gamma}_F\cap\hat{U}_F)/(\hat{\Gamma}_F\cap\hat{U}_F\cap w\hat{U}_F w^{-1}).$$ Since $\hat{\Gamma}_F\cap\hat{U}_F=\hat{\Gamma}_F\cap\hat{U}_{\mathbb A}$, our integral $$ \int_{\hat{U}_{\mathbb A}/(\hat{U}_{\mathbb A}\cap\hat{\Gamma}_F)} \sum_{b\in(\hat{\Gamma}_F\cap\hat{U}_F)/(\hat{\Gamma}_F\cap\hat{U}_F\cap w\hat{U}_F w^{-1})} \Phi_{\chi}(g\eta^{mD} u b w) \ du$$ becomes
\begin{equation} \label{eqn-integral-1}  \int_{\hat{U}_{\mathbb A}/(\hat{\Gamma}_F\cap\hat{U}_F\cap w\hat{U}_F w^{-1})} \Phi_{\chi}(g\eta^{mD} u'  w) \ du'.
\end{equation}
   Here we consider the measure $du'$ as the measure induced from $du$ and the projection $$\pi':\hat{U}_{\mathbb A}/(\hat{\Gamma}_F\cap\hat{U}_F\cap w\hat{U}_F w^{-1}) \twoheadrightarrow \hat{U}_{\mathbb A}/(\hat{U}_{\mathbb A}\cap\hat{\Gamma}_F).$$

     Using the decomposition (\ref{eqn-decom}) for $\hat{U}_{\mathbb A}$, we observe that integrating over this coset is the same as first integrating over $$\hat{U}_{\mathbb A}/(\hat{U}_{\mathbb A}\cap w\hat{U}_{\mathbb A} w^{-1})$$ and then over $$(\hat{U}_{\mathbb A}\cap w\hat{U}_{\mathbb A} w^{-1})\big/(\hat{\Gamma}_F\cap\hat{U}_F\cap w\hat{U}_F w^{-1}).$$
    In Appendix A we see that the measure $du'$ decomposes into measures $du_1$ and $du_2$ on these spaces, respectively. Using these measures and decompositions, we can manipulate our integral (\ref{eqn-integral-1}) to be
\begin{equation} \label{eqn-integral-2} \int_{\hat{U}_{w,\mathbb A}} \left  ( \ \ \int_{(\hat{U}_{\mathbb A}\cap w\hat{U}_{\mathbb A} w^{-1})/(\hat{\Gamma}_F\cap\hat{U}_F\cap w\hat{U}_F w^{-1})}  \Phi_{\chi}(g\eta^{mD} u_1 u_2  w) \ du_2\  \right ) \ du_1, \end{equation} where we set  $\hat{U}_{w,\mathbb A} = \hat{U}_{\mathbb A}/(\hat{U}_{\mathbb A}\cap w\hat{U}_{\mathbb A} w^{-1})$.

  Since $\Phi_{\chi}$ is $\hat{U}_{\mathbb A}$-right invariant, we let $u_2'=w^{-1}u_2 w\in\hat{U}_{\mathbb A}$ and  rewrite $$\Phi_{\chi}(g\eta^{mD} u_1 u_2  w)=\Phi_{\chi}(g\eta^{mD} u_1 w u_2')=\Phi_{\chi}(g\eta^{mD} u_1 w).$$  As a result, the integral (\ref{eqn-integral-2}) becomes
\[
 \int_{\hat{U}_{w,\mathbb A}}  \left ( \ \  \int_{(\hat{U}_{\mathbb A}\cap w\hat{U}_{\mathbb A} w^{-1})/(\hat{\Gamma}_F\cap\hat{U}_F\cap w\hat{U}_F w^{-1})}  \Phi_{\chi}(g\eta^{mD} u_1 w) \ du_2\ \right ) \ du_1.
\]
Since the values $ \Phi_{\chi}(g\eta^{mD} u_1 w)$ no longer depend on $u_2$ and the measure $du_2$ has a total measure of 1, this equals
\begin{eqnarray*}
& &\ \ \ \int_{\hat{U}_{w,\mathbb A}} \Phi_{\chi}(g\eta^{mD} u_1 w)\, du_1 \ \ \int_{(\hat{U}_{\mathbb A}\cap w\hat{U}_{\mathbb A} w^{-1})/(\hat{\Gamma}_F\cap\hat{U}_F\cap w\hat{U}_F w^{-1})} du_2  \\
&=&\ \ \int_{\hat{U}_{w,\mathbb A}} \Phi_{\chi}(g\eta^{mD} u_1 w) du_1.
\end{eqnarray*}
 The following proposition summarizes our results from this subsection:
%Main proposition STEP 1
\begin{Prop} \label{prop-step-1} For  $g\in \hat G_{\mathbb A}$ and $w \in \hat W$,  we have $$\int_{\hat{U}_{\mathbb A}/(\hat{U}_{\mathbb A}\cap\hat{\Gamma}_F)} E_{\chi, w}(g\eta^{mD} u) \ du= \int_{\hat{U}_{w,\mathbb A}} \Phi_{\chi}(g\eta^{mD} u_1 w) \ du_1.$$
\end{Prop}

\medskip

\subsection{Further Calculation}

  We continue our computation by further manipulating the integral in Proposition \ref{prop-step-1}.  Fix an Iwasawa decomposition $g=k h u$. Since $\eta^{mD}$ normalizes $\hat{U}_{\mathbb A}$, we have
\[
\Phi_{\chi}(g\eta^{mD} u_1 w) = \Phi_{\chi}(khu\eta^{mD} u_1 w)=  \ \Phi_{\chi}(h\eta^{mD} u' u_1 w).
\]
The decomposition (\ref{eqn-decom}) allows us to write $u'=u'_- \, u'_+$ for $u'_-\in\hat{U}_{w,\mathbb A}$ and $u'_+\in\hat{U}_{\mathbb A}\cap w\hat{U}_{\mathbb A} w^{-1}$.  Clearly, $w^{-1}u'_+w\in\hat{U}_{\mathbb A}$, so by the right invariance of $\Phi_{\chi}$ we have \begin{eqnarray*}
\Phi_{\chi}(h\eta^{mD} u' u_1 w) \ &=& \ \Phi_{\chi}(h \eta^{mD} u'_- u'_+ u_1 w)\\
&=& \  \Phi_{\chi}(h \eta^{mD} u'_- u'_+ u_1 (u'_+)^{-1}w w^{-1}u'_+ w)\\
&=& \  \Phi_{\chi}(h\eta^{mD} u'_- u'_+ u_1 (u'_+)^{-1} w).
\end{eqnarray*}
The decomposition (\ref{eqn-decom}) also induces the natural projection $$\tilde{\pi}: \hat{U}_{\mathbb A}\twoheadrightarrow\hat{U}_{w,\mathbb A}, $$ and we obtain that $$\Phi_{\chi}(h\eta^{mD} u'_- u'_+ u_1 (u'_+)^{-1} w)=\Phi_{\chi} (h\eta^{mD} u'_-\ \tilde{\pi} ( u'_+ u_1 (u'_+)^{-1} )\, w ).$$ For any $u_1\in\hat{U}_{w,\mathbb A}$ and a fixed $u_+\in\hat{U}_{\mathbb A} \cap w\hat{U}_{\mathbb A} w^{-1}$, the map that sends $u_1$ to $\tilde{\pi}(u_+ u_1 (u_+)^{-1})$ is a unimodular change of variables, so the integral in Proposition \ref{prop-step-1} becomes
\begin{eqnarray*}
\int_{\hat{U}_{w,\mathbb A}} \Phi_{\chi}(g\eta^{mD} u_1 w) \ du_1
&=& \int_{\hat{U}_{w,\mathbb A}} \Phi_{\chi}(h\eta^{mD} u'_- \tilde{\pi}(u'_+ u_1 (u'_+)^{-1}) w) \ du_1\\
&=& \int_{\hat{U}_{w,\mathbb A}} \Phi_{\chi}(h\eta^{mD} u'_-  u_1 w) \ du_1.
\end{eqnarray*}
However, note that $u'_- \in\hat{U}_{w,\mathbb A}$ remains fixed as $u_1$ ranges over $\hat{U}_{w,\mathbb A}$, and since $du_1$ is $\hat{U}_{w,\mathbb A}$-translation invariant, our integral may now be expressed as
\begin{equation} \label{eqn-integral-3} \int_{\hat{U}_{w,\mathbb A}} \Phi_{\chi}(h\eta^{mD} u_1 w) \ du_1.\end{equation}
  Continuing, we note that $$ \Phi_{\chi}(h\eta^{mD} u_1 w)= \Phi_{\chi} (h\eta^{mD} u_1 (\eta^{mD})^{-1}h^{-1} w w^{-1} h \eta^{mD} w ),$$ and since $w^{-1} h \eta^{mD} w\in\hat H_{\mathbb A}$, we obtain from Lemma \ref{lem-123} part (1)
\[
 \Phi_{\chi}(h\eta^{mD} u_1 w) = (w^{-1} h \eta^{mD} w)^{\chi} \, \Phi_{\chi}(h \eta^{mD} u_1 (\eta^{mD})^{-1} h^{-1} w) .
\]
Since $(w^{-1} h \eta^{mD} w)^{\chi}=(h \eta^{mD})^{w\chi},$ the integral (\ref{eqn-integral-3}) becomes
 \begin{equation}  \label{eqn-integral-4} (h\eta^{mD})^{w\chi}\int_{\hat{U}_{w,\mathbb A}} \Phi_{\chi}(h\eta^{mD} u_1 (\eta^{mD})^{-1}h^{-1} w) du_1.\end{equation}

\medskip

Set  $\hat{\Delta}_w=\hat{\Delta}_{W,+}\cap w\, \hat{\Delta}_{W,-}.$ Then applying the change of variables $$ h\eta^{mD} u_1 (h\eta^{mD})^{-1} \mapsto u_1$$ has the following effect on our integral in (\ref{eqn-integral-4}):
$$(h\eta^{mD})^{w\chi}\int_{\hat{U}_{w,\mathbb A}} \Phi_{\chi}(h\eta^{mD} u_1 (h\eta^{mD})^{-1} w) du_1=(h\eta^{mD})^{w\chi} (h\eta^{mD})^{-\Sigma}\int_{\hat{U}_{w,\mathbb A}} \Phi_{\chi}(u_1w) du_1,$$
where $\Sigma= \displaystyle\sum_{\alpha\in\hat{\Delta}_w} \alpha$. It is known (\cite [p.50]{GaLep}) that $$ \sum_{\alpha\in\hat{\Delta}_{w}} \alpha=\rho-w\rho,$$ where $\rho\in\hat{\frak h}^*$ such that $\rho(h_{\alpha_i})=1$ for $i=1,\dots,l+1$. Therefore,
\begin{eqnarray}
(h\eta^{mD})^{w\chi} (h\eta^{mD})^{-\Sigma}\int_{\hat{U}_{w,\mathbb A}} \Phi_{\chi}(u_1 w) du_1 &=& (h\eta^{mD})^{w\chi} (h\eta^{mD})^{w\rho-\rho}\int_{\hat{U}_{w,\mathbb A}} \Phi_{\chi}(u_1 w) du_1 \nonumber \\
&=& (h\eta^{mD})^{w(\chi+\rho)-\rho}\int_{\hat{U}_{w,\mathbb A}} \Phi_{\chi}(u_1 w) du_1. \label{eqn-integral-5}
\end{eqnarray}

  Let $\hat{U}_{-,w,\mathbb A}= w^{-1}\hat{U}_{w,\mathbb A} w$. Then $w^{-1}u_1w\in\hat{U}_{-,w,\mathbb A}$.  So considering $\Phi_{\chi}(u_1w)$ for $u_1\in\hat{U}_{w,\mathbb A}$ is exactly the same as considering $\Phi_{\chi}(wu_-)$ for $u_-\in\hat{U}_{-,w,\mathbb A}$.  Finally, since we may assume $w\in\hat{\mathbb K}$, the $\hat{\mathbb K}$-left invariance of $\Phi_{\chi}$ allows us to rewrite our integral (\ref{eqn-integral-5}) as
$$(h\eta^{mD})^{w(\chi+\rho)-\rho}\int_{\hat{U}_{-,w,\mathbb A}} \Phi_{\chi}(u_-) du_-,$$
where $du_-$ is the Haar measure induced by conjugating by $w^{-1}$.

Using Definition \ref{def-constant-term}, our results from this section appear in the following proposition:

\begin{Prop} \label{prop-step-2}

For any $g\in\hat G_{\mathbb A}$, we have $$E_{\chi}^{\#}(g\eta^{mD})=\sum_{w\in\hat{W}} \ (h\eta^{mD})^{w(\chi+\rho)-\rho}\int_{\hat{U}_{-,w,\mathbb A}} \Phi_{\chi}(u_-) du_-.$$
\end{Prop}

\medskip

%%%% CALCULATING THE LOCAL INTEGRALS.
\subsection{Calculating the Local Integrals} \label{sub-cal-local}

  In this subsection, we shift our focus away from the global $\hat G_{\mathbb A}$ and into the local pieces of $\hat G_{\nu}$. We aim to calculate some local integrals that will help us determine the value of the integral in Proposition \ref{prop-step-2}. We will see in the next subsection that the integral in Proposition \ref{prop-step-2} may be expressed as a product of the local integrals that we discuss in this section.

\medskip

%Defining Uw, U-w

  As before we use $\hat{\Delta}_w$ to denote $\hat{\Delta}_{W,+}\cap w\hat{\Delta}_{W,-}$. Then $\hat{\Delta}_w$ is a finite set of affine Weyl roots that we can explicitly describe.  If $w=w_{i_r}\dots w_{i_1}$ is a minimal expression in terms of the generators of $\hat{W}$, then by setting $\beta_j=w_{i_r}\dots w_{i_{j+1}}\alpha_{i_j}$, we have $\hat{\Delta}_w=\{\beta_1, \dots, \beta_r\}$.  Using these roots we can completely describe $$\hat{U}_{w,\nu}=\{\chi_{\beta_r}(s_r)\dots\chi_{\beta_1}(s_1) \ | \ s_i\in F_\nu\}.$$ For more information see (\cite{R},\ \S6) and (\cite{LG2},\ \S6). Moreover, each element in this group is uniquely expressed in this way, so if we set $U_{\beta_i,\nu}=\{\chi_{\beta_i}(s)\ | \ s\in F_\nu\}$, then we have that $\hat{U}_{w,v}$ uniquely decomposes into $U_{\beta_r,\nu}\dots U_{\beta_1,\nu}$.

  In $\S13$ of \cite{LG1}, we see that for any $\beta \in\hat{\Delta}_W$ the effect of conjugation by $w$ is  \begin{equation}w \, \chi_{\beta}(s) \, w^{-1} \in U_{w\cdot \beta, \nu} \, .\end{equation} So if we set $U_{-,w,\nu}=w^{-1} \hat{U}_{w,\nu} w$, then $$U_{-,w,\nu}=U_{w^{-1}\beta_r}\dots U_{w^{-1}\beta_1}$$ with uniqueness of expression. Calculating these roots we find:
\begin{eqnarray*}
w^{-1}\beta_j&=&w^{-1}(w_{i_r} \dots w_{i_{j+1}}\alpha_{i_j})=(w_{i_1}\dots w_{i_r})(w_{i_r} \dots w_{i_{j+1}}\alpha_{i_j})\\
&=&w_{i_1}\dots w_{i_j}\alpha_{i_j}= -w_{i_1}\dots w_{i_{j-1}}\alpha_{i_j}.
\end{eqnarray*}
For convenience we set $\gamma_j= w_{i_1}\dots w_{i_{j-1}}\alpha_{i_j}$ and conclude
\begin{equation}U_{-,w,\nu}=U_{-\gamma_r,\nu}\dots U_{-\gamma_1, \nu}.\end{equation}

\begin{Rmk} \label{rmk-measure} Each of the spaces $U_{-\gamma_i,\nu}$ is isomorphic to $F_\nu$, so we can define a measure on these spaces using this isomorphism and the usual Haar measure $\mu_\nu$ on $F_\nu$. The measure $du_-$ on $\hat{U}_{-,w,\nu}$ may now be considered the product measure induced by the $\mu_{\nu}$.   For more information, see Appendix A.
\end{Rmk}

  If we let $w'=w_{i_{r-1}}\dots w_{i_1}$ and $\beta_j'=w_{i_{r-1}}\dots w_{i_{j+1}}\alpha_{i_j}$, then by the same construction we get that $\hat{\Delta}_{w'}=\{\beta_1', \dots, \beta_{r-1}'\}$, and hence $\hat{U}_{w',\nu}=U_{\beta_{r-1}',\nu}\dots U_{\beta_1',\nu}\,$.  Similarly, setting $U_{-,w',\nu} =(w')^{-1} \ \hat{U}_{-,w',\nu} \ w'$, we obtain its unique decomposition $$U_{-\gamma_{r-1},\nu}\dots U_{-\gamma_1,\nu}\,.$$

Recall that we fixed $\chi\in\hat{\frak h}^*$, so for every $\nu\in\mathcal V$ we can define a function $\Phi_{\chi}:\hat G_{\nu}\rightarrow\mathbb R_{>0}$ in precisely the same way we defined $\Phi_{\chi}$ on $\hat G_{\mathbb A}$. If $g=khu$ and $h=\prod_{i=1}^{l+1} h_{\alpha_i}(s_i)$ for $s_i\in F_\nu^{\times}$, then we let $$\Phi_{\chi}(g)=h^{\chi}=\prod_{i=1}^{l+1} |s_i|_{\nu}^{\chi(h_{\alpha_i})}.$$
\begin{Rmk}
Using the same argument of Proposition \ref{prop-well}, we may conclude that this map is well-defined.
\end{Rmk}

  It is our goal in this subsection to prove Proposition 4.16 which calculates the value of the local integral involving the function $\Phi_{\chi}$ on $\hat G_{\nu}$.  As before, we let $h_{\alpha}$ denote the co-root corresponding to $\alpha\in\hat{\Delta}_W$, and in particular, $h_{\alpha_i}$ for $i=1,\dots,l+1$ denote the simple co-roots corresponding to the simple roots $\alpha_i$. Finally, let $\rho$ be the element of $\hat{\frak h}^*$ defined by $\rho(h_{\alpha_i})=1$ for $i=1,\dots, l+1$.

%Statement of Prop to calculate local integrals.

\begin{Prop} \label{prop-local}
Assume $\chi(h_{\alpha_i})<-2$ for $i=1, \dots, l+1$. Then for any $\nu\in\mathcal V$ and $w\in\hat{W}$
$$\int_{U_{-,w,\nu}} \Phi_{\chi}(u_{-}) \ du_{-}= \prod_{\alpha\in\hat{\Delta}_{W,+}\cap w^{-1} \hat{\Delta}_{W,-}} \ \frac{1-\frac{1}{q_\nu^{-(\chi+\rho)(h_{\alpha})+1}}}{1-\frac{1}{q_\nu^{-(\chi+\rho)(h_{\alpha})}}}.$$
\end{Prop}

 The above identity is an affine analogue of  the Gindikin-Karpelevich formula (\cite{La}).   The proof is by induction on the length of $w\in\hat{W}$.   Our first step is to consider a local integral over a unipotent subgroup of $SL_2(F_\nu)$.

  For each $\nu\in \mathcal V$ and $a\in\hat{\Delta}_W$, we have a unique group homomorphism $\varphi_a$ from $ SL_2(F_\nu)$ into $\hat G_{\nu}$ by Lemma \ref{lem-hom}.  Moreover, $SL_2(F_\nu)$ has an Iwasawa decomposition into $K A U$ (\cite{IM}), where
\[
K = SL_2(\mathcal O_{\nu}),\quad A =  \big\{(\begin{smallmatrix} a&0\\0&a^{-1}\end{smallmatrix})\ | \ a\in F_\nu^{\times} \big\}, \quad \text{ and }\quad
U = \big\{(\begin{smallmatrix} 1& \xi\\0&1&\end{smallmatrix})\ | \ \xi\in F_\nu \big\}.
\]

We can define a real character on $A$ by fixing a real number $\kappa$ and setting
$$ \mathbf{a}^{\kappa}=(\begin{smallmatrix} a&0\\0&a^{-1}\end{smallmatrix})^{\kappa}= |a|_{\nu}^{\kappa}.$$
Using this character, we define the function $\tilde{\Phi}_{\kappa}:SL_2(F_\nu)\rightarrow\mathbb R_{>0}$ by  \[ \tilde{\Phi}_{\kappa}(g)=\tilde{\Phi}_{\kappa}(\mathbf k \mathbf a \mathbf u )=\mathbf a^{\kappa} \quad \text{ for } g= \mathbf k \mathbf a \mathbf u \in SL_2(F_\nu) ,\] where $\mathbf k \in K$, $\mathbf a \in A$ and $\mathbf u \in U$.

\begin{Rmk}
The Iwasawa decomposition of $SL_2(F_\nu)$ is not unique for an element.  However in \cite{IM}, we see that if $g=\mathbf k \,\mathbf a\,\mathbf u =\mathbf k '\,\mathbf a'\,\mathbf u '$, where $\mathbf a=(\begin{smallmatrix} a&0\\0&a^{-1}\end{smallmatrix})$ and $\mathbf a'=(\begin{smallmatrix} a'&0\\0&a'^{-1}\end{smallmatrix})$, then $|a|_{\nu}=|a'|_{\nu}$.  As a result, $\tilde{\Phi}_{\kappa}$ is well-defined.
\end{Rmk}

  We let $U_{-,\nu}=\big\{(\begin{smallmatrix} 1& 0\\ s&1&\end{smallmatrix})\ | \ s\in F_\nu \big\}\leq SL_2(F_\nu)$, and note that this group is isomorphic to the additive group $F_\nu$, and so we define a measure $d\tilde{u}_{-}$ on $U_{-,\nu}$ to be the Haar measure $\mu_\nu$ on $F_\nu$ normalized so that $\mathcal O_{\nu}$ has a total measure of 1. Then the following lemma is well-known.

%Gindikin-Karpelevich Formula
\begin{Lem} \label{lem-GK-SL2}
If we fix a real number $\kappa<-2$, then for any $\nu\in\mathcal V$ we have $$\int_{U_{-,\nu}} \tilde{\Phi}_{\kappa}(u_{-}) \ d\tilde{u}_{-}= \frac{1-\frac{1}{(q_\nu)^{-\kappa}}}{1-\frac{1}{(q_\nu)^{-(\kappa+1)}}}.$$
\end{Lem}

  Observe that the map $\varphi_a$ of Lemma \ref{lem-hom} provides an isomorphism between $U_{-,\nu}$ and $U_{-a,\nu}$. Moreover, the measures $d\tilde{u}_-$ and $du_-$ are identified under this isomorphism. As a result, we can equate the following integrals
$$\int_{U_{-a,\nu}} \Phi_{\chi}(u_{-}) \ du_{-} \ =\ \int_{U_{-,\nu}} \tilde{\Phi}_{\kappa} (\varphi_{a}^{-1}(u_{-}) ) \ d\tilde{u}_{-} \ ,$$
assuming that we choose $\kappa\in\mathbb R$ correctly.
%G-K corollary used for induction argument.
\begin{Lem}  \label{lem-GK-A} Fix $\chi\in\hat{\frak h}^*$ to be real valued such that  $\chi(h_{\alpha_i})<-2$. Then for any $\nu\in\mathcal V$ and $a\in\hat{\Delta}_W$,
\begin{equation} \label{eqn-GK-A} \int_{U_{-a,\nu}} \Phi_{\chi}(u_{-}) \ du_{-}= \frac{1-\frac{1}{(q_\nu)^{-\chi(h_{a})}}}{1-\frac{1}{(q_\nu)^{-(\chi(h_{a})+1)}}}. \end{equation}
\end{Lem}

\begin{proof}  Let $(\begin{smallmatrix} 1&0\\s&1\end{smallmatrix})\in {U}_{-,\nu}$. Then $\varphi_{a}(\begin{smallmatrix} 1&0\\s&1\end{smallmatrix})=\chi_{-{a}}(s)=u_- \in \hat{U}_{-{a},\nu}$.  The Iwasawa decomposition of $SL_2(F_\nu)$ implies $$(\begin{smallmatrix} 1&0\\s&1\end{smallmatrix})=\mathbf k \, (\begin{smallmatrix} b&0\\0&b^{-1}\end{smallmatrix})\, \mathbf u .$$
As a result, in an Iwasawa decomposition for $u_-$, we may take its $\hat{H}_{\nu} $-component to be of the form $\varphi_{a}(\begin{smallmatrix} b&0\\0&b^{-1}\end{smallmatrix})=h_{a}(b)$ for some $b\in F_\nu^{\times}$. Our goal is to choose $\kappa\in\mathbb R$ so that $$\tilde{\Phi}_{\kappa}(\varphi_{a}^{-1}(u_{-}))=\Phi_{\chi}(u_{-})\text{ for }u_{-}\in U_{-{a},\nu}. $$ We claim that we must set $\kappa=\chi(h_{a})$. Indeed, with this choice,  we have
\[ \tilde{\Phi}_{\kappa}(\varphi_{a}^{-1}(u_{-}))=  (|b|_{\nu})^{\kappa} = (|b|_{\nu})^{\chi(h_{a})}= \Phi_{\chi}(h_{a}(b)) =\Phi_{\chi}(u_{-}) .\]
Now the identity (\ref{eqn-GK-A})  is a direct result of Lemma \ref{lem-GK-SL2} and our choice of $\kappa$.
\end{proof}

  Armed with Lemma \ref{lem-GK-A}, we are prepared to prove Proposition \ref{prop-local}:

%PROOF of Prop 4
\begin{proof}[Proof of Proposition \ref{prop-local}]
We assume $\chi(h_{\alpha_i})<-2$ for $i=1, \dots, l+1$, and we want to show that for any $\nu\in\mathcal V$ and $w\in\hat{W}$
$$\int_{U_{-,w,\nu}} \Phi_{\chi}(u_{-}) \ du_{-}= \prod_{a\in\hat{\Delta}_{W,+}\cap w^{-1} \hat{\Delta}_{W,-}} \ \frac{1-\frac{1}{q_\nu^{-(\chi+\rho)(h_{a})+1}}}{1-\frac{1}{q_\nu^{-(\chi+\rho)(h_{a})}}},$$
As mentioned previously, the proof is by induction on $l(w)$, the length of the Weyl group element.

\textbf{\underline{Base case:}} \hskip 4mm To prove the base case, we assume that $l(w)=1$, and therefore that $w=w_i$ for some $i=1,\dots,l+1$. We begin by observing that $\hat{\Delta}_{w_i}=\{\alpha_i\}$. Hence we have $\hat{U}_{w_i, \nu}=\{\chi_{\alpha_i}(s)\ | \ s\in F_\nu\}=U_{\alpha_i,\nu}$. Moreover, $U_{-,w_i,\nu}=w_i^{-1}U_{\alpha_i,\nu} w_i =U_{w_i^{-1}\cdot \alpha_i,\nu}=U_{-\alpha_i,\nu}$, and so
\[
\int_{U_{-,w_i,\nu}} \Phi_{\chi}(u_{-}) \ du_{-}=  \int_{U_{-\alpha_i,\nu}} \Phi_{\chi}(u_{-}) \ du_{-} = \frac{1-\frac{1}{(q_\nu)^{-(\chi+\rho)(h_{\alpha_i})+1}}}{1-\frac{1}{(q_\nu)^{-(\chi+\rho)(h_{\alpha_i})}}}  \]
by Lemma \ref{lem-GK-A} and the fact that $\rho(h_{\alpha_i})=1$.

\vskip 4mm

\textbf{\underline{Induction step:}} \hskip 4mm Now suppose we choose $w\in\hat{W}$ with reduced expression $w_{i_r}\dots w_{i_1}$, and that our proposition holds for all $w'\in \hat{W}$ such that $l(w')<l(w)$. Specifically, we set $w' = w_{i_{r-1}}\dots w_{i_1}$.
At the beginning of this subsection, we showed that $U_{-,w,\nu}=U_{-\gamma_r,\nu}U_{-,w',\nu}$, and so our integral breaks into
\begin{equation}  \label{eqn-break} \int_{U_{-\gamma_r,\nu}} \int_{U_{-,w',\nu}}  \Phi_{\chi}(u_{-,1}\,u_{-,2}) \ du_{-,2} \ du_{-,1}.\end{equation}
The measure $du_-$ decomposes naturally by Remark \ref{rmk-measure}. Let the element $u_{-,1}$ have the Iwasawa decomposition $k_1 h_1 u_1$. Recall that by definition $\Phi_{\chi}$ is left invariant by $\hat{K}_\nu$ and right invariant by $\hat{U}_\nu$. In light of these observations and part (1) of Lemma \ref{lem-123}, we can make the following manipulations:
\begin{eqnarray*}
\Phi_{\chi}(u_{-,1}u_{-,2}) \ &=& \ \Phi_{\chi}(k_1 h_1 u_1 u_{-,2})= \ \Phi_{\chi}(h_1 u_1 u_{-,2})\\
&=& \  \Phi_{\chi}(h_1 u_1 u_{-,2} u_1^{-1} h_1^{-1}h_1u_1)= \  \Phi_{\chi}(h_1 u_1 u_{-,2} u_1^{-1} h_1^{-1} h_1) \\
&=& \  \Phi_{\chi}(h_1 u_1 u_{-,2} u_1^{-1} h_1^{-1})\Phi_{\chi}(h_1)=  \  h_1^{\chi}\ \Phi_{\chi}(h_1 u_1 u_{-,2} u_1^{-1} h_1^{-1}).
\end{eqnarray*}
The integral in (\ref{eqn-break}) now becomes
\begin{equation} \label{eqn-become} \int_{U_{-\gamma_r,\nu}} h_1^{\chi} \int_{U_{-,w',\nu}}  \Phi_{\chi}(h_1 u_1 u_{-,2} u_1^{-1} h_1^{-1}) \ du_{-,2}\  du_{-,1}\,. \end{equation}
  We wish to show $u_1 u_{-,2} u_1^{-1}\in   w'^{-1}\hat{U}_\nu w'$, and note that it suffices to prove that $w'\,(u_1 u_{-,2} u_1^{-1})\, w'^{-1}\in \hat{U}_\nu$.  We have
\[
w'u_1 u_{-,2} u_1^{-1} w'^{-1}= (w'u_1w'^{-1})(w'u_{-,2}w'^{-1})(w'u_1^{-1} w'^{-1}) .
\]
Since $u_{-,2}\in U_{-,w',\nu},$ we have $w'u_{-,2}w'^{-1}\in\hat{U}_{w',\nu} \subset \hat{U}_\nu$ by definition.  For any $u\in U_{\gamma_r}$ we have $w' u w'^{-1}\in U_{w'\cdot \gamma_r}=U_{\alpha_{i_r}}$, and so both $w'u_1w'^{-1}$ and $w'u_1^{-1}w'^{-1}$ are elements of $U_{\alpha_{i_r}} \subset \hat{U}_\nu$. Hence, $$u_1 u_{-,2} u_1^{-1}\in  w'^{-1}\hat{U}_\nu w'.$$

The decomposition (\ref{eqn-decom}) provides us with the unique group decomposition
$$w'^{-1}\hat{U}_\nu w'= (U_{-,w',\nu})(w'^{-1}\hat{U}_\nu w' \cap \hat{U}_\nu).$$
Let $\pi_{\nu}$ be the projection from $w'^{-1}\hat{U}_\nu w'\twoheadrightarrow U_{-,w',\nu}$, which exists by the decomposition above. Then $$u_1 u_{-,2} u_1^{-1}=\pi_{\nu}(u_1 u_{-,2} u_1^{-1}) u^+_2\, , \text{ for }u_2^+\in w'^{-1}\hat{U}_\nu w' \cap \hat{U}_\nu.$$  Since $\hat{H}_{\nu} $ normalizes $\hat{U}_\nu$, we get
\begin{eqnarray*}
\Phi_{\chi}(h_1 u_1 u_{-,2} u_1^{-1} h_1^{-1})& = &\Phi_{\chi}(h_1\ \pi_{\nu}(u_1 u_{-,2} u_1^{-1})\,  u_2^+ h_1^{-1}) = \Phi_{\chi}(h_1\ \pi_{\nu}(u_1 u_{-,2} u_1^{-1})\,  h_1^{-1}h_1u_2^+ h_1^{-1})\\
&=&  \Phi_{\chi}(h_1\  \pi_{\nu}(u_1 u_{-,2} u_1^{-1})\,  h_1^{-1} u')=  \Phi_{\chi}(h_1\  \pi_{\nu}(u_1 u_{-,2} u_1^{-1})\,  h_1^{-1}).
\end{eqnarray*}
The map from $U_{-,w',\nu}$ to itself defined by $u_{-,2}\mapsto \pi_{\nu}( u_1 u_{-,2} u_1^{-1})$ is a unimodular change of variables, so the integral (\ref{eqn-become}) becomes
$$\int_{U_{-\gamma_r,\nu}} h_1^{\chi} \int_{U_{-,w',\nu}}  \Phi_{\chi}(h_1 u_{-,2} h_1^{-1}) \ du_{-,2} du_{-,1}.$$
Now applying the change of variables $ h_1 u_{-,2} h_1^{-1} \mapsto u_{-,2}$ has the following effect on our integral:
$$\int_{U_{-\gamma_r,\nu}} h_1^{\chi} \int_{U_{-,w',\nu}}  \Phi_{\chi}(h_1 u_{-,2} h_1^{-1}) \ du_{-,2} du_{-,1}=\int_{U_{-\gamma_r,\nu}} h_1^{\chi} h_1^{-\Upsilon}\int_{U_{-,w',\nu}}  \Phi_{\chi}(u_{-,2}) \ du_{-,2} du_{-,1},$$
where we have \[  \Upsilon=  \sum_{a\in(w')^{-1}\hat{\Delta}_{W,+} \cap \hat{\Delta}_{W,-}}  \hskip -0.7 cm a \qquad =(w')^{-1}\rho -\rho \] (See \cite{GaLep}).

   Our calculations so far have proven the following result regarding our local integrals:
\begin{eqnarray*}
\int_{U_{-,w,\nu}} \Phi_{\chi}(u_{-}) \ du_{-} \ &=& \ \int_{U_{-\gamma_r,\nu}} h_1^{\chi +\rho-(w')^{-1}\rho}\int_{U_{-,w',\nu}}  \Phi_{\chi}(u_{-,2}) \ du_{-,2} du_{-,1}\\
&=& \ \int_{U_{-\gamma_r,\nu}} h_1^{\chi +\rho-(w')^{-1}\rho} \ du_{-,1} \ \int_{U_{-,w',\nu}}  \Phi_{\chi}(u_{-,2}) \ du_{-,2}.
\end{eqnarray*}
By our inductive hypothesis,
$$\int_{U_{-,w',\nu}}  \Phi_{\chi}(u_{-,2}) \ du_{-,2}= \prod_{a\in\hat{\Delta}_{W,+}\cap w'^{-1}(\hat{\Delta}_{W,-})} \frac{1-\frac{1}{q_\nu^{-(\chi+\rho)(h_{a})+1}}}{1-\frac{1}{q_\nu^{-(\chi+\rho)(h_{a})}}}.$$
Moreover, observe that
$$\int_{U_{-\gamma_r,\nu}} h_1^{\chi +\rho-(w')^{-1}\rho} \ du_{-,1}=\int_{U_{-\gamma_r,\nu}} \Phi_{\chi +\rho-(w')^{-1}\rho}(u_-) \ du_{-,1},$$
and so by Lemma \ref{lem-GK-A} we obtain
$$\int_{U_{-\gamma_r,\nu}} h_1^{\chi +\rho-(w')^{-1}\rho} \ du_{-,1}= \frac{1-\frac{1}{q_\nu^{-(\chi +\rho-(w')^{-1}\rho)(h_{\gamma_r})}}}{1-\frac{1}{q_\nu^{-((\chi +\rho-(w')^{-1}\rho)(h_{\gamma_r})+1)}}}.$$

Since we have
\[
(w')^{-1}\rho(h_{\gamma_r}) =  \rho(w'\cdot h_{\gamma_r})=  \rho(h_{(w')^{-1}\cdot \gamma_r})= \rho(h_{\alpha_{i_r}})= 1 ,
\]
we obtain $$(\chi +\rho-(w')^{-1}\rho)(h_{\gamma_r})=(\chi+\rho)(h_{\gamma_r})-1,$$ and
$$\int_{U_{-\gamma_r,\nu}} h_1^{\chi +\rho-(w')^{-1}\rho} \ du_{-,1} \ = \ \frac{1-\frac{1}{q_\nu^{-(\chi+\rho)(h_{\gamma_r})+1}}}{1-\frac{1}{q_\nu^{-(\chi+\rho)(h_{\gamma_r})}}}.$$
Putting all of these calculations together we see
\begin{eqnarray*}
\int_{U_{-,w,\nu}} \Phi_{\chi}(u_{-}) \ du_{-} \ &=& \  \int_{U_{-\gamma_r,\nu}} h_1^{\chi +\rho-(w')^{-1}\rho} \ du_{-,1}  \, \  \int_{U_{-,w',\nu}}  \Phi_{\chi}(u_{-,2}) \ du_{-,2} \\
\ &=&
\  \frac{1-\frac{1}{q_\nu^{-(\chi+\rho)(h_{\gamma_r})+1}}}{1-\frac{1}{q_\nu^{-(\chi+\rho)(h_{\gamma_r})}}}  \, \left (  \prod_{a\in\hat{\Delta}_{+}\cap w'^{-1}(\hat{\Delta}_{-})} \frac{1-\frac{1}{q_\nu^{-(\chi+\rho)(h_{a})+1}}}{1-\frac{1}{q_\nu^{-(\chi+\rho)(h_{a})}}}  \right ) .
\end{eqnarray*}
Finally, since $\hat{\Delta}_{W,+}\cap w^{-1} \hat{\Delta}_{W,-}=\{\gamma_r\}\cup (\hat{\Delta}_+ \cap w'^{-1}\hat{\Delta}_-)$, we obtain the desired result.
\end{proof}

%%% SECTION PUTTING IT TOGETHER
\subsection{Finishing the Computation}

  In this final subsection, we use the previous results to finish our calculation of the constant term $E^{\#}_{\chi}$.  Recall that by Proposition \ref{prop-step-2}, we have
$$E^{\#}_{\chi}(g\eta^{mD})=\sum_{w\in\hat{W}} \ (h\eta^{mD})^{w(\chi+\rho)-\rho}\int_{\hat{U}_{-,w,\mathbb A}} \Phi_{\chi}(u_-) du_-.$$
In our next step, we use Proposition \ref{prop-local} to evaluate $\int_{\hat{U}_{-,w,\mathbb A}} \Phi_{\chi}(u_-) du_-.$

Note that the space $\hat{U}_{-,w,\mathbb A}$ can be identified with the product of $\ell(w)$ copies of $\mathbb A$.
Since we have assumed that $\chi$ is real, we can apply the monotone convergence theorem to see that
$$\int_{\hat{U}_{-,w,\mathbb A}} \Phi_{\chi}(u_-) du_-=  q^{\ell(w)(1-g)} \displaystyle\lim_{S}\ \ \prod_{\nu\in S} \ \ \int_{U_{-,w,\nu}} \Phi_{\chi}(u_{-}) \ du_{-},$$ where we use the relation (\ref{eqn-me-la}) and the following remark there.
In the expression above, we take $S$ to range over the finite subsets of $\mathcal V$. Now by Proposition \ref{prop-local}, this is equal to
\begin{eqnarray*}
& & \lim_{S}\ \ \prod_{\nu\in S} \ \ \prod_{a\in\hat{\Delta}_{W,+}\cap w^{-1} \hat{\Delta}_{W,-}}  \frac{1-\frac{1}{q_\nu^{-(\chi+\rho)(h_{a})+1}}}{1-\frac{1}{q_\nu^{-(\chi+\rho)(h_{a})}}}  \\
&=&\prod_{a\in\hat{\Delta}_{W,+}\cap w^{-1} \hat{\Delta}_{W,-}}  \displaystyle\lim_{S} \ \  \prod_{\nu\in S}\ \  \frac{1-\frac{1}{q_\nu^{-(\chi+\rho)(h_{a})+1}}}{1-\frac{1}{q_\nu^{-(\chi+\rho)(h_{a})}}}  \\
&=&\prod_{a\in\hat{\Delta}_{W,+}\cap w^{-1} \hat{\Delta}_{W,-}} \ \  \prod_{\nu\in\mathcal V} \  \ \frac{1-\frac{1}{q_\nu^{-(\chi+\rho)(h_{a})+1}}}{1-\frac{1}{q_\nu^{-(\chi+\rho)(h_{a})}}}  .
\end{eqnarray*}

  Let $\zeta_F(s)$ denote the zeta function associated to the function field $F$ (see \cite{Ros}). Then $$\zeta_F(s)=\prod_{\nu\in\mathcal V} \ \ \frac{1}{1-\frac{1}{q_\nu^s}},$$ whenever $\mathrm{Re}(s)>1$. Since we have assumed $\chi(h_{\alpha_i})<-2$, we have that $-(\chi +\rho)(h_a)>1$ for any $a\in\hat{\Delta}_{W,+}$.  As a result, we obtain that
$$\prod_{\nu\in\mathcal V} \  \ \frac{1-\frac{1}{q_\nu^{-(\chi+\rho)(h_{a})+1}}}{1-\frac{1}{q_\nu^{-(\chi+\rho)(h_{a})}}}=\frac{\zeta_F (-(\chi+\rho)(h_{a}) )}{\zeta_F(-(\chi+\rho)(h_a)+1)}.$$
We set $$c(\chi,w)= q^{\ell(w)(1-g)} \prod_{a\in\hat{\Delta}_{W,+}\cap w^{-1} \hat{\Delta}_{W,-}} \ \ \frac{\zeta_F(-(\chi+\rho)(h_{a}))}{\zeta_F(-(\chi+\rho)(h_a)+1)}.$$
Finally, we obtain the main result of this section:

\begin{Thm}  \label{thm-const}
For any $g=khu\in\hat G_{\mathbb A}$ and $\chi\in\hat{\frak h}^*$ such that $\chi(h_{\alpha_i})<-2$ for $i=1,\dots,l+1$, we have
\begin{equation} \label{eqn-sum-h} E^{\#}_{\chi}(g\eta^{mD})=\sum_{w\in\hat{W}} \ (h\eta^{mD})^{w(\chi+\rho)-\rho} \ c(\chi, w).\end{equation}
\end{Thm}

   We saw at the end of Section 3 that in order to prove the almost everywhere convergence of the series $E_{\chi}$, it suffices to show \begin{equation} \label{eqn-again}   \sum_{w\in\hat{W}} \ \ \int_{\hat{U}_{\mathbb A}/(\hat{U}_{\mathbb A}\cap\hat{\Gamma}_F)}  E_{\chi, w}(h\eta^{mD} u)  \ du \ < \infty \end{equation} for $h$ varying in compact sets of $\hat H_{\mathbb A}$, when $\chi$ is a real character. In this section, we showed that the series (\ref{eqn-again}) is the same as  the series (\ref{eqn-sum-h}).  So to establish the almost everywhere convergence of the Eisenstein series $E_{\chi}$, we direct our attention to proving that the series (\ref{eqn-sum-h}) converges for $h$ varying in compact sets of $\hat H_{\mathbb A}$. We will prove this result in the next section.

  We also note that Theorem  \ref{thm-const} closely resembles Garland's result for the constant term of the Eisenstein series over $\hat G_{\mathbb R}$ (\cite{R}) with the Riemann zeta function replaced by the zeta function of the function field.

\vskip 1 cm

% Section 5, convergence of the constant term.
\section{Convergence of the Constant Term}

In this section, we prove the convergence of the series (\ref{eqn-sum-h}) by showing that if $h$ varies in a compact subset of $\hat H_{\mathbb A}$, then the series (\ref{eqn-sum-h}) is bounded above by a theta series in $\mathbb R$.  In order to establish this result, we require an additional condition on the tuple $m=(m_\nu)_{\nu\in\mathcal V}$, which determines the automorphism $\eta^{mD}$.  Specifically, we will further assume that $\sum_{\nu\in\mathcal V} \log(q_\nu)m_v>0$; this will be important for the calculation in section 5.4.

  We will approach this proof by treating the factor $(h\eta^{mD})^{w(\chi+\rho)-\rho}$ in subsections \ref{subsec-5.2} through \ref{subsec-5.5}, and the factor $c(\chi,w)$ in subsection \ref{subsec-5.6} . In the final subsection, we will combine the results to finish proving the convergence of  the series (\ref{eqn-sum-h}), which is indeed the constant term $E^{\#}_{\chi}$ of $E_\chi$. We begin with considering the compact subsets of $\hat H_{\mathbb A}$.

\subsection{Topology of the Torus and Compact Sets}

  Recall that $(\mathbb A^{\times})^{l+1}$ has the product topology induced by the standard topology on $\mathbb A^{\times}$. In Remark \ref{rmk-iso}, we defined the surjective group homomorphism $\vartheta: (\mathbb A^{\times})^{l+1} \rightarrow \hat H_{\mathbb A}$. Through this map, we put the quotient topology on $\hat H_{\mathbb A}$. Clearly the id\`{e}lic norm $|\cdot|: \mathbb A^\times\rightarrow\mathbb R_{>0}$ is a continuous map. Moreover, one can see that if $ \vartheta (s_1, \dots , s_{l+1})=\prod_{i=1}^{l+1} h_{\alpha_i}(s_i)=1$ then $|s_i|=1$ for each $i$. As a result, we have a well-defined continuous map $\bar \vartheta : \hat H_{\mathbb A} \rightarrow (\mathbb R_{>0})^{l+1}$ given by $\prod_{i=1}^{l+1} h_{\alpha_i}(s_i) \mapsto (|s_1|, \dots , |s_{l+1}|).$

     Let $\mathbf{C}$ be a compact set of $\hat H_{\mathbb A}$ and $\mathrm{pr}_i$ be the $i$-th projection of $(\mathbb R_{>0})^{l+1}$.  Then, for each $i$,  the image $\mathrm{pr}_i(\bar \vartheta(\mathbf{C}))$ is compact in $\mathbb R_{>0}$ and there exist $r_i,  R_i \in \mathbb R_{>0}$ such that \[  r_i <  \mathrm{pr}_i(\bar \vartheta(h)) < R_i \qquad \text{for any } h \in \mathbf C .\] If we set \[ r =\min \{ r_1, \dots , r_{l+1} \} \quad \text{ and } \quad R=\max \{ R_1, \dots , R_{l+1} \}, \] then for any  $h = \prod_{i=1}^{l+1} h_{\alpha_i}(s_i)  \in\mathbf{C}$, we have $$r<|s_i|<R$$ for $i=1,\dots,l+1$. Therefore, we conclude:

\begin{Lem} \label{lem-compact}
Let $\mathbf C$ be a compact set of $\hat H_{\mathbb A}$. Then there exists $r,R\in\mathbb R_{>0}$ such that for any $h=h_{\alpha_1}(s_1) \dots h_{\alpha_{l+1}}(s_{l+1})\in \mathbf C$ we have
$$ r < |s_i|< R$$ for $i=1,\dots,l+1$.
\end{Lem}

  For the rest of the section, we fix a compact subset $\mathbf C$ of $\hat H_{\mathbb A}$ along with positive real numbers $r$ and $R$ that satisfy the conditions of Lemma \ref{lem-compact}.  We wish to prove that
\begin{equation} \sum_{w\in\hat{W}} \, (h\eta^{mD})^{w(\chi+\rho)-\rho} \ c(\chi,w) \end{equation}
is bounded by a theta series in $\mathbb R$ as $h$ varies in $\mathbf C$. Clearly
\[ \sum_{w\in\hat{W}} \,(h\eta^{mD})^{w(\chi+\rho)-\rho} \ c(\chi,w)= (h\eta^{mD})^{-\rho}  \sum_{w\in\hat{W}} \, (h\eta^{mD})^{w(\chi+\rho)} c(\chi,w) . \]
In the next four subsections, we focus on finding a bound for $(h\eta^{mD})^{w(\chi+\rho)} $.

\subsection{Preliminary Calculation} \label{subsec-5.2}

In Section 3, we fixed the automorphism $\eta^{mD}$ where $m=(m_{\nu})_{\nu\in\mathcal V}$ is an infinite tuple of integers such that  $m_{\nu}=0$ for almost all $\nu\in\mathcal V$.  However, we could replace $D$ with any element of $\hat{\frak h}^{\,e}$ to obtain a similar automorphism.  To be more precise, let $n=(n_{\nu})_{\nu\in\mathcal V}$ be an infinite tuple of integers where $n_{\nu}=0$ for almost all $\nu\in\mathcal V$.  For $h\in\hat{\frak h}^{\,e}$ and $\nu\in\mathcal V$, we define the local automorphism $\eta_{\nu}^{n_\nu h}$ to
\begin{enumerate}
\item[($i$)] preserve each weight space $V^{\lambda}_{\mu, F_\nu}\,$, and
\item[($ii$)] act on each weight space by $\eta_{\nu}^{n_\nu h} \cdot v= \pi_{\nu}^{n_\nu \mu(h)} \, v$ for $v\in V^{\lambda}_{\mu,F_\nu}$.
\end{enumerate}
Then we consider the product of these local automorphisms and define the global automorphism $\eta^{n h}:=\prod_{\nu\in\mathcal V} \eta_\nu^{n_\nu h}.$
Note that this automorphism only affects a finite number of places.

An element $\chi\in\hat{\frak h}^*$ can be considered an element of $(\hat{\frak h}^{\,e})^*$ by setting $\chi(D)=0$.  With this in mind, for $\chi\in\hat{\frak h}^*$ we define
\begin{equation} (\eta^{n h})^{\chi}= \prod_{\nu\in\mathcal V} \ |\pi_{\nu}|_{\nu}^{n_\nu \chi(h)}.
\end{equation}
One can easily see that for $w \in \hat W$ we have
$$(\eta^{n h} )^{w(\chi)}=(\eta^{\ w^{-1}(n h)})^{\chi}.$$

  If $s=(s_\nu)_{\nu\in\mathcal V} \in\mathbb A^\times$ then $\mathrm{ord}(s):=(\mathrm{ord}_{\nu}(s_\nu))_{\nu\in\mathcal V}$ is a tuple of integers with the property that $\mathrm{ord}_{\nu}(s_\nu)=0$ for almost all $\nu$.
Then we see that \begin{equation} \label{eqn-ord-s} (\eta^{\mathrm{ord}(s)\, h_{\alpha_i}})^{\chi}=(h_{\alpha_i}(s))^{\chi} \qquad \text{for each } i. \end{equation} Moreover, since $\chi(D)=0$ we also have $(\eta^{mD})^{\, \chi}=1.$

   As a notational convenience, we set $\bar{\chi}:=\chi+\rho$ and consider the factor
 \begin{equation} \label{eqn-mD} (h\eta^{mD})^{w(\bar{\chi})} \end{equation}
 for some $w\in\hat{W}$.  Let $h=h_{\alpha_1}(s_1) \dots h_{\alpha_{l+1}}(s_{l+1})$. Then by our observation in (\ref{eqn-ord-s}) we have
 \[
 h^{w(\bar{\chi})} = \prod_{i=1}^{l+1} |s_i|^{\,w(\bar{\chi})(h_{\alpha_i})}= \left( \ \prod_{i=1}^{l+1} \ \eta^{\,\mathrm{ord}(s_i) h_{\alpha_i}}  \ \right)^{w(\bar{\chi})}.
 \]
 As a result, the factor (\ref{eqn-mD}) becomes
 \begin{eqnarray}
\left(\ \prod_{i=1}^{l+1} (\eta^{\,\mathrm{ord}(s_i) h_{\alpha_i}})\  \eta^{mD}\ \right )^{w(\bar{\chi})} & = & \left(\ \eta^{\,\mathrm{ord}(s_1)h_{\alpha_1} +\dots+\mathrm{ord}(s_{l+1})h_{\alpha_{l+1}} + mD}  \ \right)^{w(\bar{\chi})} \nonumber \\ &=& \left(\ \eta^{w^{-1}(\mathrm{ord}(s_1)h_{\alpha_1} \,+\,\dots\,+\,\mathrm{ord}(s_{l+1})h_{\alpha_{l+1}} \,+\, mD)} \ \right)^{\bar{\chi}} \nonumber \\ &=& \prod_{\nu\in\mathcal V} \ |\pi_{\nu}|_{\nu}^{\, \bar{\chi} ( w^{-1}(\mathrm{ord}_{\nu}(s_{1,\nu})h_{\alpha_1}\, +\,\dots\,+\,\mathrm{ord}_{\nu}(s_{l+1,\nu})h_{\alpha_{l+1}} \,+\, m_{\nu}D))} \nonumber \\
&=& \prod_{\nu\in\mathcal V} \ \left ( \frac 1{q_\nu} \right )^{ \bar{\chi} ( w^{-1}(\mathrm{ord}_{\nu}(s_{1,\nu})h_{\alpha_1}\, +\,\dots\,+\,\mathrm{ord}_{\nu}(s_{l+1,\nu})h_{\alpha_{l+1}}\, + \,m_{\nu}D))}. \label{eqn-q}
\end{eqnarray}
For the remainder of this subsection, we will focus on calculating
\begin{equation} w^{-1}\big(\mathrm{ord}_{\nu}(s_{1,\nu})h_{\alpha_1} +\dots+\mathrm{ord}_{\nu}(s_{l+1,\nu})h_{\alpha_{l+1}} + m_{\nu}D\big).
\end{equation}

  The affine Weyl group of $\hat{\frak g}^{\,e}$ can be decomposed into a semi-direct product of the classical Weyl group and a group of translations: $\hat{W}= W \ltimes T$.  In particular,  $T=\{T_H \ | \ H\in\frak h_{\mathbb Z} \}$.  For more information see \cite{LG1}, \cite{R}, or \cite{Kac}. Using this product decomposition, we may express $w^{-1}=w_1 T_H$, for some $w_1\in W$, and $H\in\frak h_{\mathbb Z}$. Moreover, we note that for any $H\in\frak h_{\mathbb Z}$ the translation $T_H$ does not affect the imaginary co-root $h_{\delta}$, and we recall that the classical Weyl group leaves $h_{\delta}$ and $D$ invariant.

  In order to calculate the element of $\hat{\frak h}^{\,e}$ in (5.7), we will use this decomposition of $w^{-1}$ and a particular result from \cite{R} which we summarize in the following lemma:
\begin{Lem}[Garland, \cite{R}] \label{lem-trans} Let $h' \in\frak h$ and $T_H$ be the translation element associated to $H\in\frak h_{\mathbb Z}$. We have the following formula:
\begin{equation}T_H\cdot(h'-D)=h' + H - D + \left ( \frac{(H,H)}{2} +(h',H)  \right ) h_{\delta},
\end{equation} where $(\ ,\ )$ is the normalized bilinear form on $\frak h$.
\end{Lem}

   We let $w^{-1}=w_1 T_H$ for $w_1\in W$ and $T_H\in T$. For each $\nu\in\mathcal V$, we set \[ \sum_{i=1}^{l+1} \mathrm{ord}_{\nu}(s_{i,\nu})h_{\alpha_i}= h_{\nu} + e_{\nu} h_{\delta} ,\] where $h_{\nu}\in\frak h$ and $e_{\nu}\in\mathbb R$. We begin our calculation by making these substitutions:
\[
w^{-1}\big(\mathrm{ord}(s_{1, \nu})h_{\alpha_1} +\dots+\mathrm{ord}(s_{l+1, \nu})h_{\alpha_{l+1}} + m_\nu D\big)= w_1 \, T_H\big(h_{\nu} + e_{\nu} h_{\delta} + m_{\nu}  D\big).
\]
After factoring out $-m_{\nu} $ for $m_\nu \neq 0$,  we have
\begin{eqnarray*}
w_1 \, T_H\big(h_{\nu} + e_{\nu} h_{\delta} + m_{\nu}  D\big) \ &=& \ w_1 \, \left( -m_{\nu}  \, T_H \left(\frac{-h_{\nu} }{m_{\nu} } - D - \frac{e_{\nu} }{m_{\nu} }h_{\delta} \right)   \right)\\
&=& \ w_1 \, \left ( -m_{\nu}  \,\bigg[ T_H \left(\frac{-h_{\nu} }{m_{\nu} } - D\right) - T_H\left(\frac{e_{\nu} }{m_{\nu} }h_{\delta} \right)\bigg]   \right ).
\end{eqnarray*}
The translation $T_H$ does not affect the imaginary root, so if we apply Lemma \ref{lem-trans} we see this equals
\begin{eqnarray}
& & w_1 \, \left ( -m_{\nu}  \,\bigg[ -\frac{h_{\nu} }{m_{\nu} } + H - D + \left ( \frac{(H,H)}{2} -\bigg(\frac{h_{\nu} }{m_{\nu} },H\bigg) \right ) h_{\delta}  - \bigg(\frac{e_{\nu} }{m_{\nu} }h_{\delta} \bigg)\bigg]   \right ) \nonumber
\\
&=& \ w_1 \, \left ( h_{\nu}  -m_{\nu}  H + m_{\nu}  D + \left ( \frac{-m_{\nu} (H,H)}{2} +(h_{\nu} ,H) + e_{\nu} \right ) h_{\delta} \right )  \nonumber \\ &=& \ (w_1h_{\nu} ) - m_{\nu} (w_1H) + m_{\nu}  D + \bigg(\frac{-m_{\nu} (H,H)}{2} +(h_{\nu} ,H) + e_{\nu} \bigg)h_{\delta}, \label{eqn-Weyl-three}
\end{eqnarray}
where the last equality follows from the fact that  the classical Weyl group element $w_1$ does not affect $D$ or $h_{\delta}$. One can check that the same formula holds for $m_\nu =0$.

Substituting (\ref{eqn-Weyl-three}) into (\ref{eqn-q}), we obtain that
\begin{eqnarray*}
(h\eta^{mD})^{w(\bar{\chi})}
&=& \prod_{\nu\in\mathcal V}  \left (\frac{1}{q_\nu} \right )^{\bar{\chi} ( w^{-1}(\mathrm{ord}_{\nu}(s_{1,\nu})h_{\alpha_1} +\dots+\mathrm{ord}_{\nu}(s_{l+1,\nu})h_{\alpha_{l+1}} + m_{\nu}D))}\\
&=&  \prod_{\nu\in\mathcal V}  \left (\frac{1}{q_\nu} \right )^{\bar{\chi} ((w_1h_{\nu} ) - m_{\nu} (w_1H) + m_{\nu}  D + (\frac{-m_{\nu} (H,H)}{2} +(h_{\nu} ,H) + e_{\nu} )h_{\delta} ) }.
\end{eqnarray*}
%%%%%%%%%%%%%%%%%%%%%
%THE THREE PRODUCTS TO BOUND%
%%%%%%%%%%%%%%%%%%%%%

We continue to break up this product by writing it as a product of three factors that we obtain by grouping the parts appearing in the exponent of $\frac{1}{q_\nu}$. Since $\bar \chi(D)=0$, the factor containing $m_{\nu}D$ disappears from our computation. In particular, we have the following proposition.
\begin{Prop} \label{prop-three-factors} The term $(h\eta^{mD})^{w(\bar{\chi})}$ appearing in $E_{\chi}^{\#}(g\eta^{mD})$ is the product of the following three factors:
\begin{enumerate}[{\em (i)}]
\item \ $\prod_{\nu\in\mathcal V} \ (\frac{1}{q_\nu})^{\bar{\chi}(w_1h_{\nu} )+e_{\nu} \bar{\chi}(h_{\delta})} $
\item \ $ \prod_{\nu\in\mathcal V} \ (\frac{1}{q_\nu})^{-(m_{\nu} \bar{\chi}(w_1H) + \frac{m_{\nu} (H,H)}{2}\bar{\chi}(h_{\delta}) )} $
\item \ $ \prod_{\nu\in\mathcal V} \ (\frac{1}{q_\nu})^{(h_{\nu} ,H)\bar{\chi}(h_{\delta})}.$
\end{enumerate}
\end{Prop}
\begin{Rmk}
Note the calculations in this subsection work for any $w\in\hat{W}$.
\end{Rmk}

  In this subsection, we calculated the term $(h\eta^{mD})^{w(\bar{\chi})}$ explicitly and showed that it is a product of the factors (i), (ii), and (iii) above. In the next three subsections, we will work with each factor individually showing that if we let $h$ vary in a compact set $\mathbf C$, we have an upper bound that depends on $w\in \hat{W}$.

\subsection{Bounding Factor (i) of the Product}

	  We wish to find an upper bound for the product \begin{equation}\prod_{\nu\in\mathcal V} \left ( \frac{1}{q_\nu} \right )^{\bar{\chi}(w_1h_{\nu} )+e_{\nu} \bar{\chi}(h_{\delta})}.\end{equation}
	First notice that since we set $h_{\nu} +e_{\nu} h_{\delta}=\sum_{i=1}^{l+1} \mathrm{ord}_{\nu}(s_{i,\nu}) h_{\alpha_i}$, we have
	\begin{eqnarray*}
	\prod_{\nu\in\mathcal V} \left ( \frac{1}{q_\nu} \right )^{\bar{\chi}(w_1h_{\nu} )+e_{\nu} \bar{\chi}(h_{\delta})} &=& \prod_{\nu\in\mathcal V} |\pi_{\nu}|_{\nu}^{\bar{\chi}(w_1h_{\nu} )+e_{\nu} \bar{\chi}(w_1h_{\delta})}= \prod_{i=1}^{l+1} (\eta^{\mathrm{ord}(s_i) \ w_1(h_{\alpha_i})})^{\bar{\chi}}\\
	&=& \prod_{i=1}^{l+1} (\eta^{\mathrm{ord}(s_i) h_{\alpha_i}})^{w_1^{-1}\bar{\chi}}=  h^{w_1^{-1}\bar{\chi}},
	\end{eqnarray*}
by our observation in (\ref{eqn-ord-s}). Hence it is our goal to prove there exists an upper bound for the factor $$h^{w_1^{-1}\bar{\chi}}$$ for any $w_1\in W$, the classical Weyl group.

  Since the classical Weyl group is a finite group and we fixed $\bar{\chi}\in\hat{\frak h}$, we know there exist real numbers $m_1$ and $M_1$ such that $$m_1< w_1^{-1}\bar{\chi}(h_{\alpha_i})<M_1,$$ for all $i=1,\dots,l+1$ and $w_1\in W$.
Moreover, we assumed in the beginning of the section that we are choosing $h$ to vary in a compact set $\mathbf C$ of $\hat H_{\mathbb A}$ such that for any $h= h_{\alpha_1}(s_1)\dots h_{\alpha_{l+1}}(s_{l+1})\in \mathbf C$ we have $r< |s_i|<R$ for each $i$. Combining these, we see that there exists a positive constant $\mathcal M$ such that
\[ h_{\alpha_i}(s_i)^{w_1^{-1} \bar \chi} = |s_i|^{w_1^{-1} \bar \chi (h_{\alpha_i})} < \mathcal M \] for all $i=1,\dots,l+1$ and $w_1\in W$.  Hence, we conclude:

\begin{Lem} \label{lem-first}
\begin{equation} \label{eqn-first-factor}
\prod_{\nu\in\mathcal V} \left ( \frac{1}{q_\nu} \right )^{\bar{\chi}(w_1h_{\nu} )+e_{\nu} \bar{\chi}(h_{\delta})}
= h^{w_1^{-1}\bar{\chi}} < \mathcal M^{l+1},
\end{equation}
for all $w_1\in W$.
\end{Lem}

\subsection{Bounding Factor (ii) of the Product}

  We now turn our focus to factor (ii), so we consider the infinite product \begin{equation}\prod_{\nu\in\mathcal V} \left ( \frac{1}{q_\nu} \right )^{-(m_{\nu} \bar{\chi}(w_1H) + \frac{m_{\nu} (H,H)}{2}\bar{\chi}(h_{\delta}) )}.\end{equation}
However, recall that in the definition of the automorphism $\eta^{mD}$ we specified $m=(m_{\nu} )_{\nu\in\mathcal V}$ where $m_{\nu} \in\mathbb Z$, $m_{\nu} =0$ for all but a finite number of $\nu$. Now we further assume that $\sum_\nu (\log q_\nu) \, m_\nu>0$.  Let $S=\{\nu\in\mathcal V \ | \ m_{\nu} \neq0\}$. Then our infinite product above reduces to the finite product:
\begin{eqnarray}
& & \prod_{\nu\in S} \left ( \frac{1}{q_\nu} \right )^{-(m_{\nu} \bar{\chi}(w_1H) + \frac{m_{\nu} (H,H)}{2}\bar{\chi}(h_{\delta}) )} \nonumber \\
& = & \ \prod_{\nu\in S} \left (\frac{1}{q_\nu} \right )^{-m_{\nu} \bar{\chi}(w_1H)} \ \prod_{\nu\in S} \left (\frac{1}{q_\nu} \right )^{-\frac{m_{\nu} (H,H)}{2}\bar{\chi}(h_{\delta})}. \label{eqn-sec}
\end{eqnarray}
  We will treat each product in (\ref{eqn-sec}) separately, beginning with the factor involving $\bar{\chi}(w_1H)$. If we write $\| H \| = (H,H)^{1/2}$ for $H \in \frak h$, we can prove the following lemma:

  \begin{Lem} \label{lem-N1}
There exists $\mathcal N_1 >0$ such that  for any $H\in\frak h_{\mathbb Z}$ and $w_1\in W$,
\begin{equation} \label{eqn-des}
\prod_{\nu\in S} \left ( \frac{1}{q_\nu} \right )^{-m_{\nu} \bar{\chi}(w_1H)} < \mathcal N_1^{\|H\|}.
\end{equation}
\end{Lem}

\begin{proof} Let $B_0=\{H\in\frak h \ | \ \| H \|=1 \}$. Then this is clearly a compact set of $\frak h$, where the topology on $\frak h$ is the metric topology induced by the norm $\| \cdot \|$.
The linear map $\bar \chi :B_0\rightarrow \mathbb R$ is a continuous map, and hence is bounded. Since $W$ preserves $B_0$, we can find a real number $N_1$ such that $$\bar{\chi}(w_1H)<N_1,$$ for all $w_1\in W$ and $H\in B_0$.
Let $H$ be any element of $\frak h_{\mathbb Z}$. Then $\bar{\chi}(w_1H)=\bar{\chi}(w_1 \frac{H}{\| H \|}) \, \| H\|$, and we note that $ \frac{H}{\| H \|}\in B_0$. So we see that for all $w_1\in W$ and $H\in\frak h_{\mathbb Z}$,
$$ \bar{\chi}(w_1H)<N_1\|H\|.$$

 Now we have
 \[ \prod_{\nu\in S} \left ( \frac{1}{q_\nu} \right )^{-m_{\nu} \bar{\chi}(w_1H)} = \exp\left ( \sum_{\nu \in S} (\log q_\nu) \, m_\nu  \bar{\chi}(w_1H) \right ) < \exp \left ( \sum_{\nu \in S} (\log q_\nu) \, m_\nu  N_1\|H\|\right ) .\]
We set $\mathcal N_1 = \exp \left ( \sum_{\nu \in S} (\log q_\nu) \, m_\nu  N_1 \right )$, and obtain the desired inequality \eqref{eqn-des}.
\end{proof}

\medskip

  Next we consider the other factor from (\ref{eqn-sec}), specifically
$$\prod_{\nu\in S} \left ( \frac{1}{q_\nu} \right )^{-\frac{m_{\nu} (H,H)}{2}\bar{\chi}(h_{\delta})}.$$
Recalling that $\chi (h_{\alpha_i}) <-2$ for each $i =1, 2, \dots, l+1$, we see $\bar{\chi}(h_{\delta})<0$. We have
\[ \prod_{\nu\in S} \left ( \frac{1}{q_\nu} \right )^{-\frac{m_{\nu} (H,H)}{2}\bar{\chi}(h_{\delta})} = \exp \left ( \frac 1 2 \sum_{\nu \in S} (\log q_\nu) \, m_\nu \|H\|^2 \bar{\chi}(h_{\delta}) \right ) . \]
We write $\mathcal N_2 = \exp \left ( \frac 1 2 \sum_{\nu \in S} (\log q_\nu) \, m_\nu  \bar{\chi}(h_{\delta}) \right )$. Since we assumed $\sum_{\nu \in S} (\log q_\nu) \, m_\nu >0$, we get $0<\mathcal N_2 <1$. Thus we obtain:
\begin{Lem} \label{lem-N2}
If we write
\begin{equation} \prod_{\nu\in S} \left ( \frac{1}{q_\nu} \right )^{-\frac{m_{\nu} (H,H)}{2}\bar{\chi}(h_{\delta})}=\mathcal N_2^{\ \|H\|^2},
\end{equation} then $0<\mathcal N_2 <1$.
\end{Lem}

%%%%%%%%%%%%%%%%
%%%%%%%%%%%%%%%%
\subsection{Bounding Factor (iii) of the Product} \label{subsec-5.5}

  We finally consider factor (iii).  In particular, we find a bound for the infinite product \begin{equation} \label{eqn-third} \prod_{\nu\in\mathcal V} \left (\frac{1}{q_\nu} \right )^{(h_{\nu} ,H)\bar{\chi}(h_{\delta})}.\end{equation}

Since we set $h_{\nu} +e_{\nu} h_{\delta}= \sum_{i=1}^{l+1} \mathrm{ord}_{\nu}(s_{i,\nu}) h_{\alpha_i}\,$, and $(h_{\delta},H)=0$ for all $H\in\frak h_{\mathbb Z}$, we can make the following changes
\[
(h_{\nu} ,H)=(\ \sum_{i=1}^{l+1} \mathrm{ord}_{\nu}(s_{i,\nu}) h_{\alpha_i} - e_{\nu} h_{\delta} \ ,\ H \ )= \sum_{i=1}^{l+1} \mathrm{ord}_{\nu}(s_{i,\nu}) (h_{\alpha_i}, H).
\]
As a result, the infinite product in (\ref{eqn-third}) becomes
\[
\prod_{\nu\in\mathcal V}  \prod_{i=1}^{l+1} \left ( \frac{1}{q_\nu} \right )^{\mathrm{ord}_{\nu}(s_{i,\nu}) \, (h_{\alpha_i}, H) \bar{\chi}(h_{\delta})}  =  \prod_{i=1}^{l+1} \prod_{\nu\in\mathcal V} \left ( \frac{1}{q_\nu} \right )^{\mathrm{ord}_{\nu}(s_{i,\nu}) \, (h_{\alpha_i}, H) \bar{\chi}(h_{\delta})}= \prod_{i=1}^{l+1} \, |s_i|^{\, (h_{\alpha_i}, H) \bar{\chi}(h_{\delta})}.
\]

  As before, we let $B_0=\{ H\in\frak h \ | \ \|H\|=1\}$. For each $i=1,\dots,l+1$, we consider the continuous maps determined by the bilinear form $$(h_{\alpha_i}, \cdot):B_0\rightarrow\mathbb R.$$
Since we are only considering $l+1$ different images of the compact set $B_0$ in $\mathbb R$, we can find numbers $n$ and $N$ such that $n<(h_{\alpha_i},H)<N$ for all $H\in B_0$. Using the same argument as in Lemma \ref{lem-N1}, we conclude that for any $H\in\frak h_{\mathbb Z}$ and $i=1, 2, \dots, l+1$, we have
$$n \| H \| < (h_{\alpha_i},H)<N\|H\|,$$
or equivalently,
$$N\|H\|\bar{\chi}(h_{\delta})< (h_{\alpha_i},H)\bar{\chi}(h_{\delta}) < n\|H\|\bar{\chi}(h_{\delta}),$$
since $\bar{\chi}(h_{\delta})<0$.

  Since we assumed that $h$ varies in the compact set $ \mathbf C \subset\hat H_{\mathbb A}$, we know that $r<|s_i|<R$ for $i=1,\dots,l+1$. It is now straightforward to prove the following lemma.

\begin{Lem} \label{lem-third}
There exists a constant $\mathcal N_3$ such that for any $H\in\frak h_{\mathbb Z}$ we have \begin{equation}
\prod_{\nu\in\mathcal V} \left (\frac{1}{q_\nu} \right )^{(h_{\nu} ,H)\bar{\chi}(h_{\delta})}=\prod_{i=1}^{l+1} \, |s_i|^{\, (h_{\alpha_i}, H) \bar{\chi}(h_{\delta})} < \mathcal N_3^{(l+1)\|H\|}.
\end{equation}
\end{Lem}

\medskip

 Now we collect the results of subsections \ref{subsec-5.2} through \ref{subsec-5.5} and summarize them in a proposition.  Recall that Proposition \ref{prop-three-factors} proved that  $(h \eta^{mD})^{w\bar{\chi}}$ is the product of the factors (i), (ii) and (iii).
By combining this with the results of Lemmas \ref{lem-first}, \ref{lem-N1}, \ref{lem-N2} and \ref{lem-third},  we obtain:

\begin{Lem} \label{lem-upper}
For $w\in\hat{W}$, we write $w^{-1}=w_1 \, T_{H}$, where $w_1 \in W$ and $H \in\frak h_{\mathbb Z}$. We have the following upper bound for the factor $(h \eta^{mD})^{w\bar{\chi}}$:
$$
(h \eta^{mD})^{w\bar{\chi}}  \, < \, \mathcal M^{l+1} \  \mathcal N_1^{\|H\|}  \ \mathcal N_3^{(l+1)\|H\|} \ \mathcal N_2^{\|H\|^2}.
$$
\end{Lem}

  In subsection \ref{subsec-5.7}, we wish to bound the constant term by a theta series. To this end, we rewrite
   \[ \mathcal N_1^{\|H\|}  \ \mathcal N_3^{(l+1)\|H\|}= e^{\log(\mathcal N_1 \, \mathcal N_3^{(l+1)})\|H\|} \quad \text{ and } \quad \mathcal N_2^{\|H\|^2} = e^{\log \mathcal N_2 \| H\|^2}. \]
We set $\sigma_1=\log(\mathcal N_1 \, \mathcal N_3^{(l+1)})$ and $\sigma_2= - \log \mathcal N_2$. Since $\mathcal N_2 <1$, we have $\sigma_2>0$.  We state our final result in the following proposition:

\begin{Prop} \label{prop-h-bound}
For $w\in\hat{W}$, suppose $w^{-1}=w_1 \, T_{H}$ where $w_1 \in W$ and $H \in\frak h_{\mathbb Z}$. There exist constants $\mathcal M$, $\sigma_1$, and $\sigma_2$ that do not depend on $w\in\hat{W}$, such that $\sigma_2>0 $ and
\begin{equation}(h \eta^{mD})^{w\bar{\chi}}  \, < \, \mathcal M^{l+1} \ e^{\sigma_1 \|H\| - \sigma_2\|H\|^2}.
\end{equation}
\end{Prop}

\medskip

\subsection{Bounding the Zeta Functions} \label{subsec-5.6}

   In this subsection, we focus on bounding the $c(\chi,w)$ factor in the constant term $E_{\chi}^{\#}$. Recall the definition:
$$c(\chi,w)= q^{\ell(w)(1-g)} \prod_{a\in\hat{\Delta}_{W,+}\cap w^{-1} \hat{\Delta}_{W,-}}\ \ \frac{\zeta_F(-(\chi+\rho)(h_{a}))}{\zeta_F(-(\chi+\rho)(h_a)+1)}.$$
Standard techniques involving zeta functions establish the following lemma:
\begin{Lem} \label{lem-zeta}
Let $s\in\mathbb C$ and $\varepsilon>0$. Then for all $s$ such that $\mathrm{Re}(s) \geq 1+\varepsilon$ we have $$\left | \ q^{(1-g)}\  \frac{\zeta_F(s)}{\zeta_F(s+1)} \ \right | < M_{\varepsilon},$$
where $M_{\varepsilon}$ is a positive constant.  In particular, we can take $M_{\varepsilon}=q^{(1-g)}(\zeta_F(1+\varepsilon))^2$.
\end{Lem}

\begin{Cor} \label{cor-pos-const}
There exists a positive constant $M_{\varepsilon}$ such that for each $w\in\hat{W}$ we have $$c(\chi,w)  < M_{\varepsilon}^{\ell(w)}.$$
\end{Cor}

\begin{proof}
For any $a\in\hat{\Delta}_{W,+}$, we have $h_a = \sum_{i=1}^{l+1} k_i h_{\alpha_i}$, where $k_i \in \mathbb Z_{\ge 0}$ and at least one $k_i \neq 0$. Since $\chi(h_{\alpha_i}) <-2$ for $i=1, \dots , l+1$, there exists $\varepsilon >0$ such that  $$-(\chi+\rho)(h_a)\geq1+\varepsilon$$ for any $a\in\hat{\Delta}_{W,+}$. It is known that $\#(\hat{\Delta}_{W,+}\cap w^{-1} \hat{\Delta}_{W,-})= \ell(w)$ (\cite{Kumar}).
 Now the corollary follows from Lemma \ref{lem-zeta}.
 \end{proof}

  Each $w\in\hat{W}$ can be expressed as $w=w_1 T_H$ for $w_1\in W$ and some $H\in\frak h_{\mathbb Z}$, and we have $\ell(w)\le \ell(w_1) +\ell(T_H)$. However, since the classical Weyl group is finite, $\ell(w_1)$ can only be as large as the length of the longest element of $W$. Thus for each $w\in\hat{W}$, $$M_{\varepsilon}^{\ell(w)}\le M_{\varepsilon}^{\ell(w_1)}\ M_{\varepsilon}^{\ell(T_H)}\leq\overline{M_{\varepsilon}}\ M_{\varepsilon}^{\ell(T_H)}.$$
In (8.17) of \cite{R}, we see that there exists a postive constant $\sigma_3$ such that $\ell(T_H)\leq\sigma_3 \|H\|$.  In light of these observations, we have the following proposition:

\begin{Prop} \label{prop-c-bound}
There exists a positive constants $\overline{M_{\varepsilon}}$, $M_{\varepsilon}$, and $\sigma_3$ such that $$c(\chi,w) < \overline{M_{\varepsilon}}\ M_{\varepsilon}^{\sigma_3 \|H\|},$$
for each $w=w_1T_H\in\hat{W}$.
\end{Prop}

\subsection{Bounding the Constant Term by a Theta Series} \label{subsec-5.7}

In this final subsection, we prove that if $h$ varies in a compact set of $\hat H_{\mathbb A}$, then the sum
\begin{equation} \sum_{w\in\hat{W}} \, (h\eta^{mD})^{w(\chi+\rho)-\rho} \ c(\chi,w) \end{equation}
is bounded above by a theta series.  Hence the constant term of the Eisenstein $E_{\chi}^{\#}$ is absolutely convergent for these compact sets.

By substituting the results of Propositions \ref{prop-h-bound} and \ref{prop-c-bound}, we see
\begin{eqnarray*}
(h\eta^{mD})^{-\rho } \sum_{w\in\hat{W}} \, (h\eta^{mD})^{w(\chi+\rho)} \, c(\chi,w) &<& (h\eta^{mD})^{-\rho } \sum_{H\in\frak h_{\mathbb Z}} \, \mathcal M^{l+1} \ e^{\sigma_1 \|H\| - \sigma_2\|H\|^2} \ \overline{M_{\varepsilon}}\ M_{\varepsilon}^{\sigma_3 \|H\|} \\
&=&(h\eta^{mD})^{-\rho } \#(W) \mathcal M^{l+1} \ \overline{M_{\varepsilon}} \ \ \sum_{H\in\frak h_{\mathbb Z}} e^ { (\log(M_{\varepsilon})\sigma_3 +\sigma_1)\|H\| -\sigma_2 \|H\|^2},
\end{eqnarray*}
where $\#(W)$ is the cardinality of the finite Weyl group $W$.

It is essential to note that $\sigma_2>0$, so this theta series converges.  As a result of this computation, we have proven the following theorem.

\begin{Thm} \label{thm-conv}  Let $\chi\in\hat{\frak h}^*$ such that $\chi(h_{\alpha_i})<-2$ for $i=1,\dots,l+1$,  and let $$c(\chi,w)=q^{l(w)(1-g)} \prod_{a\in\hat{\Delta}_{W,+}\cap w^{-1} \hat{\Delta}_{W,-}}\ \ \frac{\zeta_F(-(\chi+\rho)(h_{a}))}{\zeta_F(-(\chi+\rho)(h_a)+1)}.$$ Assume that $\sum_\nu (\log q_\nu) \, m_\nu>0$ for $m=(m_\nu)$. Then the infinite series
$$\sum_{w\in\hat{W}} \, (h\eta^{mD})^{w(\chi+\rho)-\rho} \ c(\chi,w)$$
 converges absolutely and uniformly for $h$ varying in any compact set of $\hat H_{\mathbb A}$.
\end{Thm}

\begin{Rmk}
Combining the above theorem with (\ref{eqn-const}), Definition \ref{def-constant-term} and Theorem \ref{thm-const},  we have proved the identities
\[ E^{\#}_\chi  (g\eta^{mD}) =
 \int_{\hat{U}_{\mathbb A}/(\hat{U}_{\mathbb A}\cap\hat{\Gamma}_F)}  E_{\chi}(g\eta^{mD} u) \ du = \sum_{w\in\hat{W}} \, (h\eta^{mD})^{w(\chi+\rho)-\rho} \ c(\chi,w) ,\] and established convergence of the constant term $E_\chi^{\#}$.
\end{Rmk}

\vskip 1 cm

\section{Convergence of the Eisenstein Series}{}

In this section, we will use the results of the previous sections to prove the convergence of the Eisenstein series $E_{\chi}$.  Recall that in Section 4, we skipped the proof of the measurability of $E_{\chi}$ with respect to $du$. In the next subsection, we prove this fact.

\subsection{Measurability of the Eisenstein Series}

  The constant term of the Eisenstein series $E_{\chi}$ is defined to be the following integral:
$$E_{\chi}^{\#}(g\eta^{mD})=\int_{\hat{U}_{\mathbb A}/(\hat{U}_{\mathbb A}\cap\hat{\Gamma}_F)} E_{\chi}(g\eta^{mD} u) \ du.$$ It is the purpose of this subsection to prove that the map $u\mapsto E_{\chi}(g\eta^{mD} u)$ is a $du$-measurable function. Observe however, that since $\Phi_{\chi}$ is left invariant by $\hat{\mathbb K}$, and $\eta^{mD}$ normalizes $\hat{U}_{\mathbb A}$, it is enough to show that for a fixed $h\in\hat H_{\mathbb A}$ the map $u\mapsto E_{\chi}(h\eta^{mD} u)$ is a $du$-measurable function.

  As in Section 4, we express the Eisenstein series as the sum $$E_{\chi}(h\eta^{mD} u)=\sum_{w\in\hat{W}} \, E_{\chi, w}(h\eta^{mD} u).$$ In turn, each $E_{\chi, w}(h\eta^{mD} u)$ is also a sum of particular values for $\Phi_{\chi}$.  Recall that we set $\hat{\Gamma}_F(w)= \hat{\Gamma}_F\cap(\hat{B}_F \, w \, \hat{B}_F)$ and defined
\begin{equation}
E_{\chi, w}(h\eta^{mD} u) =  \sum_{\gamma\in\hat{\Gamma}_F(w)/\hat{\Gamma}_F(w)\cap\hat{B}_F} \Phi_{\chi}(h \eta^{mD} u \gamma),
\end{equation}
and we saw in Section 4 that we may take the coset representatives $\gamma$ above to be of the form $\{b w \}$ where $b\in\hat{U}_{w,F}$. Now with respect to our coherently ordered basis, $\hat{U}_{w,F}$ is a finite-dimensional space for each $w\in\hat{W}$, and as such we can choose our coset representatives of $\hat{\Gamma}_F(w)/(\hat{\Gamma}_F(w)\cap\hat{B}_F)$ to come from this finite-dimensional space.

   In Appendix A, we construct the measure $du$ by expressing $\hat{U}_{\mathbb A}/(\hat{U}_{\mathbb A}\cap\hat{\Gamma}_F)$ as a projective limit of compact spaces. Part of this construction is important for this discussion, so we briefly state some definitions from that section.  For the coherently ordered basis $\mathcal B=\{v_{\lambda}, \, v_1, \, \dots\}$, we let $V^{\lambda}_{F_\nu,s}$ denote the $F_\nu$-span of the vectors $\{v_{\lambda}, \ v_1, \dots, \ v_s\}$. By setting $\hat{U}_{\nu,s}=\{u\in\hat{U}_{F_\nu} \ | \ u|_{V^{\lambda}_{F_\nu,s}} \equiv id \}$, we may define the restricted direct product
$$\hat{U}_{\mathbb A,s}={\prod}'  \hat{U}_{\nu,s}\text{ with respect to the subgroup }\hat{U}_{\nu,s}\cap\hat{K}_\nu,$$ and let $$\hat{U}_{\mathbb A}^{(s)} =\hat{U}_{\mathbb A}/\hat{U}_{\mathbb A,s}.$$

  As proved in the appendix, $\hat{U}_{\mathbb A}=\displaystyle{\varprojlim_s}\, \hat{U}_{\mathbb A}^{(s)}$.  Since we may choose our coset representatives $\gamma\in\hat{\Gamma}_F(w)/(\hat{\Gamma}_F(w)\cap\hat{B}_F)$ so that they come from a finite-dimensional space, there exists an $s$ large enough so that for any $u\in\hat{U}_{\mathbb A,s}$ and $\gamma$ as above, we have $u \gamma=\gamma u'$ for some $u'\in \hat{U}_{\mathbb A}$. Then we observe that  \begin{equation} \label{eqn-phi-gamma} \Phi_{\chi}(h\eta^{mD} u \gamma)=\Phi_{\chi}(h\eta^{mD} \gamma u') = \Phi_{\chi}(h\eta^{mD} \gamma). \end{equation}

  Using $\Phi_{\chi}$ and $\gamma\in\hat{\Gamma}_F(w)/(\hat{\Gamma}_F(w)\cap\hat{B}_F)$, we define the function $\psi_{\gamma}$ from $\hat{U}_{\mathbb A}$ to $\mathbb R_{>0}$ by setting
\begin{equation}
\psi_{\gamma}(u)=\Phi_{\chi}(h\eta^{mD} u\gamma).
\end{equation}
By the observation (\ref{eqn-phi-gamma}) in the previous paragraph, $\psi_{\gamma}$ defines a function on the finite dimensional space $\hat{U}_{\mathbb A}/\hat{U}_{\mathbb A,s}$ for $s$ large enough. We will see in Appendix A that we may consider $\hat{U}_{\mathbb A}/\hat{U}_{\mathbb A,s}$ as embedded into the group of upper triangular $(s+1)\times (s+1)$ block matrices with entries from $\mathbb A$. Most importantly, for any $\gamma\in\hat{\Gamma}_F(w)/(\hat{\Gamma}_F(w)\cap\hat{B}_F)$ the function $\psi_{\gamma}$ can be written as a composition of continuous maps and hence measurable on the space $\hat{U}_{\mathbb A}/\hat{U}_{\mathbb A,s}$. Then it follows from the definition of the measure $du$ in Appendix A that the function $\psi_\gamma$ is a measurable function on $\hat{U}_{\mathbb A}/(\hat{U}_{\mathbb A}\cap\hat{\Gamma}_F)$.

   For a fixed $h\in\hat H_{\mathbb A}$ and $w\in\hat{W}$, we define the function $\psi_w$ from $\hat{U}_{\mathbb A}/(\hat{U}_{\mathbb A}\cap\hat{\Gamma}_F)$ to $\mathbb R_{>0}$ by sending $$u\longmapsto E_{\chi, w}(h\eta^{mD} u)=\sum_{\gamma\in\hat{\Gamma}_F(w)/(\hat{\Gamma}_F(w)\cap\hat{B}_F)} \, \psi_{\gamma}(u).$$  The function $\psi_{\gamma}$ is a positive, measurable function for every $\gamma$ above, and so $\psi_w(u)$ is also a measurable function on $\hat{U}_{\mathbb A}/(\hat{\Gamma}_F\cap \hat{U}_{\mathbb A})$. To see this, we view \begin{equation}\psi_w \ = \text{sup} \{\text{finite sums of }\psi_{\gamma}\}, \end{equation} and note that this is measurable.

  Likewise, for a fixed $h\in\hat H_{\mathbb A}$ we can consider $E_{\chi}$ as a function from the quotient $\hat{U}_{\mathbb A}/(\hat{U}_{\mathbb A}\cap\hat{\Gamma}_F)$ to the positive real numbers by sending $u\mapsto E_{\chi}(h\eta^{mD} u)$. Since $E_{\chi}$ can be expressed as the sum over the affine Weyl group of the positive, measurable functions $\psi_w$, we conclude:

\begin{Lem}
For any $h\in\hat H_{\mathbb A}$, the function $E_{\chi}$ is a $du$-measurable function into the positive real numbers.
\end{Lem}

\subsection{Convergence of the Series} \label{subsec-almost}

  The Eisenstein series $E_{\chi}$ on $\hat G_{\mathbb A}\eta^{mD}$ can be considered as the function from $\hat H_{\mathbb A}\times\hat{U}_{\mathbb A}/(\hat{U}_{\mathbb A}\cap\hat{\Gamma}_F)$ to $\mathbb R_{>0}$ defined by
\begin{equation} (h,  u) \ \mapsto \ E_{\chi}(h\,\eta^{mD} \,u).
\end{equation}
Moreover, Theorem \ref{thm-conv} proved that the constant term
 \[ E^{\#}_\chi  (g\eta^{mD}) =
 \int_{\hat{U}_{\mathbb A}/(\hat{U}_{\mathbb A}\cap\hat{\Gamma}_F)}  E_{\chi}(g\eta^{mD} u) \ du = \sum_{w\in\hat{W}} \, (h\eta^{mD})^{w(\chi+\rho)-\rho} \ c(\chi,w) \]
 is absolutely convergent for $h$ varying in compact sets of $\hat H_{\mathbb A}$. However, this also tells us that the Eisenstein series $E_{\chi}$ is integrable with respect to $du$ for $h$ varying in any compact subset of $\hat H_{\mathbb A}$.  Hence the series $E_\chi$ is convergent almost everywhere on $\hat H_{\mathbb A}\times\hat{U}_{\mathbb A}/(\hat{U}_{\mathbb A}\cap\hat{\Gamma}_F)$, since $\hat H_{\mathbb A}$ is locally compact. Moreover, we can prove the following proposition:

\begin{Prop} Let $\chi\in\hat{\frak h}^*$ be a real character such that $\chi(h_{\alpha_i})<-2$ for $i=1,\dots,l+1$, and let $m=(m_\nu)_{\nu\in\mathcal V}$ be a tuple  such that $m_\nu \in \mathbb Z_{\ge 0}$ and $0<\sum_\nu m_\nu <\infty$.  Then the series $E_{\chi}(h\eta^{mD} u)$ (absolutely) converges to a positive real number for all $(h,u)\in\hat H_{\mathbb A}\times\hat{U}_{\mathbb A}/(\hat{U}_{\mathbb A}\cap\hat{\Gamma}_F)$.
\end{Prop}
\begin{proof}
The remarks in the previous paragraph tell us that the series $E_{\chi}(h\eta^{mD} u)$ converges for all $(h,u)\in\hat H_{\mathbb A}\times\hat{U}_{\mathbb A}/(\hat{U}_{\mathbb A}\cap\hat{\Gamma}_F)$ off a set of measure zero. Assume that $E_{\chi}(h\eta^{mD} u)= \infty$ for some $(h,u)$.
We claim that there exists a subset $U' \subset \hat{U}_{\mathbb A}$ of positive measure such that $  h \eta^{mD} u' (h \eta^{mD})^{-1} \in \hat{U}_{\mathbb A} \cap \hat{\mathbb K}$ for all $u' \in U'$. If the claim is true, we will have
 \[ E_{\chi}(h\eta^{mD} u' u) =  E_\chi(h \eta^{mD} u) = \infty \] for all $u' \in U'$. Since the set $ U' u$ has positive measure, it is a contradiction.

Now we prove the claim. We write $u' =(u'_\nu)$ and
\[ u'_\nu = \prod_{\alpha \in \Delta_+} \chi_\alpha (\sigma_\alpha) \prod_{i=1}^l h_{\alpha_i}(\sigma_i) \prod_{\alpha \in \Delta_-} \chi_\alpha (\sigma'_\alpha), \] where $\sigma_\alpha \in F_\nu[[X]]$, $\sigma'_\alpha \in X F_\nu[[X]]$ and $\sigma_i \in F_\nu[[X]]$ with $\sigma_i \equiv 1$ (mod $X$). It follows from direct computation that if $\sigma_\alpha = a_0 + a_1 X + a_2 X^2 + \cdots$,\  then  $\eta^{m_\nu D} \chi_\alpha(\sigma_\alpha) \eta^{-m_\nu D}= \chi_\alpha(\tau_\alpha)$, where $\tau_\alpha = a_0 + a_1 \pi_\nu^{m_\nu} X + a_2 \pi_\nu^{2m_\nu}X^2 + \cdots$. We also write $h=(h_\nu)_{\nu\in\mathcal V}$ and $h_\nu=\prod_{i=1}^{l+1} h_{\alpha_i}(s_{i,\nu})$. Then we obtain $h_{\alpha_i}(s_{i,\nu}) \chi_\alpha (\sigma_\alpha) h_{\alpha_i}(s_{i,\nu})^{-1} = \chi_\alpha(s_{i,\nu}^{\alpha(h_{\alpha_i})} \sigma_\alpha)$ by (13.10) of \cite{LG1}. Recall that $m_\nu \ge 0$ for all $\nu$ and $\mathrm{ord}_\nu(s_{i,\nu})=0$ for almost all $\nu$.  Because of these conditions on $m_\nu$ and $\mathrm{ord}_\nu(s_{i,\nu})$, it is straightforward to construct the set $U'$ of the claim.

\end{proof}

  Having established this important proposition, the following theorem is a simple consequence of the dominated convergence theorem and the fact that $E_{\mathrm{Re}(\chi)}$ dominates $E_{\chi}$ for any complex character $\chi$.

\begin{Thm} \label{thm-main-conv}
For a complex-valued $\chi\in\hat{\frak h}^*$ such that $\mathrm{Re}(\chi)(h_{\alpha_i})<-2$ for $i=1,\dots, l+1$, and for a tuple  $m=(m_\nu)_{\nu\in\mathcal V}$ such that $m_\nu \in \mathbb Z_{\ge 0}$ and $0<\sum_\nu m_\nu <\infty$, the infinite series
$$E_{\chi}(h\eta^{mD} u)=\sum_{\gamma \in \hat{\Gamma}_F/\hat{\Gamma}_F\cap\hat{B}_F} \ \Phi_{\chi}(h\eta^{mD} u \gamma)$$
absolutely converges for all $(h,u)\in\hat H_{\mathbb A}\times\hat{U}_{\mathbb A}/(\hat{U}_{\mathbb A}\cap\hat{\Gamma}_F)$.
\end{Thm}

\vskip 1 cm

\section{Functional Equations for the Constant Term}

  In this section, we will establish meromorphic continuation of the constant term of the Eisenstein series and prove their functional equations. We will begin by stating some results for the zeta function of the global function field $F$.

\subsection{Background on the Zeta Function of a Function Field}

  We refer the reader to \cite{Ros} for specifics on the definition of the zeta function associated to a global function field $F$. The following result from \cite{Ros} describes the functional equation for $\zeta_F(s)$.

\begin{Thm} Let $F$ be a global function field in one variable over a finite constant field $\mathbb{F}_q$. Suppose $F$ is of genus $g$. Then there exists a polynomial $L_F(u)\in \mathbb Z[u]$ of degree $2g$ such that
\begin{equation} \label{eqn-zeta} \zeta_F(s)= \frac{L_F(q^{-s})}{(1-q^{-s})(1-q^{1-s})},\end{equation}
for $\mathrm{Re}(s)>1$.  Moreover, (\ref{eqn-zeta}) provides an analytic continuation of $\zeta_F$ to the complex plane.  If we set $\xi_F(s)=q^{(g-1)s}\zeta_F(s)$, then we have the functional equation
\begin{equation} \xi_F(s)=\xi_F(1-s).\end{equation}
\end{Thm}

  Using the functional equation for $\xi_F(s)$, we prove the following lemma:

\begin{Lem}  \label{lem-product-zeta} Let $\zeta_F(s)$ be the zeta function associated to $F$, and $\xi_F(s)=q^{(g-1)s}\zeta_F(s)$ its completed form. Then the following identities hold:
\[ \frac{\zeta_F(s)}{\zeta_F(1-s)}=q^{(2s-1)(1-g)}, \quad \frac{\zeta_F(-s)}{\zeta_F(1+s)}=q^{(-2s-1)(1-g)}, \quad \text{and} \quad \frac{\zeta_F(s)}{\zeta_F(1-s)} \frac{\zeta_F(-s)}{\zeta_F(1+s)}=q^{-2(1-g)}. \]
\end{Lem}
\begin{proof}
These computations follow by replacing $\zeta_F(s)$ with $q^{s(1-g)}\xi_F(s)$ and using the functional equation of $\xi_F(s)$. For example, we consider calculation for the second identity:
\begin{eqnarray*}
\frac{\zeta_F(-s)}{\zeta_F(1+s)} &=& \frac{\zeta_F(-s)}{\zeta_F(1-(-s))}= \frac{q^{-s(1-g)} \ \xi_F(-s)}{q^{(1+s)(1-g)} \ \xi_F(1-(-s))}\\
&=& \frac{q^{-s(1-g)}}{q^{(1+s)(1-g)}} = q^{(-2s-1)(1-g)}.
\end{eqnarray*}
The first identity follows in the same manner, and third one is simply the product of the other two identities.
\end{proof}

\subsection{Meromorphic Continuation}
We first note that the function $c(\chi, w)$ is well-defined as meromorphic function for any $\chi \in (\hat {\mathfrak h}^{e}_{\mathbb C})^*$. We assume that
\begin{equation} \label{eqn-cond-1} \mathrm{Re} (\chi +\rho)(h_\delta)  <0.\end{equation} Since $\rho(h_\delta) = h^\vee$, the dual Coxeter number, this assumption is equivalent to \begin{equation} \label{eqn-cond-2}  \mathrm{Re}\  \chi (h_\delta)< - h^\vee. \end{equation}
We fix $\varepsilon >0$, and set
\[ C_\varepsilon = \{ \chi \in (\hat {\mathfrak h}^{e}_{\mathbb C})^* \,|\,  \mathrm{Re} (\chi +\rho)(h_\delta)  < - \varepsilon \}. \]
Recall that if $a \in \hat \Delta_{W,+}$, we have $h_a = h_\alpha + m h_\delta$ for some $\alpha \in \Delta$ and $m \in \mathbb Z_{\ge 0}$.
Thus there exists a finite set $\Xi_{\varepsilon} \subset \hat \Delta_{W,+}$ such that $\mathrm{Re} \, (\chi+\rho) (h_a) < -1$ for all $a \in \hat \Delta_{W,+} \setminus \Xi_{\varepsilon} $ and for any $\chi \in C_\varepsilon$.

Let $\mathcal F \subset C_\varepsilon$ be the set of all $\chi \in C_\varepsilon$ such that the function $\frac {\zeta_F(- (\chi+ \rho)(h_a))}{\zeta_F(- (\chi+ \rho)(h_a) +1)}$ has a pole  for some $a \in \Xi_\varepsilon$. Then $\mathcal F$ is contained in the union of a countable, locally finite family of hyperplanes.

\begin{Lem} \label{lem-var}
Suppose that $\mathbf B$ is a bounded open subset of $(\hat {\mathfrak h}^{e}_{\mathbb C})^*$ whose closure is contained in the set $C_\varepsilon$. Assume that  the function $\frac {\zeta_F(- (\chi+ \rho)(h_a))}{\zeta_F(- (\chi+ \rho)(h_a) +1)}$ has no pole for any $\chi \in \mathbf B$ and for any $a \in \Xi_\varepsilon$. Then there are positive constants $M'_\varepsilon$ and $M_\varepsilon$ such that  for each $w\in\hat{W}$ we have $$|c(\chi,w)|  < M'_\varepsilon M_{\varepsilon}^{\ell(w)}.$$
\end{Lem}

\begin{proof}
Since $\Xi_\varepsilon$ is a finite set and independent of $w$, this lemma is proved by a slight modification of the proof of Corollary \ref{cor-pos-const}.
\end{proof}

\begin{Thm} \label{thm-mero-const}
The constant term $E_{\chi}^{\#}(g\eta^{mD})=\sum_{w\in\hat{W}} \ (h\eta^{mD})^{w(\chi+\rho)-\rho} \ c(\chi, w)$ is a meromorphic function for $\chi \in (\hat {\mathfrak h}^{e}_{\mathbb C})^*$ satisfying
 $\mathrm{Re} (\chi +\rho)(h_\delta)  <0$, or equivalently $\mathrm{Re}\,  \chi (h_\delta)< - h^\vee$.
 \end{Thm}

\begin{proof}
Assume that $\mathbf B$ is a bounded open subset of $(\hat {\mathfrak h}^{e}_{\mathbb C})^*$ whose closure is contained in the set $C_\varepsilon \setminus \mathcal F$, where the set $\mathcal F$ is defined above. Consider a
compact subset $\mathbf C$ of $\hat H_{\mathbb A}$. If we replace Corollary \ref{cor-pos-const} with Lemma \ref{lem-var}, all the other arguments in the proof of Theorem \ref{thm-conv} remain valid to prove that
the infinite sum $\sum_{w\in\hat{W}} \ (h\eta^{mD})^{w(\chi+\rho)-\rho} \ c(\chi, w)$ converges uniformly and absolutely as $\chi$ varies over $\mathbf B$ and $h$ varies over $\mathbf C$. Since $\varepsilon$ was arbitrary, the theorem follows.
  \end{proof}

\subsection{Proving Functional Equations}

  We wish to prove functional equations for the constant term of the Eisenstein series $E_{\chi}$ on $\hat G_{\mathbb A}$. The following property of the function $c(\chi, w)$ is essential:
\begin{Prop} \label{prop-cocycle}
 With $c(\chi,w)$ defined as above and $w, w'\in\hat{W}$, we have \begin{equation} c(\chi, w w')=  c(w' \circ \chi,w) \, c(\chi, w'),\end{equation} where $w \circ \chi $ is the usual shifted action of $\hat W$ on $\hat{\frak h}^*$, i.e. $w \circ \chi = w(\chi +\rho)-\rho$.
\end{Prop}

  We will prove this proposition at the end of this section. First, we show that this proposition leads to the following functional equation for the constant term of the Eisenstein series.

\begin{Thm}
Assume that $\chi \in (\hat {\mathfrak h}^{e}_{\mathbb C})^*$ satisfies the condition $\mathrm{Re}\,  \chi (h_\delta)< - h^\vee$. Then we have $\mathrm{Re}\,  (w \circ \chi) (h_\delta)< - h^\vee$ and
\begin{equation} E_{\chi}^{\#}(g\eta^{mD})=c(\chi, w) \, E^{\#}_{w \circ\chi}(g\eta^{mD}),
\end{equation}for any $w\in\hat{W}$.
\end{Thm}
\begin{proof} For the first assertion, we only need to consider the simple reflections $w_i \in \hat W$. We have \[ (w_i \circ \chi)(h_\delta) = ( w_i(\chi +\rho)-\rho)(h_\delta) = (\chi - (\chi + \rho)(h_i) \alpha_i)(h_\delta) = \chi (h_\delta) .\] Since we assumed $\mathrm{Re}\, \chi(h_\delta)<-h^\vee$, the the first assertion follows.

Now we fix an arbitrary element $w$ of the affine Weyl group. Then for any $\bar w\in\hat{W}$ we set $w'=\bar w w^{-1}$ so that $\bar w=w' w$. Now by Proposition \ref{prop-cocycle}, we have
\begin{eqnarray*}
E_{\chi}^{\#}(g\eta^{mD}) &=& \sum_{\bar w\in\hat{W}} (h\eta^{mD})^{\bar w \circ {\chi}} \, c(\chi,\bar w) = \sum_{w'\in\hat{W}} (h\eta^{mD})^{(w' w) \circ {\chi}} \ c(\chi,w' w)\\
&=&  \sum_{w'\in\hat{W}} (h\eta^{mD})^{w' \circ (w\circ{\chi})} c(w \circ \chi,w') \ c(\chi, w) \\
&=&  c(\chi, w)\,  \sum_{w'\in\hat{W}} (h\eta^{mD})^{w' \circ (w\circ{\chi})}  c(w \circ \chi,w') =  c(\chi, w) \ E^{\sharp}_{w\circ \chi}(g\eta^{mD}).
\end{eqnarray*}
\end{proof}

   In order to establish this functional equation for $E_{\chi}^{\#}$, it suffices to prove Proposition \ref{prop-cocycle}.
\begin{proof}[Proof of Proposition \ref{prop-cocycle}]

  Recall that we want to show that for any $w,w'\in\hat{W}$, we have $$ c(\chi, w w')= c(w' \circ \chi,w) \, c(\chi, w').$$
 We proceed by an induction argument, and first consider the following cases.

\underline{Case 1:} Suppose that $w_i$ is a simple reflection, $w\in\hat{W}$ satisfies $$\ell(w w_i)=1 +\ell(w),$$ and $w^{-1}=w_{i_r}\dots w_{i_1}$ is a reduced expression. Then since $\hat{\Delta}_{W,+}\cap w^{-1} \hat{\Delta}_{W,-}=\hat{\Delta}_{w^{-1}}$ the discussion in Section \ref{sub-cal-local} shows $$\hat{\Delta}_{W,+}\cap w^{-1} \hat{\Delta}_{W,-}=\{\beta_1, \dots \beta_r\},$$ where $\beta_j = w_{i_r}\dots w_{i_j+1}\alpha_{i_j}$.  If we let $\bar{w}=w w_i$, then we can see $$\hat{\Delta}_{W,+}\cap \bar{w}^{-1}\hat{\Delta}_{W,-}= \{w_i\beta_1,\, \dots, \, w_i\beta_r\} \cup \{ \alpha_i\}.$$
Now we have
\begin{eqnarray*}
c(\chi,\bar{w}) &=& q^{(1+l(w))(1-g)} \ \prod_{\alpha\in\hat{\Delta}_{W,+}\cap \bar{w}^{-1} \hat{\Delta}_{W,-}} \frac{\zeta_F(-(\chi +\rho)(h_{\alpha})}{\zeta_F(-(\chi +\rho)(h_{\alpha}) +1)} \\
&=& q^{(1-g)} \  \frac{\zeta_F(-(\chi +\rho)(h_{\alpha_i})}{\zeta_F(-(\chi +\rho)(h_{\alpha_i}) +1)} \ \ q^{l(w)(1-g)}  \prod_{\alpha\in\hat{\Delta}_{W,+}\cap w^{-1} \hat{\Delta}_{W,-}} \frac{\zeta_F(-(\chi +\rho)(h_{w_i \alpha})}{\zeta_F(-(\chi +\rho)(h_{w_i\alpha}) +1)} \\
&=&   c(\chi, w_i)\, c(w_i \circ \chi, w),
\end{eqnarray*}
since $(\chi + \rho)(h_{w_i \alpha}) = w_i (\chi +\rho)(h_{\alpha}) = (w_i \circ \chi + \rho)(h_\alpha)$.
This is our desired result, so we move on to our second case.

\underline{Case 2:} Suppose $w_i$ is a simple reflection and $w\in\hat{W}$ such that $\ell(w w_i)=\ell(w)-1$. If this is the case, then $w$ has a reduced expression $w=w_{i_1}\dots w_{i_r}$ where $w_{i_r}=w_i$. If we set $w'=w_{i_1}\dots w_{i_{r-1}}$, then $w=w'w_i$ and $w'=ww_i$, and $\ell(w'w_i)=1+\ell(w')$. As such, we can apply the result of Case 1 to this situation with $\tilde{\chi}=w_i \circ \chi$ and we see
$$c(\tilde \chi, w)=c(\tilde \chi, w'w_i)= c(w_i \circ \tilde  \chi, w' ) \, c(\tilde \chi, w_i) =  c( \chi, w w_i ) \, c( \tilde \chi, w_i).$$
So by solving for the factor $ c(\chi, w w_i )$, we see that $$c({\chi}, w w_i )= c(w_i \circ  {\chi},w)\, c(w_i \circ {\chi}, w_i)^{-1}.$$ In order to prove our result for this case, it suffices to show that   \begin{equation} c({\chi}, w_i) \, c(w_i \circ{\chi},w_i) =1.\end{equation} Observe that
\begin{eqnarray*}
c({\chi}, w_i) \, c(w_i \circ{\chi},w_i) &=& q^{2(1-g)}  \, \frac{\zeta_F(- (\chi + \rho)(h_{\alpha_i}))}{\zeta_F(-(\chi + \rho)(h_{\alpha_i})+1)} \ \frac{\zeta_F(-(w_i \circ \chi + \rho)(h_{\alpha_i}))}{\zeta_F(-(w_i \circ \chi + \rho)(h_{\alpha_i})+1)}  \\
&=& \ q^{2(1-g)}  \frac{\zeta_F(-(\chi + \rho)(h_{\alpha_i}))}{\zeta_F(-(\chi + \rho)(h_{\alpha_i})+1)}   \, \frac{\zeta_F((\chi + \rho)(h_{\alpha_i}))}{\zeta_F((\chi + \rho)(h_{\alpha_i})+1)} .
\end{eqnarray*}
In Lemma \ref{lem-product-zeta}, we calculated the value of this product of zeta functions, so we conclude that
$$ c({\chi}, w_i) c(w_i \circ {\chi},w_i) =  q^{2(1-g)} \ q^{-2(1-g)}=1.$$
As a result, we see $c({\chi},w w_i)=c(w_i\circ {\chi}, w) c({\chi},w_i)$, the desired result. We now consider the general case.

\underline{General Case:}  The proof is by induction on the length of $w'$.  If $\ell(w')=1$, then the proofs of Case 1 and Case 2 secure our result. Now suppose the result holds for $\ell(w') \le k$. We consider $w'' \in \hat W$ such that $\ell(w'')=k+1$, and write $w''=w'w_i$ for some simple reflection $w_i$ and with $\ell(w'')=\ell(w')+1$.  By Cases 1 and 2 and the induction hypothesis, we obtain
\begin{eqnarray*} c(\chi,w w'')&=& c(\chi, w w' w_i)=  c(w_i \circ \chi, w w')\, c(\chi, w_i)= c( (w' w_i) \circ \chi, w) \, c(w_i \circ \chi,w') \, c(\chi,w_i) \\ &=&  c( (w'w_i) \circ \chi, w) \,  c(\chi, w'w_i) =  c(w'' \circ \chi,w) \, c(\chi,w''),\end{eqnarray*}
since $ c(w_i \circ \chi,w') \, c(\chi,w_i)  = c(\chi, w'w_i)$.
This completes our proof of Proposition \ref{prop-cocycle}.
\end{proof}

\vskip 1cm

\appendix

\section{Measures}

In this appendix, we will describe various measures that are important for our calculation of the constant term of an Eisenstein series. The first subsection addresses the technique of constructing a measure by means of a projective limit of a family of measures.

\subsection{The projective limit construction of a measure}

We begin by stating a result of \cite{Ash}.  Let $\{ U^{(s)}, \pi^{s}_{s'} \}_{s\in\mathbb Z_{>0}}$ be a projective family of compact spaces, and equip each $U^{(s)}$ with a regular, Borel, probability measure $du^{(s)}$. If $s>s'$, then by our assumption we have the map $$\pi^s_{s'}: U^{(s)} \twoheadrightarrow U^{(s')}.$$
We say $\{du^{(s)}\}_{s\in\mathbb Z_{>0}}$ is a consistent family of measures with respect to the projections $\pi^s_{s'}$ if for any measurable set $X'\subset U^{(s')}$ we have
$$ du^{(s)}((\pi^s_{s'})^{-1}(X'))=du^{(s')}(X' ),$$ for any $s>s'$.

\begin{Thm}[\cite{Ash}] \label{thm-Ash}
Suppose $\{ U^{(s)}, \pi^{s}_{s'} \}_{s\in\mathbb Z_{>0}}$ is a projective family of compact spaces where $du^{(s)}$ is a regular, Borel, probability measure on $U^{(s)}$. If $\{du^{(s)}\}_{s\in\mathbb Z_{>0}}$ is a consistent family of measures with respect to the projections $\pi^s_{s'}$, then
\begin{enumerate}
\item there is a unique regular, Borel, probability measure $du$ on the projective limit $\underset{s}\varprojlim \ U^{(s)}$,
\item if $\pi^{(s')}:\underset{s}\varprojlim U^{(s)} \twoheadrightarrow U^{(s')}$ is the canonical projection, then for any measurable set $X\subset U^{(s')}$ we have $$du((\pi^{(s')})^{-1}(X))=du^{(s')}(X).$$
\end{enumerate}
\end{Thm}
\begin{Cor}
Under the conditions of Theorem \ref{thm-Ash}, if we further assume that each measure $du^{(s)}$ is translation invariant, then $du$ is also translation invariant.
\end{Cor}
\begin{proof}
Our measure $du$ is a Borel measure, so it is enough to show that this property holds for any open set $Y\subset \underset{s}\varprojlim U^{(s)}$.  However, the projective limit of topological spaces inherits the coarsest topology such that the canonical projections $\pi^{(s)}$ are all continuous.  It is a standard result that a basis for this topology consists of the sets $(\pi^{(s)})^{-1}(X^{(s)})$ for an open set $X^{(s)}\subset U^{(s)}$.  Hence, the translation invariance of the measure $du$ is a result of part (2) of Theorem \ref{thm-Ash}, and the invariance of the measure $du^{(s)}$.
\end{proof}

\medskip

\subsection{Measure on the Arithmetic Quotient}

   As mentioned in Section 4, we need to define a measure on the quotient space $\hat{U}_{\mathbb A}/(\hat{U}_{\mathbb A}\cap \hat{\Gamma}_F)$.  The main result of \cite{Ash} allows us to define this measure $du$ by expressing $\hat{U}_{\mathbb A}/(\hat{U}_{\mathbb A}\cap\hat{\Gamma}_F)$ as a projective limit of compact spaces $\hat{U}_{\mathbb A}^{(s)}/\hat{\Gamma}_{\hat U}^{(s)}$ equipped with  a consistent family of measures $du^{(s)}$. For a similar construction over $\mathbb Q$, see \cite{R}.

  Recall that we fixed a coherently ordered basis $\mathcal B=\{v_{\lambda}, v_1,  v_2,  \dots\}$ which we use as a basis for the vector space $V_{F_{\nu}}^{\lambda}$. For $s\in\mathbb Z_{>0}$, set $V^{\lambda}_{F_{\nu},s}$ to be the $F_{\nu}\text{-span of the vectors }\{v_{\lambda},  v_1, \dots,  v_s\}$, and $\hat{U}_{\nu,s}=\{u\in\hat{U}_{F_{\nu}} \ | \ u\big|_{V^{\lambda}_{F_{\nu},s}} \equiv \mathrm{id} \}$.  We define the restricted direct product
$$\hat{U}_{\mathbb A,s}:={\prod}' \ \hat{U}_{\nu,s}\text{ with respect to the subgroups }\hat{U}_{\nu,s}\cap\hat{K}_{\nu},$$ and let $$\hat{U}_{\mathbb A}^{(s)}:=\hat{U}_{\mathbb A}/\hat{U}_{\mathbb A,s}.$$

  We first note that any element $u\in\hat{U}_{\mathbb A}^{(s)}$ is an infinite tuple $(u_{\nu})_{\nu\in\mathcal{V}}$. In \cite{LG1} we see that each $u_{\nu}$ is in the subgroup of $(s+1)\times (s+1)$, strictly upper triangular block matrices with entries from $F_{\nu}$, where the blocks are determined by the weight spaces of $V^{\lambda}_{\mathbb Z}$. Moreover, we know that for all but a finite number of $\nu$, our entries are from $\mathcal{O}_{\nu}$. As such, we may identify $\hat{U}_{\mathbb A}^{(s)}$ with $n$ copies of $\mathbb A$, for some $n$. We give $\hat{U}_{\mathbb A}^{(s)}$ the topology that makes this identification a homeomorphism and denote this homeomorphism by $\phi^{(s)}$.

\begin{Lem} \label{lem-loc-cmpt} We have \begin{enumerate}
\item $\hat{U}_{\mathbb A}^{(s)}$ is a locally compact group for every $s\in\mathbb Z_{>0}$,
\item $\hat{U}_{\mathbb A}\cong\underset{s}\varprojlim \hat{U}_{\mathbb A}^{(s)}$.
\end{enumerate}
\end{Lem}
\begin{proof}
The first part follows from the local compactness of $\mathbb A$ and the homeomorphism $\phi^{(s)}$.
In order to prove the second part, we first check that $\{\hat{U}_{\mathbb A}^{(s)}, \pi^s_{s'}\}$ forms a projective system. Suppose we have $s> s'$. Then $V^{\lambda}_{F_{\nu},s'}\subset V^{\lambda}_{F_{\nu},s}$ for all $\nu\in\mathcal{V}$, as a result we have that $\hat{U}_{\mathbb A,s}\subset\hat{U}_{\mathbb A,s'}$ and  obtain a unique surjective map $\pi^s_{s'}:\hat{U}_{\mathbb A}^{(s)}\twoheadrightarrow\hat{U}_{\mathbb A}^{(s')}$ such that
$\pi^s_{s'} \circ p^{(s)}= p^{(s')}$, where $p^{(s)}: \hat{U}_{\mathbb A} \rightarrow \hat{U}_{\mathbb A}/\hat{U}_{\mathbb A,s}=\hat{U}_{\mathbb A}^{(s)}$ is the canonical projection. Now it is clear that $\{\hat{U}_{\mathbb A}^{(s)}, \pi^s_{s'}\}$ is a projective system and we have a surjective map
$$\Phi: \hat{U}_{\mathbb A} \twoheadrightarrow \underset{s}\varprojlim\hat{U}_{\mathbb A}^{(s)}\, ,$$ such that $\pi^{(s)}\circ\Phi=p^{(s)}$ for every $s\in\mathbb Z_{>0}$, where $\pi^{(s)}$ denote the canonical projections from $\underset{s}\varprojlim \hat{U}_{\mathbb A}^{(s)}$ to $\hat{U}_{\mathbb A}^{(s)}$.   In order to prove the isomorphism, we only need to show that $\Phi$ is an injective map.

  Suppose that $x\in\ker(\Phi)$. Then $x\in\ker(\pi^{(s)}\circ\Phi)$ for every $s\in\mathbb Z_{>0}$.  By the commutativity of our maps, we have that $x\in\ker(p^{(s)})$ for every $s\in\mathbb Z_{>0}$, and so $x\in\displaystyle\bigcap_{s\in\mathbb Z_{>0}}\hat{U}_{\mathbb A,s}=\{\mathrm{id}\}$. Thus, $\Phi$ is injective as well as surjective and we obtain the desired isomorphism between $\hat{U}_{\mathbb A}$ and $\underset{s}\varprojlim \hat{U}_{\mathbb A}^{(s)}$.
\end{proof}

Under the homeomorphism $\phi^{(s)}$, the product measure $\bar{\mu}^{(s)}= \mu \times\dots \times\mu$ on $\mathbb A^{n}$ becomes a Haar measure on $\hat{U}_{\mathbb A}^{(s)}$.  To see this, note that the effect of left multiplication in $\hat{U}_{\mathbb A}^{(s)}$ through this homeomorphism is affine transformation of $\mathbb A^{n}$.  In other words, for $x$ and $y$ in $\mathbb A^{n}$ we have $$xy = A_x(y) + b_x,$$ where $A_x$ is a linear transformation of $\mathbb A^{n}$ and $b_x\in\mathbb A^{n}$.  The invariance of $\bar{\mu}^{(s)}$ is hence a direct result of the change of variables theorem and the fact that $|\det(A_x)|=1$ for all $x$.

  We wish to define a measure on $\hat{U}_{\mathbb A}/(\hat{\Gamma}_F\cap\hat{U}_{\mathbb A})$. The next lemma establishes this space as a projective limit of compact spaces.

\begin{Lem}
Let $\hat{\Gamma}_{\hat U}^{(s)}=p^{(s)}(\hat{\Gamma}_F\cap\hat{U}_{\mathbb A})$, where  $p^{(s)}:\hat{U}_{\mathbb A} \twoheadrightarrow \hat{U}_{\mathbb A}/\hat{U}_{\mathbb A,s}=\hat{U}_{\mathbb A}^{(s)}$ is the canonical projection.  Then
\begin{enumerate}
\item we have $\hat{U}_{\mathbb A}/(\hat{\Gamma}_F\cap\hat{U}_{\mathbb A})\cong\underset{s}\varprojlim \ \hat{U}_{\mathbb A}^{(s)}/\hat{\Gamma}_{\hat U}^{(s)}, \text{ and}$
\item for each $s\in\mathbb Z_{>0}$ the space $ \hat{U}_{\mathbb A}^{(s)}/\hat{\Gamma}_{\hat U}^{(s)}$ is compact and has an invariant probability measure $du^{(s)}$ induced from the measure $\bar{\mu}^{(s)}$ on $\hat{U}_{\mathbb A}^{(s)}$.
\end{enumerate}
\end{Lem}
\begin{proof}
(1)  Since $\pi^s_{s'}(\hat{\Gamma}_{\hat U}^{(s)})=\hat{\Gamma}_{\hat{U}}^{(s')}$, we obtain the induced projection
 $\bar{\pi}^s_{s'} : \hat{U}_{\mathbb A}^{(s)}/\hat{\Gamma}_{\hat{U}}^{(s)} \rightarrow \hat{U}_{\mathbb A}^{(s')}/\hat{\Gamma}_{\hat{U}}^{(s')}$, and the collection $\left\{\hat{U}_{\mathbb A}^{(s)}/\hat{\Gamma}_{\hat U}^{(s)}, \bar{\pi}^s_{s'}\right\}_{s\in\mathbb Z_{>0}}$ is a projective family. Then we have the natural map
 \[ \hat{U}_{\mathbb A} \ \cong \  \underset{s}\varprojlim\ \hat{U}_{\mathbb A}^{(s)} \ \rightarrow \ \underset{s}\varprojlim \ \hat{U}_{\mathbb A}^{(s)}/\hat{\Gamma}_{\hat U}^{(s)} .\]
Considering the kernel of the map, we obtain $$\hat{U}_{\mathbb A}/(\hat{\Gamma}_F\cap\hat{U}_{\mathbb A})\cong\underset{s}\varprojlim \ \hat{U}_{\mathbb A}^{(s)}/\hat{\Gamma}_{\hat U}^{(s)}.$$

(2)  Recall that we can consider the space $\hat{U}_{\mathbb A}^{(s)}$ to be embedded into the group of strictly upper triangular $(s+1)\times (s+1)$ block matrices with entries from $\mathbb A$; similarly, $\hat{\Gamma}_{\hat U}^{(s)}$ can be considered as a discrete subspace of $\hat{U}_{\mathbb A}^{(s)}$, consisting of strictly upper triangular $(s+1)\times (s+1)$ block matrices with entries from $F$ diagonally embedded in $\mathbb A$. As a result, the quotient space $\hat{U}_{\mathbb A}^{(s)}/\hat{\Gamma}_{\hat U}^{(s)}$ is a classical object and it is well known that this is a compact space. For more information see \cite{Plat}.

  Moreover, since $\hat{\Gamma}_{\hat U}^{(s)}$ is a discrete subgroup of the unimodular group $\hat{U}_{\mathbb A}^{(s)}$, we have an invariant measure $du^{(s)}$ on $\hat{U}_{\mathbb A}^{(s)}/\hat{\Gamma}_{\hat U}^{(s)}$ induced from $\bar{\mu}^{(s)}$ on $\hat{U}_{\mathbb A}^{(s)}$ (\cite{Plat}), and we normalize $du^{(s)}$ to have total measure 1.
\end{proof}

 Finally, we must prove that the measures $du^{(s)}$ form a consistent family.  The induced measure $du^{(s)}$ may be considered as the restriction of the measure $\bar{\mu}^{(s)}$ to a fundamental domain for $\hat U_{\mathbb A}^{(s)} /\hat{\Gamma}_{\hat U}^{(s)}$. For each $s \in \mathbb Z_{>0}$, we choose a fundamental domain $\Omega^{(s)}$ for $\hat U_{\mathbb A}^{(s)} /\hat{\Gamma}_{\hat U}^{(s)}$ such that $\pi^s_{s'} (\Omega^{(s)}) =  \Omega^{(s')}$. Then we have,  for a measurable subset $X$ of $\hat U_{\mathbb A}^{(s')}$,
 \begin{equation} \label{eqn-mm} \bar{\mu}^{(s)} \left ( (\pi^s_{s'})^{-1}(X) \cap \Omega^{(s)} \right ) = \bar{\mu}^{(s')}( X \cap \Omega^{(s')}) .\end{equation}

\begin{Lem} \label{lem-consist}
The set $\{du^{(s)} \}_{s\in\mathbb Z_{>0}}$ forms a consistent family of measures.
\end{Lem}
\begin{proof}  Let $\bar{\pi}^{(s)}:\hat{U}_{\mathbb A}^{(s)}\rightarrow\hat{U}_{\mathbb A}^{(s)}/\hat{\Gamma}_{\hat U}^{(s)}$ be the canonical projection for $s\in\mathbb Z_{>0}$. Suppose that $X$ is a measurable set of $\hat{U}_{\mathbb A}^{(s')}/\hat{\Gamma}_{\hat{U}}^{(s')}$. Then by our observation above
$$du^{(s')}(X)=\bar{\mu}^{(s')}\bigg((\bar{\pi}^{(s')})^{-1}(X) \cap {\Omega}^{(s')} \bigg).$$
To establish the consistency of the measures, we need to calculate $du^{(s)}( (\bar{\pi}^{s}_{s'})^{-1}(X) )$. Using (\ref{eqn-mm}), we obtain
\begin{eqnarray*}
du^{(s)}( (\bar{\pi}^{s}_{s'})^{-1}(X) ) &=& \bar{\mu}^{(s)}\bigg( (\bar{\pi}^{s}_{s'} \circ \bar{\pi}^{(s)} )^{-1}(X) \cap {\Omega}^{(s)} \bigg) = \bar{\mu}^{(s)}\bigg( ( \bar{\pi}^{(s')} \circ \pi^{s}_{s'})^{-1}(X) \cap {\Omega}^{(s)} \bigg)
\\ &=& \bar{\mu}^{(s')}\bigg((\bar{\pi}^{(s')})^{-1}(X) \cap {\Omega}^{(s')} \bigg)= \ du^{(s')}(X),
\end{eqnarray*}
and we  conclude these measures form a consistent family.
\end{proof}

Finally, we obtain the main result of this appendix in the following proposition.

\begin{Prop} There exists a unique, $\hat{U}_{\mathbb A}$-left invariant, probability measure $du$ on the arithmetic quotient $\hat{U}_{\mathbb A}/(\hat{U}_{\mathbb A}\cap\hat{\Gamma}_F)$.
\end{Prop}
\begin{proof} This follows from Lemma \ref{lem-consist} and Theorem \ref{thm-Ash}.
\end{proof}

\medskip

\subsection{Measures on other spaces}

  Before we discuss measures on other spaces, we briefly recall some constructions from Section 4. With respect to our coherently ordered basis $\mathcal B$, we fix $\hat{U}_{-,F}$ to be the group of strictly lower triangular block matrices and set $\hat{U}_{w, F}$ to be $\hat{U}_F\cap w\hat{U}_{-,F} w^{-1}$.  This definition works for all of our fields $F_{\nu}$, and so we can define $$\hat{U}_{w,\mathbb A}:={\prod_{\nu\in\mathcal{V}}}' \ \hat{U}_{w,\nu} \ \text{  with respect to }\hat{U}_{w,\nu}\cap \hat{K}_{\nu}.$$
By the decomposition (\ref{eqn-decom}), we know
\begin{equation} \label{eqn-o-m} \hat{U}_{\mathbb A} \ = \ \hat{U}_{w,\mathbb A} \, (\hat{U}_{\mathbb A}\cap w\hat{U}_{\mathbb A} w^{-1}),\end{equation} and  we have the covering projection  $$\pi':\hat{U}_{\mathbb A}/(\hat{\Gamma}_F\cap\hat{U}_F\cap w\hat{U}_F w^{-1}) \twoheadrightarrow \hat{U}_{\mathbb A}/(\hat{U}_{\mathbb A}\cap\hat{\Gamma}_F).$$
We define the measure $du'$ on  $\hat{U}_{\mathbb A}/(\hat{\Gamma}_F\cap\hat{U}_F\cap w\hat{U}_F w^{-1})$ to be the one induced from the measure $du$ on $\hat{U}_{\mathbb A}/(\hat{U}_{\mathbb A}\cap\hat{\Gamma}_F)$ through the projection $\pi'$.

  In light of the decompositions (\ref{eqn-o-m}), our measure $du'$ decomposes into two measures:
\begin{enumerate}
\item[($i$)] the measure $du_1$ on $\hat{U}_{w,\mathbb A}$, and
\item[($ii$)] the measure $du_2$ on $(\hat{U}_{\mathbb A}\cap w\hat{U}_{\mathbb A} w^{-1})\big/(\hat{\Gamma}_F\cap\hat{U}_F\cap w\hat{U}_F w^{-1})$.
\end{enumerate}

We construct the measure $du_2$ by expressing this space as a projective limit of compact spaces equipped with a Haar probability measure. We omit the details and note that the construction will be very similar to our proofs regarding the construction of the measure $du$ on $ \hat{U}_{\mathbb A}/(\hat{U}_{\mathbb A}\cap\hat{\Gamma}_F)$. Using Theorem \ref{thm-Ash}, the measure $du_2$ is a Haar probability measure.

  To construct $du_1$, we recall that every element of $u\in\hat{U}_{w,\mathbb A}$ takes the form $$u=\prod_{a\in\hat{\Delta}_{W,+}\cap w^{-1} \hat{\Delta}_{W,-}} \chi_a(s_a)\ ,\text{ for }s_a\in\mathbb A.$$  For more information, see Section 4.3.  The set $\hat{\Delta}_{W,+}\cap w^{-1} \hat{\Delta}_{W,-}$ is finite and of size $\ell(w)$. As in Lemma \ref{lem-loc-cmpt}, we have a homeomorphism from $\hat{U}_{w,\mathbb A}$ to $\mathbb A^{\ell(w)}$. The product measure induced from $\ell(w)$ copies of the Haar measure $\mu$ on $\mathbb A$ becomes a Haar measure on $\hat{U}_{w,\mathbb A}$, where we normalize $\mu$ so that $\mu(\mathbb A/F)=1$. We set $du_1$ to be this measure.

  If we consider the Haar measure $\mu'_\nu$ on $F_\nu$ with the normalization $\mu'_\nu(\mathcal O_\nu)=1$  for each $\nu \in \mathcal V$, we obtain from 2.1.3 of \cite{Weil}  that
  \begin{equation} \label{eqn-me-la} \mu = q^{(1-g)} \prod_\nu \mu'_\nu . \end{equation}
If we let $du'_{1,\nu}$ be the measure on $\hat{U}_{w,F_\nu}$, then we have
  \[
   du_1 = q^{\ell(w)(1-g)} \prod_\nu du'_{1,\nu} .
  \]
A similar relation holds between $du_-$ for $\hat{U}_{-,w, \mathbb A}$ and the product of measures on local components in Section 4.4.

\end{document}